\documentclass{article}
\usepackage{graphicx} 
\usepackage{amsmath,amsthm}
\usepackage{setspace}
 \usepackage[utf8]{inputenc} 
\usepackage[T1]{fontenc}    
\usepackage{xcolor}
\usepackage[
    colorlinks=true,
    citecolor=blue,
    linkcolor=green!60!black,
    urlcolor=blue
]{hyperref}       
\usepackage{url}            
\usepackage{booktabs}       
\usepackage{amsfonts}       
\usepackage{amsmath}       
\usepackage{amssymb}
\usepackage{nicefrac}       
\usepackage{cancel}
\usepackage{lipsum}
\usepackage{graphicx}
\usepackage{float}
\usepackage{caption} 
\usepackage{subcaption}

\usepackage[linesnumbered,ruled,vlined]{algorithm2e}
\SetKwInput{KwInput}{Input}                
\SetKwInput{KwOutput}{Output}  

\SetCommentSty{mycommfont}

\newcommand{\matr}[1]{\begin{bmatrix} #1 \end{bmatrix}}    
\usepackage{bbm}
\def\transp{^{\text{\sf T}}}
\def\1{\mathbbm{1}}

\newcommand{\Diag}{\text{Diag}}
\newcommand{\R}{\mathbb{R}}

\usepackage{thmtools}

\newtheorem{example}{Example}

\newtheorem{corollary}{Corollary}
\newtheorem{lemma}{Lemma}
\newtheorem{remark}{Remark}

\usepackage[title]{appendix}

\usepackage{titlesec}

\titleformat{\paragraph}
  {\normalfont\normalsize\bfseries}{\theparagraph}{1em}{}
\titlespacing*{\paragraph}{0pt}{3pt}{3pt}

\usepackage{geometry}
\usepackage[square,numbers,sort&compress]{natbib}

\title{Fragility of Minimum-Variance Portfolios}
\author{Daniel Ovalle$^1$, Carl D. Laird$^1$, Ignacio E. Grossmann$^1$ \& Javier Pe\~na$^2$\thanks{Corresponding author. Email: \texttt{jfp@andrew.cmu.edu}}\\ \\
    \small $^1$Department of Chemical Engineering, Carnegie Mellon University, Pittsburgh, PA.\\
    \small $^2$Tepper School of Business, Carnegie Mellon University, Pittsburgh, PA.\\}
\date{\today}

\begin{document}

\maketitle
\begin{abstract}
    
Minimum-variance portfolios are well known to be highly sensitive to covariance estimation error.  In this paper, we show that by imposing a block diagonal correlation structure, we can derive closed-form expressions for long-only minimum-variance portfolios that make this fragility explicit. 
These analytical solutions reveal that fragility is driven by threshold effects arising from the interaction between correlation structure and the assets' volatilities.  Motivated by the latter insight, we propose robustification approaches that can be interpreted as structured shrinkage schemes that selectively attenuate unstable coupling while preserving the dominant risk structure. 
Unlike global shrinkage techniques, the proposed corrections are analytically grounded, require minimal tuning, and remain closely aligned with the minimum-variance solution. The framework extends naturally to general covariance matrices through clustering-based approximations.
Empirical results on controlled simulations highlight a clear regime dependence. In homogeneous volatility settings, correlation-oblivious shrinkage achieves the best trade-off between risk and stability. In contrast, under heterogeneous volatility, correlation-aware shrinkage performs best by inducing sparsity and avoiding exposure to high-risk assets. Across regimes, the proposed methods consistently reduce out-of-sample variance and turnover relative to classical and clustering-based benchmarks, providing a principled and practical approach to robust portfolio construction.

\end{abstract}

\section{Introduction}

The portfolio selection problem introduced by  \citet{markowitz1952} provides a canonical framework for allocating capital across a universe of risky assets. In this work, we focus exclusively on the minimum-variance formulation, which seeks to construct a portfolio with the smallest possible return variance subject to full-investment and long-only constraints. Let $n \in \mathbb{N}$ denote the number of assets under consideration, and let $x \in \mathbb{R}^n$ denote the vector of portfolio weights, where each component $x_i$ represents the fraction of total capital invested in asset $i$. The long-only and budget constraints require that $x \ge 0$ (componentwise) and $\1\transp x = 1$, where $\1 \in \mathbb{R}^n$ denotes the vector of all ones.

Let $r\in\mathbb{R}^n$ denote the vector of asset returns and let $V \in \mathbb{R}^{n \times n}$ denote the covariance matrix of asset returns. The  variance of the portfolio with weights $x$ is $x\transp V x$. Thus the fully-invested long-only minimum-variance portfolio is the solution to the quadratic program
\begin{equation}\label{eq.minvar}
\begin{aligned}
\min_{x \in \mathbb{R}^n} \quad & x\transp V x, \\
\text{s.t.} \quad & \1\transp x = 1, \\
& x \ge 0.
\end{aligned}
\end{equation}
This problem is convex and thus admits a global minimum since the covariance matrix 
$V$ is symmetric positive semidefinite.  

A well-documented limitation of the Markowitz framework is the pronounced sensitivity of optimized portfolios to errors in the estimation of the covariance matrix \cite{michaud1989markowitz}. 
Early empirical and analytical studies noted that 
small perturbations in the inputs of mean-variance portfolio optimization can lead to disproportionately large changes in the optimal allocation and severe degradation of out-of-sample risk performance 
\cite{best1991sensitivity, chopra1993effect, muller1993empirical, michaud2001efficient, brandt2010portfolio, chen2016efficient}. 
This fragility persists even in the minimum-variance setting, where expected returns are excluded, and thus cannot be attributed solely to errors in mean estimation.
In practice, minimum-variance portfolios constructed from empirical covariances are often highly concentrated, unstable through time, and frequently dominated by simple diversified allocations in realized risk \cite{clarke2006minimum, demiguel2009optimal, bailey2012balanced, lopez2016building, roncalli2013introduction}. From an economic perspective, this behavior contradicts the primary motivation for minimum-variance investing, e.g., to obtain a stable and implementable allocation that relies on comparatively well-estimated inputs. 
This raises a central question to this paper: where exactly in the covariance matrix does this instability originate, and how is it transmitted into portfolio weights?

The origin of this instability is already visible in the analytical form of the unconstrained minimum-variance solution, which is proportional to  $V^{-1} \1$ \cite{markowitz1952, clarke2011minimum}.
Because the optimizer depends explicitly on the inverse covariance matrix, perturbations in $V$ are mapped nonlinearly into the portfolio through an operator that amplifies estimation error when the covariance matrix is ill-conditioned or nearly singular.
Such ill-conditioning is not an artifact of the modeling choice but an inherent feature of empirical asset returns. Financial datasets are characterized by strong cross-asset correlations, near-collinearity, low effective rank in the covariance structure, and short history that limits sample sizes. These structural properties yield estimates that are intrinsically noisy and highly unstable \cite{pedersen2021enhanced}.
High-dimensional covariance estimation provides complementary evidence for this mechanism: when the cross-sectional dimension is not negligible relative to the sample length, sample eigenvalues are systematically distorted and spurious small eigenvalues emerge that dominate the inverse, while the associated eigenvectors are themselves noisy and unstable \cite{karoui2008spectrum, bun2017cleaning, hall2005geometric, shen2016statistics, aoshima2018survey}.
Since the minimum-variance portfolio depends on the full spectral decomposition of $V^{-1}$ both sources of noise propagate through the optimization and produce extreme and rapidly varying allocations \cite{karoui2010high, karoui2013realized}. 
Despite substantial progress in documenting and partially characterizing these effects across theoretical, empirical, and high-dimensional asymptotic settings, a complete and unified explanation of the instability of minimum-variance portfolios remains elusive, and identifying the precise structural mechanisms through which the covariance matrix generates this sensitivity continues to be an active area of research \cite{bandi2012tractable,goldberg2022dispersion, gurdogan2024quadratic, shkolnik2025portfolio}. 
In particular, the existing literature explains why the optimizer is sensitive, but does not isolate which economically meaningful features of the covariance structure govern how this instability propagates into the portfolio weights. Nor does it provide actionable guidance for mitigating the effect in practice. This gap motivates the analysis that follows.

Empirical evidence shows that financial assets are far from uniformly correlated. Markets are naturally organized into nested groups, such as sectors, industries, regions, and investment styles, where assets within a group tend to move together more strongly than with assets outside the group \cite{tumminello2005tool, tumminello2011community, musmeci2015relation}. This hierarchical or block-like organization implies that the covariance matrix is structured: certain directions of risk are concentrated within clusters, while cross-cluster dependencies are weaker. Ignoring this structure, treating all assets symmetrically, exacerbates the instability problem highlighted in the minimum-variance optimization \cite{michaud1989markowitz, demiguel2009optimal, lopez2016building, bailey2012balanced}. 
Recognizing clustering as a stylized feature of financial markets highlights the challenge of understanding which parts of the covariance matrix drive instability and suggests that this structure could be key to constructing robust, interpretable portfolios.

To address the instability of minimum-variance portfolios, the literature has developed several classes of remedies. Robust optimization provides formal protection against perturbations by optimizing portfolio weights over uncertainty sets for the covariance matrix and expected returns \cite{bental2002robust,bertsimas2011theory, popescu2007robust,goldfarb2003robust,
glasserman2013robust,
tutuncu2004robust,
ghaoui2003worst, 
boyd2024markowitz,cornuejols2026addressing}.
While this approach provides worst-case performance guarantees, it does not identify which directions of the covariance structure drive the instability. Moreover, its practical implementation requires specifying the geometry, scale, and error structure of the uncertainty set for the covariance matrix or expected returns; design choices that are difficult to infer from finite samples, often subjective, and in many cases as challenging as estimating the covariance itself \cite{ fabozzi2010robust, bertsimas2018data, stubbs2005computing}.
To circumvent the subjectivity of static uncertainty sets, operations research literature increasingly employs distributionally robust optimization to immunize portfolios against empirical sample biases via moment bounds or optimal transport \cite{delage2010distributionally, natarajan2008incorporating, blanchet2019quantifying, gao2023distributionally}.
Shrinkage methods and norm-constrained machine learning approaches stabilize the estimation problem by blending the sample covariance with structured targets or penalizing dense allocations, improving conditioning and damping sampling noise \cite{ledoit2004honey, connor1993test, james1961estimation, ledoit2004well, chen2010shrinkage, longerstaey1996riskmetricstm, ledoit2012nonlinear, ledoit2017nonlinear, ledoit2020analytical, demiguel2009generalized, ban2018machine, bertsimas2022scalable}. Similar target-implied stabilization can be extracted via inverse optimization frameworks \cite{bertsimas2012inverse}.
These methods reduce extreme weight fluctuations and turnover, but they operate globally and are target-driven, without directly addressing the fragile components responsible for instability. Finally, clustering and hierarchical allocation approaches exploit the empirical organization of financial assets into sectors, industries, and other nested groups \cite{tumminello2005tool, tumminello2011community, musmeci2015relation, marti2021review, jain2019can, lau2017cross}. 
By imposing multiscale or block-hierarchical structures on the correlation matrix, these methods produce empirically robust and more diversified portfolios, yet they offer no closed-form solution, yield fully dense portfolios, and do not provide an explicit link between the identified clusters and the underlying instability \cite{lopez2016building, raffinot2018hierarchical, raffinot2018hcaa}.
Across shrinkage and hierarchical approaches, a common pattern emerges: instability is mitigated indirectly through smoothing, aggregation, or structural regularization. This often improves robustness in practice, but it does so without explicitly distinguishing between stabilizing the estimator and preserving the directions that matter for optimal allocation. As a result, reductions in fragility can come at the cost of latent suboptimality in risk, in the sense that portfolios become more stable but not necessarily closer to the true minimum-variance solution.

In this paper, we derive closed-form expressions for minimum-risk portfolios under a block-diagonal correlation structure, which serves as a tractable approximation to clustered dependence patterns in financial markets. This analytical solution makes explicit how the inverse covariance operator amplifies estimation noise in portfolio weights, and thereby identifies the components of the covariance structure responsible for instability. 
This characterization directly motivates a simple corrective mechanism: selectively attenuating the unstable components revealed by the solution, rather than applying global or target-based regularization. The resulting procedure can be interpreted as a structured form of covariance shrinkage, but one that is mechanism-driven and explicitly linked to the portfolio mapping induced by the block structure.
Relative to hierarchical and clustering-based approaches, the method is fully analytical and avoids recursive constructions while retaining interpretability in terms of latent market structure \cite{lopez2016building, raffinot2018hierarchical, raffinot2018hcaa}. It yields portfolios that remain close to the minimum-variance solution in-sample while substantially improving stability out-of-sample, reconciling the optimality of the Markowitz formulation with the robustness typically associated with structure-imposing methods.

 The remainder of the paper is organized as follows. Section \ref{sec:rob_reparam} introduces the notation used throughout the paper, including a reparameterization of the original problem. Section \ref{sec:constant_correlation} studies instability under a constant correlation structure, providing closed-form optimality conditions and illustrative examples. Section \ref{sec:structured_correlation} extends the analysis to a block-constant correlation structure, again deriving closed-form optimality results and associated examples. Section \ref{sec:rob_robustification} presents the proposed robustification approaches, built upon the closed-form insights developed in the previous sections. Section \ref{sec:results} reports numerical experiments that benchmark the methods in terms of robustness and optimality. Finally, Section \ref{sec:conclusions} concludes and outlines directions for future research.

\section{Preliminaries and Reparametrization} \label{sec:rob_reparam}

Following the notation in \citet{pedersen2021enhanced}, we decompose the covariance matrix as
\[
V = \Sigma \, \Omega \, \Sigma,
\]
where $\Sigma = \mathrm{Diag}(\sigma_1, \ldots, \sigma_n) \in \mathbb{R}^{n \times n}$ is the diagonal matrix of asset volatilities, with $\sigma_i > 0$ denoting the standard deviation of asset $i$, and $\Omega \in \mathbb{R}^{n \times n}$ is the correlation matrix.  Thus $\Omega_{ii} = 1$ for all $i$ and $\Omega$ is symmetric positive semidefinite.  
This factorization separates specific risk, encoded by $\Sigma$, from the correlation dependence structure, encoded by $\Omega$.  

For notational convenience, define the inverse-volatility vector $\theta \in \mathbb{R}^n$ as
\begin{equation}\label{eq.theta}
\theta_i = \sigma_i^{-1}, \quad i=1,\ldots,n,
\end{equation}
and let $\Theta = \mathrm{Diag}(\theta) \in \mathbb{R}^{n \times n}$. We use $(\cdot)^+$ and $(\cdot)^-$ for componentwise positive and negative parts.

Note that a simple change of variables in problem \eqref{eq.minvar} reveals that its solution can be expressed in terms of a correlation-based problem as formalized in the following lemma and proved in Appendix \ref{app:lemma.correlation.reformulation}.

\begin{restatable}{lemma}{lemmacorrelationreformulation}\label{lemma:correlation.reformulation}
Let $V = \Sigma \Omega \Sigma$, with $\theta = \sigma^{-1}$ and $\Theta = \mathrm{Diag}(\theta)$. Then the solution of \eqref{eq.minvar} can be written as
\begin{equation}\label{eq.sol.x}
x = \frac{\Theta y}{\1\transp \Theta y},
\end{equation}
where $y$ solves
\begin{equation}\label{eq.min.correl}
\begin{aligned}
\min_{y \in \mathbb{R}^n} \quad & \frac{1}{2} y\transp \Omega y - \theta\transp y \\
\text{s.t.} \quad & y \ge 0.
\end{aligned}
\end{equation}
\end{restatable}

This reparametrization shows that the long-only minimum-variance portfolio is fully determined by the correlation structure and inverse volatilities, while separating the budget constraint from the dependence structure. In particular, the support of $x$ coincides with the support of $y$ ($x_i=0 \iff y_i=0$), so sparsity and instability properties can be studied directly in the transformed problem.

The problem \eqref{eq.min.correl} is a convex quadratic program with nonnegativity constraints, and its optimality conditions are
\begin{subequations}\label{eq.kkt}
\begin{align}
\Omega y &= \theta + \mu, \label{eq.kkt_a} \\
y &\ge 0, \quad \mu \ge 0, \quad \mu\transp y = 0, \label{eq.kkt_b}
\end{align}
\end{subequations}
where $\mu$ are Lagrange multipliers associated with the inequality constraints.

This formulation will serve as the basis for deriving closed-form solutions under structured correlation matrices and for characterizing stability properties of the long-only minimum-variance portfolio.

\section{Instability Under a Constant-Correlation Structure}\label{sec:constant_correlation}

In this section, we focus on constant-correlation matrices as a building block to understand the behavior of long-only minimum-variance portfolios. 
Constant-correlation matrices provide a simple and analytically tractable setting, particularly relevant for assets within a single market group (such as a sector, industry, region, or investment style) where correlations tend to be relatively high and uniform. 
Furthermore, the constant-correlation setting admits closed-form solutions for the transformed problem  \eqref{eq.min.correl}, making it possible to explicitly characterize regimes in which the portfolio is sensitive to small changes in parameters or estimation noise.

Formally, we consider the case  where the correlation matrix $\Omega$ takes the form   $\Omega = (1-\rho)I + \rho \1\1\transp$ for some $\rho > {-\frac{1}{n-1}}$ , or equivalently, when all the non-diagonal entries of $\Omega$ are all equal to $\rho$:
\begin{equation}\label{eq.single.block}
\Omega = \matr{1 & \rho & \cdots & \rho \\ \rho & 1 & \cdots & \rho \\ \vdots & \vdots & \ddots & \vdots \\ \rho & \rho & \cdots  & 1}. 
\end{equation}
For notational convenience, we index assets so that volatilities are ordered non-decreasingly,
\begin{equation}\label{eq.order}
\sigma_1 \le \cdots \le \sigma_n \iff \theta_n \le \cdots \le \theta_1.
\end{equation}

This constant-correlation structure can be interpreted as a single ``block'' within a broader block-diagonal model, where assets in the same group exhibit homogeneous dependence. Analyzing this setting in isolation provides a tractable foundation for understanding how correlation-driven coupling affects portfolio allocation and fragility before introducing interactions across multiple blocks.

\subsection{Closed-Form Solution}

The constant-correlation structure allows a closed-form solution to the transformed problem \eqref{eq.min.correl}, as stated in the proposition below {and proved in Appendix~\ref{app:proof.prop.one.block}}.

\begin{restatable}{proposition}{propositiononeblock}\label{prop.one.block} Suppose the correlation matrix $\Omega$ is of the form~\eqref{eq.single.block} for some ${-\frac{1}{n-1}}< \rho <1$.  Then the solution to~\eqref{eq.min.correl} is
\begin{equation}\label{eq.sol.y}
y = \frac{1}{1-\rho}\left(\theta - \bar\theta \1 \right)^+,
\end{equation}
where the threshold $\bar\theta$ is the solution to the following fixed-point problem
\begin{equation}\label{eq.fixed.point}
\bar\theta = 
\frac{\rho}{1-\rho} \, \1\transp \left(\theta - \bar\theta \1 \right)^+. 
\end{equation}
Furthermore, if~\eqref{eq.order} holds then
the solution to the fixed-point problem~\eqref{eq.fixed.point} is
\begin{equation}\label{eq.fixed.point.sol}
\bar \theta = \frac{\rho}{1+(i-1)\rho}\sum_{j=1}^i\theta_j,
\end{equation}
where $i$ is the largest index such that 
\begin{equation}\label{eq.largest.i}
\rho \le \frac{\theta_i}{\theta_i + \sum_{j=1}^{i-1} (\theta_j - \theta_i)}.
\end{equation}
\end{restatable}

The expressions~\eqref{eq.sol.y} and~\eqref{eq.fixed.point.sol} characterize the regimes in which the solution~\eqref{eq.sol.x} to the long-only minimum-variance problem~\eqref{eq.minvar} becomes fragile. 
Let $i$ be the largest index such that $\bar \theta \le \theta_i$ or the equivalent inequality~\eqref{eq.largest.i} holds. 
Fragility arises when $i > 1$ and $\bar \theta$ is close to $\theta_{i}$ or when $i<n$ and $\bar \theta$ is close to $\theta_{i+1}$. In either of these cases a small perturbation in 
$\rho$, $\sigma_i$, or $\sigma_{i+1}$,  can shift the active set in~\eqref{eq.sol.y}. 
Such a shift changes which components are truncated by the positive-part operator, leading to a discontinuous change in the portfolio weights~\eqref{eq.sol.x}. 
The following examples make this observation explicit and connect the analytical characterization with the behavior of estimated portfolios.

\subsection{Illustrative Examples} \label{sec:single_examples}

The closed-form expressions derived above make it possible to explicitly identify parameter regimes in which the long-only minimum-variance portfolio is fragile or robust. 
Throughout the manuscript, we will provide simple examples that illustrate these regimes and connect the analytical thresholds with observable portfolio behavior.
Each analytical example is accompanied by a corresponding numerical experiment, showing how proximity to the fixed-point threshold translates into abrupt changes in portfolio weights and loss of out-of-sample optimality.

To illustrate the fragility mechanism, we conduct controlled simulations in which the true covariance matrix $V$ is known and returns are generated from a multivariate normal distribution with covariance $V$. All experiments follow the same evaluation protocol: portfolios are estimated using a rolling window of $w = 63$ trading days (approximately one quarter) and rebalanced daily over $T = 250$ out-of-sample periods.
At each rebalancing date $t$, we compute the realized portfolio variance $x_t^\top V x_t$ and benchmark it against the minimum attainable variance under the true covariance $V$:
\[
\sigma^2(V) := 
\min\{x^\top V x: x \in \mathbb{R}^n_+, \1^\top x = 1\}.
\]
The cumulative deviation from this benchmark is summarized by the \emph{area} metric,
\[
\text{Area} := \sum_{t=1}^{T} \left(x_t^\top V x_t - \sigma^2(V)\right),
\]
which corresponds to a discrete-time integral of excess variance over the evaluation horizon.
We also report the average per-period \emph{turnover},
\[
\text{Turnover} := \frac{1}{T-1} \sum_{t=2}^{T} \|x_t - x_{t-1}\|_1,
\]
and the average \emph{active-set size}, defined as the number of nonzero positions in $x_t$.

\begin{example}\label{ex.two} Suppose $n=2$ and $0< \sigma_1\le \sigma_2$.  
This two-asset example captures the simplest scenario of a 
correlated pair of assets, showing that portfolio fragility can still arise even in this minimal setting.  We distinguish two regimes.

\begin{itemize}
    \item \noindent{\bf Case 1:} $\rho \le \sigma_1/\sigma_2 = \theta_2/\theta_1$. Proposition~\ref{prop.one.block} implies that
$
\bar\theta=\frac{\rho}{1+\rho} \left(
\theta_1+\theta_2\right)$ and  thus 
\[
(1-\rho)y = \frac{1}{1+\rho}\matr{\theta_1 - \rho\theta_2 \\ \theta_2 -  \rho\theta_1} \; \text{ and } x = \frac{1}{\frac{\sigma_1}{\sigma_2} + \frac{\sigma_2}{\sigma_1}-2\rho}\matr{\frac{\sigma_2}{\sigma_1} - \rho \\ \frac{\sigma_1}{\sigma_2} - \rho}.
\]
When $\rho \approx \sigma_1/\sigma_2 \approx \sigma_2/\sigma_1$, the portfolio $x$ is fragile: the numerator and denominator in each of the components of $x$ are small and nearly symmetric, so small perturbations in $\sigma_1$, $\sigma_2$, or $\rho$ induce large changes in $x$. This corresponds to a regime in which $\bar\theta$ lies close to both $\theta_1$ and $\theta_2$, leading to instability in the active set.

\item \noindent{\bf Case 2:} $\rho > \sigma_1/\sigma_2$.  Proposition~\ref{prop.one.block} implies that $\bar\theta = \rho\theta_1$ and thus
\[
(1-\rho)y = \matr{ (1-\rho)\theta_1\\ 0} \text{ and } x = \matr{1 \\ 0}.
\]
Again, when $\rho \approx \sigma_1/\sigma_2 \approx \sigma_2/\sigma_1$, the portfolio $x$ is fragile: if parameters shift so that $\rho < \sigma_1/\sigma_2 < 1/\rho$, both assets become active (recovering Case 1), while if $\sigma_1/\sigma_2 > 1/\rho$, the active set flips entirely:
\[
(1-\rho)y = \matr{ 0 \\ (1-\rho)\theta_2} \text{ and } x = \matr{0 \\ 1}.
\]
\end{itemize}
These regimes can be summarized directly in terms of $\bar\theta$:
\begin{itemize}
    \item The portfolio is \textbf{fragile} when $\bar\theta \approx \theta_1 \approx \theta_2$, since small perturbations can switch the active set or significantly reweight exposures.
    \item The portfolio is \textbf{robust} when either $\bar\theta \ll \theta_1$ or $\bar\theta \gg \theta_2$, in which case the active set remains unchanged under small perturbations.
\end{itemize}

\end{example}

Using the expressions derived above we can illustrate the fragility regime predicted by the threshold $\bar\theta$ using this two-asset setting.
Consider $\sigma_1 = 0.99$, $\sigma_2 = 1$ and $\rho = 0.99$. In that case, the volatilities are nearly identical and the parameters lie close to the fragile regime predicted by the analytical solution e.g.,  $\rho=0.99=\frac{\sigma_1}{\sigma_2}\approx\frac{\sigma_2}{\sigma_1}$. 
As shown in Figure~\ref{fig:2x2_unstable}, small estimation errors induce abrupt allocation switches, high turnover, and a persistent gap between true and realized variances. 

\begin{figure}[!t]
    \centering
    \begin{subfigure}[t]{0.48\textwidth}
        \centering
        \includegraphics[width=\textwidth]{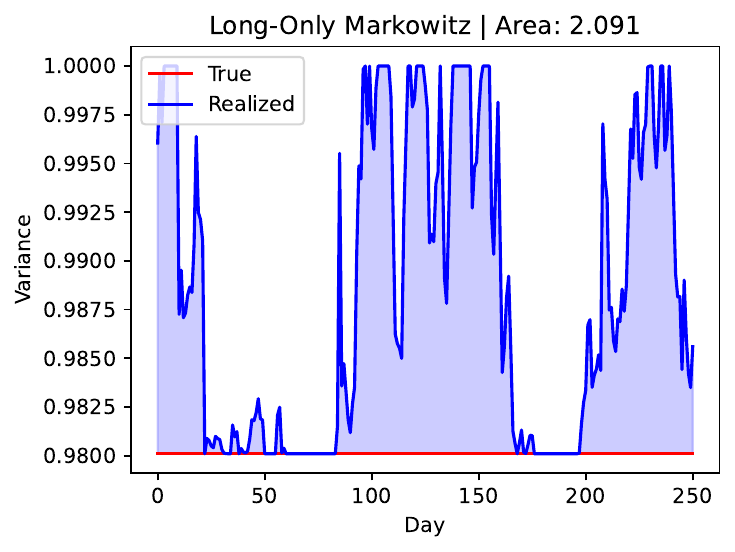}
        \caption{Evolution of true and realized portfolio variances. The shaded area quantifies the loss due to estimation noise, highlighting fragility.}
        \label{fig:2x2_unstable_var}
    \end{subfigure}
    \hfill
    \begin{subfigure}[t]{0.48\textwidth}
        \centering
        \includegraphics[width=\textwidth]{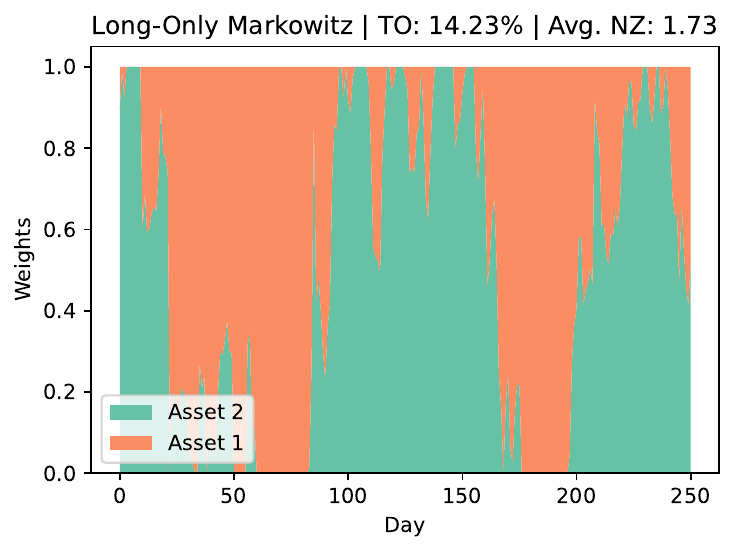}
        \caption{Time series of Markowitz weights. Frequent switches between assets result in 14.23\% turnover and an average of 1.73 active positions, illustrating instability.}
        \label{fig:2x2_unstable_por}
    \end{subfigure}

    \caption{Fragile long-only Markowitz portfolio with $\sigma_1 = 0.99$, $\sigma_2 = 1$, and $\rho = 0.99$.}
    \label{fig:2x2_unstable}
\end{figure}

It is easy to see that Example~\ref{ex.two} can be generalized as follows.

\begin{example}\label{ex.single.block}
Suppose~\eqref{eq.order} holds and that the solution of the fixed-point
equation~\eqref{eq.fixed.point} is given by~\eqref{eq.fixed.point.sol}
where $i$ is the largest index such that
\[
\bar\theta \le \theta_i .
\]
In this setting we obtain the following characterization.

\begin{itemize}
\item The long-only minimum-variance portfolio  is
\textbf{fragile} if $i > 1$ and 
$
\bar\theta \approx \theta_i $ or if $i < n$ and $
\bar\theta \approx \theta_{i+1}.  
$
In either of these cases a small perturbation in $\rho$ or in one of the volatilities $\sigma_i, \sigma_{i+1}$
changes the active set.

\item The long-only minimum-variance portfolio  is
\textbf{robust} in any other case, that is, when one of the following three possibilities occurs:
\begin{itemize}
\item[(a)] $i=1$ and $\bar\theta \gg \theta_2$
\item[(b)] $i=n$ and $\bar \theta \ll \theta_n$
\item[(c)] $1<i<n$ and $\theta_{i+1} \ll \bar\theta \ll \theta_i$
\end{itemize}
Indeed, in any of these three cases the active set remains unchanged under small parameter perturbations on $\rho$ or on any of the volatilities $\sigma_j$.

\end{itemize}

\end{example}

We now illustrate that fragility in higher dimensions does not require either extreme correlations or nearly identical volatilities.
Consider a three-asset long-only Markowitz portfolio with volatilities
$\sigma_1 = 1$, $\sigma_2 = 1.8$, and $\sigma_3 = 2$, and equicorrelation
$\rho = 0.48$. 
In this case the volatilities are clearly heterogeneous and the correlation is moderate, yet the parameters lie close to the threshold
$\theta_{i+1} < \bar\theta \le \theta_i$ identified in
Example \ref{ex.single.block}.
More specifically $\theta_{i+1}\approx\bar\theta$ as, 
\[
\theta_3 = \frac{1}{2} \approx\bar \theta = \frac{\rho}{1+\rho}\left(\theta_1+\theta_2\right) \approx 0.5045.
\]
The true and realized variances along with the time series of the portfolio weights and the associated turnover are shown in Figure \ref{fig:3x3_unstable}.

\begin{figure}[!t]
    \centering
    \begin{subfigure}[t]{0.48\textwidth}
        \centering
        \includegraphics[width=\textwidth]{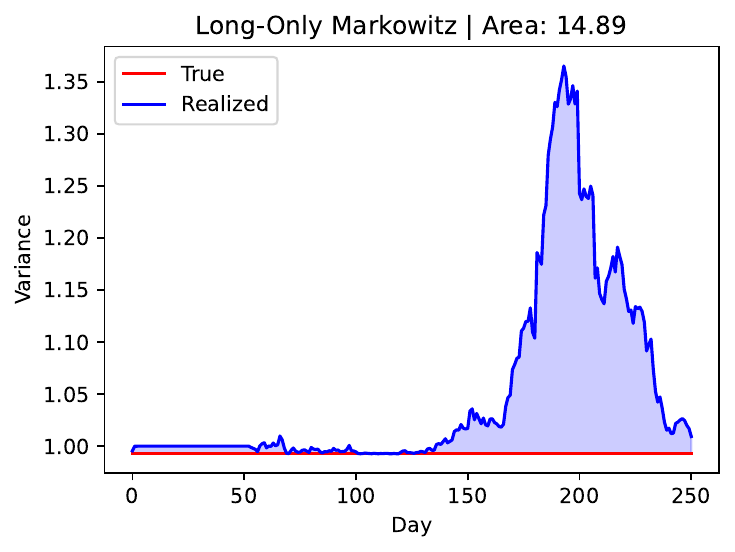}
        \caption{True  and realized portfolio variances. 
        The shaded area  remains significant because the parameters lie close to the fixed-point threshold.}
        \label{fig:3x3_unstable_var}
    \end{subfigure}
    \hfill
    \begin{subfigure}[t]{0.48\textwidth}
        \centering
        \includegraphics[width=\textwidth]{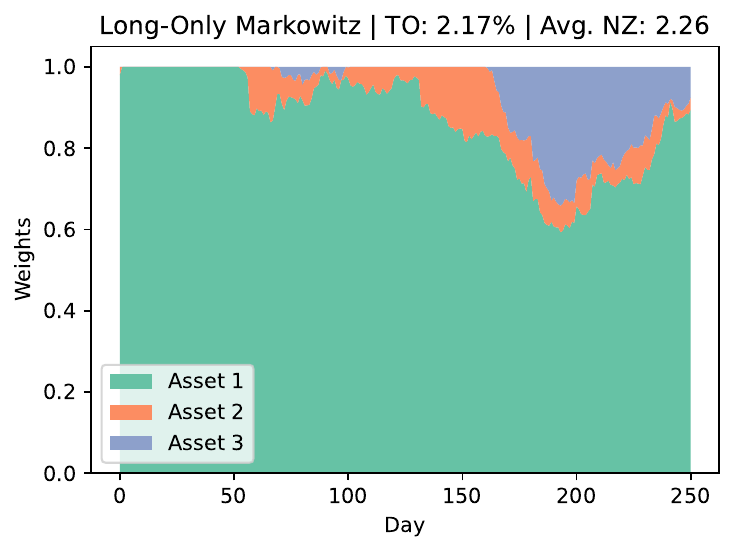}
        \caption{Time series of the estimated weights. 
        Repeated changes in the active set generate abrupt reallocations resulting 2.17\% turnover and an average of 2.26 active positions, despite the moderate correlation.}
        \label{fig:3x3_unstable_por}
    \end{subfigure}

    \caption{Fragile three-asset long-only Markowitz portfolio with 
    $\sigma_1 = 1$, $\sigma_2 = 1.8$, $\sigma_3 = 2$, and $\rho = 0.48$. 
    Although correlation is moderate and volatilities are not identical, the proximity to the fixed-point boundary produces instability in both allocations and out-of-sample variance.}
    \label{fig:3x3_unstable}
\end{figure}

Despite the absence of strong correlation, the estimated portfolio displays clear fragile behavior: the active set changes repeatedly, weights switch abruptly across assets, and a persistent gap emerges between true and realized variance.
Notably, the portfolio exhibits regimes with one, two, and three active positions, revealing a high sensitivity of the optimal solution to small perturbations.
This confirms that fragility is governed by the distance of $\bar{\theta}$ to its neighboring breakpoints, rather than by the magnitude of $\rho$ or by near-equality of the volatilities alone.

\section{Instability Under Block-Diagonal Constant-Correlation Structures}
\label{sec:structured_correlation}

The single-block constant-correlation model analyzed in the previous section serves as a tractable foundation for understanding how a common correlation component interacts with heterogeneous volatilities to determine both the active set and the fragility of the long-only minimum-variance portfolio. In real financial markets, however, assets are not organized as a single homogeneous group. Instead, they naturally form clusters (such as sectors, industries, regions, or investment styles) within which correlations tend to be relatively high and fairly uniform, while correlations across clusters are typically weaker and more diffuse. This empirical structure motivates the use of block-diagonal correlation matrices whose blocks follow a constant-correlation pattern.
Such block structures are widely employed as stylized representations of factor and hierarchical covariance models: a dominant market component induces a common level of comovement across all assets, while sector- or group-specific components generate clusters of strongly correlated securities.

Formally, we consider a universe partitioned into $K$ groups. Let $n = n_1+\cdots+n_K$, and 
denote by $\1_i \in \R^{n_i}$ the vector of all ones and by $I_i \in \R^{n_i\times n_i}$ the identity matrix of dimension $n_i$.
 Suppose that the correlation matrix
$\Omega \in \R^{n\times n}$ is of the form
\begin{equation}\label{eq.block.structure}
\Omega = (1-\rho)\Diag(\Omega_1 ,  \dots , \Omega_K) + \rho \1\1\transp,
\end{equation}
where each block $\Omega_i \in \R^{n_i\times n_i}$  is a constant-correlation matrix of the form $\Omega_i = (1-\rho_i)I_i + \rho_i \1_i \1_i\transp$ 
 {with $ \rho_i\in\left(-\frac{1}{n_i-1}, 1\right)$} for $i\in \{1,\dots,K\}$ and
 $-\left(
\sum_{i=1}^K
\frac{n_i}{1+(n_i-1)\rho_i}
-1
\right)^{-1}
< \rho < 1$.
The parameters $\rho_i$ capture within-block correlation for group $i$, while $\rho$ represents a common cross-asset correlation shared across all groups. These conditions ensure that $\Omega$ is a valid positive-definite correlation matrix (see Appendix~\ref{app,proof.pd}).

For a vector $y \in \R^n = \R^{n_1+\cdots+n_K}$, we denote by $y^{(i)} \in \R^{n_i}$ its $i$-th block, and similarly write $\sigma^{(i)}$ and $\theta^{(i)}$ for the $i$-th blocks of $\sigma$ and $\theta$ respectively.

\subsection{Closed-Form Solution}

The structure~\eqref{eq.block.structure} admits a closed-form solution to the transformed problem \eqref{eq.min.correl} that is directly analogous to the one-block case.
Here, the stability threshold for each group depends not only on its internal volatility configuration but also on the common cross-block correlation that couples all assets. Fragility is still driven by proximity to a single fixed-point boundary, but this boundary is now shaped by both within-block structure and interactions across blocks.
The following proposition formalizes this characterization; {the proof is provided in Appendix \ref{app:proof.prop.k.block}.}

\begin{restatable}{proposition}{propositionkblock}\label{prop.K.block} Suppose the correlation matrix $\Omega$ is of the form~\eqref{eq.block.structure}.
Then the solution to~\eqref{eq.min.correl} is block-separable and given for each group by
\begin{equation}\label{eq.sol.y.block}
y^{(i)} = \frac{1}{(1-\rho)(1-\rho_i)} \left(\theta^{(i)}  - (\hat \theta + \bar \theta_i) \1_i\right)^+,
\end{equation}
where the thresholds $\bar\theta_i$ and $\hat \theta$ are the solution to the following fixed-point problem
\begin{equation}\label{eq.fixed.point.block}
\begin{aligned}
\bar \theta_i &= \frac{\rho_i}{1-\rho_i} \cdot \1_i \transp\left(\theta^{(i)}  - (\hat \theta + \bar \theta_i) \1_i\right)^+\\
\hat \theta &=
\frac{\rho}{1-\rho}\cdot\sum_{j=1}^K \frac{1}{1-\rho_j} \cdot  \1_j \transp\left(\theta^{(j)}  - (\hat \theta + \bar \theta_j) \1_j\right)^+.
\end{aligned}
\end{equation}

\end{restatable}

The following corollary highlights two consistency checks obtained by suppressing either cross-block or within-block correlations.

\begin{corollary}\label{cor.sanity.block}
Consider the setting of Proposition~\ref{prop.K.block}. The expression for $y$ and the fixed-point thresholds $\hat \theta$ and  $\bar\theta_i$ reduce to familiar forms in the following extreme cases:

\begin{enumerate}
    \item If the cross-block correlation vanishes, $\rho = 0$, then $\hat \theta = 0$ and each block reduces to an independent constant-correlation problem with the blocks decoupled:
    \[
y^{(i)} =   \frac{1}{1-\rho_i} \left(\theta^{(i)}  -  \bar \theta_i \1_i\right)^+
\;
\text{ where } \qquad
  \bar \theta_i = \frac{\rho_i}{1-\rho_i} \, \1_i\transp \left(\theta^{(i)}  -  \bar \theta_i \1_i\right)^+.
    \]

    \item If the within-block correlations vanish, $\rho_j = 0$ for all $j=1,\dots,K$, then all blocks are coupled only through the global correlation $\rho$. In this case,
    \[
    y =   \frac{1}{1-\rho} \left(\theta  -  \hat \theta \1\right)^+
\;
\text{ where } \qquad
    \hat \theta = \frac{\rho}{1-\rho} \, \1\transp \left(\theta  -  \hat \theta \1\right)^+,
    \]
    recovering the single-block fixed-point problem from Proposition \ref{prop.one.block}.
\end{enumerate}
\end{corollary}

Equations~\eqref{eq.sol.y.block}--\eqref{eq.fixed.point.block} characterize the fragile regimes of the block-diagonal model. 
For each group, the active set is determined by the location of the threshold $\bar\theta_i+\hat \theta$ within the ordered sequence $\theta^{(i)}$, but the thresholds are jointly pinned down by the cross-block correlation. 
Fragility occurs whenever a threshold $\bar\theta_i+\hat \theta$ is close to a component of $\theta^{(i)}$, as in this case small perturbations in the parameters propagate through the fixed-point system, changing the active assets and inducing discontinuous reallocations across and within groups. 
The next examples connect this analytical description with the behavior of estimated portfolios.

\subsection{Illustrative Examples} \label{sec:K.block.examples}

The block structure allows us to localize fragility at the group level while accounting for the global coupling induced by $\rho$. 
The next example shows how the relative position of each $\bar\theta_j$ within its ordered sequence $\theta^{(j)}$ determines whether the portfolio is stable, and how proximity to a breakpoint in any block is sufficient to generate discontinuous reallocations.

\begin{example}\label{ex.block.K} 
Suppose the correlation matrix $\Omega$ has the block form~\eqref{eq.block.structure},
and the assets in each block are ordered by non-increasing volatility, that is,
\[
\sigma^{(i)}_1\le  \cdots \le \sigma^{(i)}_{n_i}
\iff
\theta^{(i)}_{n_i} \le  \cdots \le \theta^{(i)}_1
\]
holds for each block $i=1,\dots,K$.
 
Let $\bar\theta = (\bar\theta_1,\dots,\bar\theta_K)$ be the solution of the
fixed-point system~\eqref{eq.fixed.point.block}, and for each block
$i=1,\dots,K$ let $j_i \in \{1,\dots,n_i\}$ satisfy
\[
\theta^{(i)}_{j_i+1} \le \hat \theta + \bar\theta_i < \theta^{(i)}_{j_i}.
\]
This yields the following blockwise characterization.

\begin{itemize}
\item The long-only minimum-variance portfolio is \textbf{fragile} if for at least one block $i$,
$
\hat \theta +\bar\theta_i \approx \theta^{(i)}_{j_i}
$ or $
\hat \theta +\bar\theta_i \approx \theta^{(i)}_{j_i+1},
$
since small perturbations in the volatilities or correlations may alter the
active set within that block.

\item The long-only minimum-variance portfolio is \textbf{robust} if for every block $i$,
$
\theta^{(i)}_{j_i+1} \ll \hat \theta +\bar\theta_i \ll \theta^{(i)}_{j_i},
$
so that the active sets remain invariant under small parameter changes.
\end{itemize}

\end{example}

Consider a seven-asset example with an explicit block structure. 
Let $\sigma = \1$, so that $V = \Omega$, and partition the assets into $K=3$ groups with
$\rho_1=\rho_3=0.9$, $\rho_2=0.8$, and cross-block correlation $\rho=0.6$. 
The corresponding $\hat\theta=0.829$  yields the fixed-point thresholds
$
\hat \theta\1 +\bar\theta = (0.991,\,0.981,\,0.993).
$
Since within each block the coordinates satisfy $\theta^{(i)}_j = 1 \approx \hat \theta +\bar\theta_i$, all groups operate close to their respective breakpoints. 
The system therefore lies in a fragile regime: small perturbations propagate through the coupled fixed-point equations and induce changes in the active sets across blocks.
The block covariance structure is displayed in Figure~\ref{fig:cov_ex3}, while the resulting variance dynamics and portfolio weights are reported in Figure~\ref{fig:block_ex3}. 
Consistent with the analytical prediction, estimation noise leads to frequent reallocations within and across groups, elevated turnover, and a persistent gap between true and realized variances.

\begin{figure}[h]
    \centering
        \includegraphics[width=0.6\textwidth]{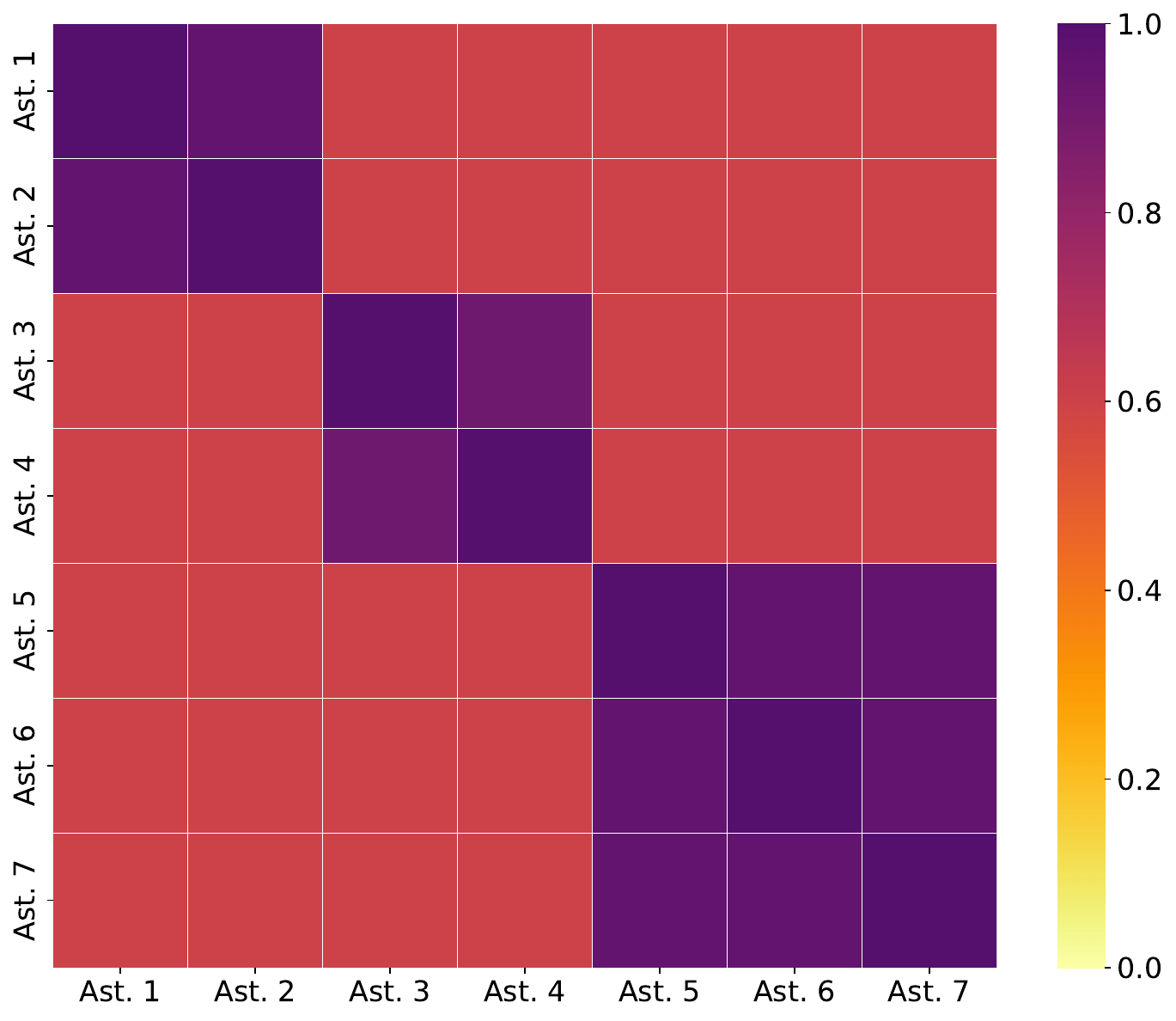}
        \caption{Three-block correlation matrix $\Omega$ with within-block correlations $\rho_1=\rho_3=0.9$, $\rho_2=0.8$ and cross-block correlation $\rho=0.6$.}
        \label{fig:cov_ex3}
\end{figure}

\begin{figure}[!t]
    \centering
    \begin{subfigure}[t]{0.48\textwidth}
        \centering
        \includegraphics[width=\textwidth]{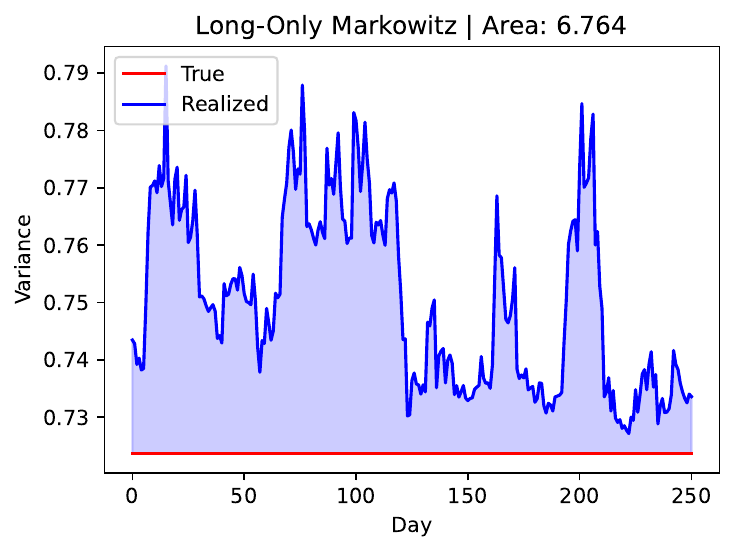}
        \caption{True and realized portfolio variances; the shaded region measures the performance loss due to estimation noise.}
        \label{fig:var_block_ex3}
    \end{subfigure}
    \hfill
    \begin{subfigure}[t]{0.48\textwidth}
        \centering
        \includegraphics[width=\textwidth]{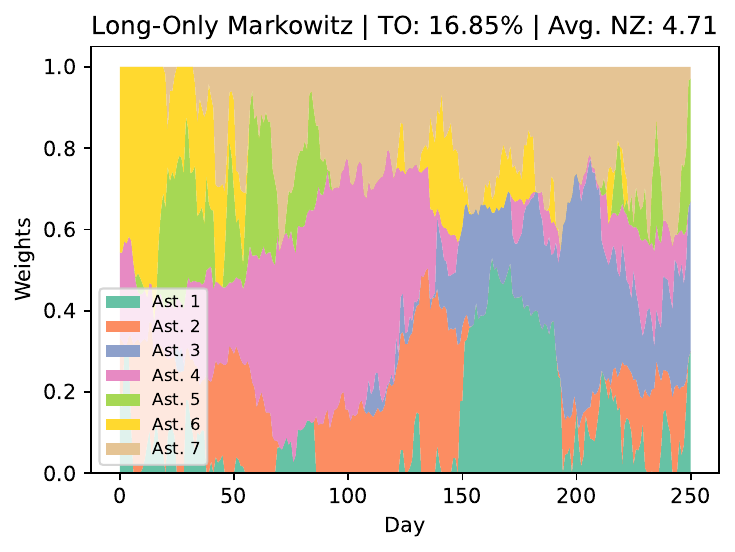}
        \caption{ Time series of portfolio weights, showing frequent reallocations across blocks (turnover 16.85\%, average 4.71 active positions).}
        \label{fig:por_block_ex3}
    \end{subfigure}

    \caption{Fragility long-only portfolio under the three-block structure.  }
    \label{fig:block_ex3}
\end{figure}

We also illustrate that block-induced fragility arises in real data. 
Consider a universe consisting of Visa, Mastercard, Shell, Chevron, Exxon Mobil, Coca Cola, and Pepsi, using monthly returns from January 2018 to January 2025. 
We residualize each return series with respect to the market factor, effectively working within a one-factor model and isolating the cross-sectional dependence structure. 
This step removes the dominant market-wide comovement and reveals the sectoral correlation patterns most relevant for portfolio allocation.
The residual correlation matrix, displayed in Figure~\ref{fig:cov_ex3r}, exhibits a clear approximate block structure: payment networks (Visa, Mastercard), energy firms (Shell, Chevron, Exxon Mobil), and beverage companies (Coca Cola, Pepsi) form visibly correlated groups once the market component is removed. While the matrix is not exactly block diagonal, the dominant organization is nevertheless hierarchical and clustered, with residual cross-group interactions and estimation noise perturbing the structure. 

Figure~\ref{fig:block_ex3r} illustrates the resulting instability.
Although the portfolio maintains, on average, approximately five to six active
positions, Visa, Pepsi, and, for the most part, Coca-Cola remain invested
across the sample period, while Mastercard and the three Energy names
repeatedly enter and exit the allocation, generating substantial turnover and
a persistent gap between true and realized variances.
Consistent with the fragility mechanism characterized by the block fixed-point system in Proposition~\ref{prop.K.block}, the example suggests that small perturbations around an approximately clustered dependence structure can induce significant changes in the active set. While the empirical correlation matrix is not exactly block diagonal, its dominant organization is sufficiently structured for the instability mechanisms identified by the analytical model to become visible in practice. 
This observation motivates the structured block-diagonal approximation developed later in the paper (see Section \ref{sec:block.approx}), which is designed to capture the dominant dependence patterns while attenuating noisy cross-group interactions, thereby enabling the instability mechanisms described in Proposition~\ref{prop.K.block} to characterize instances that arise from empirical data.

\begin{figure}[h]
    \centering
        \includegraphics[width=0.6\textwidth]{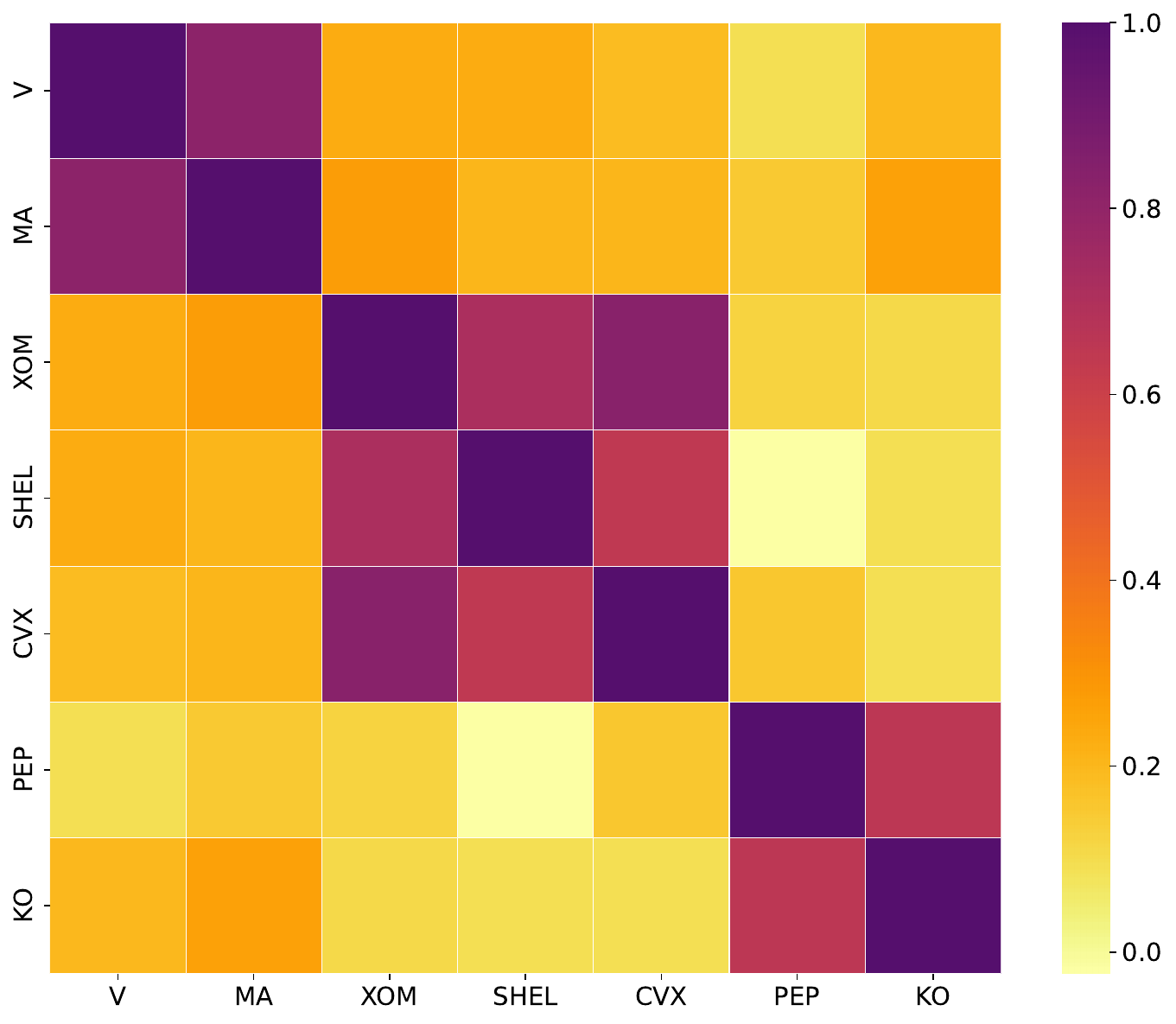}
        \caption{Residual correlation matrix $\Omega$ after removing the market factor from monthly returns (2018–2024). 
Sectoral clustering becomes visible once the common market component is eliminated.}
        \label{fig:cov_ex3r}
\end{figure}

\begin{figure}[!t]
    \centering
    \begin{subfigure}[t]{0.48\textwidth}
        \centering
        \includegraphics[width=\textwidth]{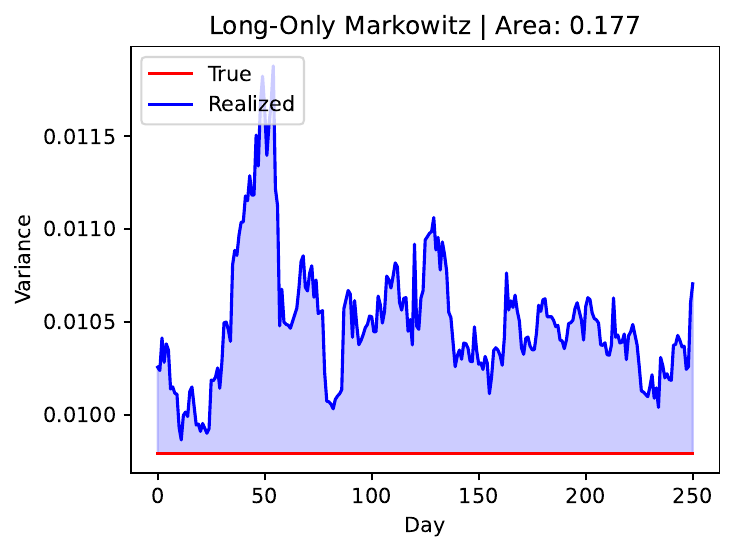}
        \caption{True and realized portfolio variances; the shaded region measures cumulative variance loss due to estimation noise. }
        \label{fig:var_block_ex3r}
    \end{subfigure}
    \hfill
    \begin{subfigure}[t]{0.48\textwidth}
        \centering
        \includegraphics[width=\textwidth]{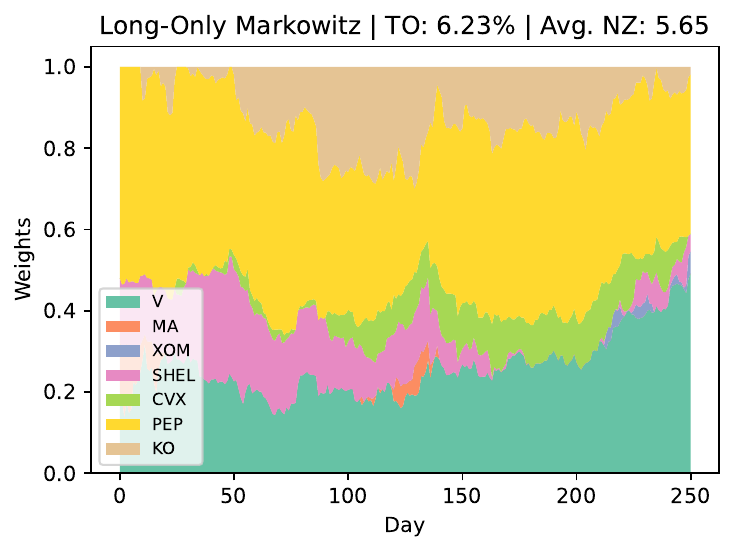}
        \caption{Portfolio weights over time, showing recurrent asset inclusion and exclusion (turnover 6.23\%, average 5.65 active positions).}
        \label{fig:por_block_ex3r}
    \end{subfigure}

    \caption{Long-only minimum-variance portfolio for the residualized real-data universe. }
    \label{fig:block_ex3r}
\end{figure}

\section{Proposed Robustification Approaches} \label{sec:rob_robustification}

The previous sections show that fragility in the long-only minimum-variance problem~\eqref{eq.minvar} arises when the fixed-point thresholds lie close to active-set breakpoints. 
Instability is therefore not driven by estimation error alone, but by the way the correlation structure couples assets and governs the truncation mechanism in the solution. 
When thresholds approach these critical values, even small perturbations in volatilities or correlations modify the truncation pattern induced by the positive-part operator, leading to discontinuous reallocations and performance degradation.

The robustification strategies developed below are guided by this structural
insight.  Rather than modifying the optimization problem itself, we adjust the
correlation structure in a manner that systematically shifts the fixed-point
thresholds away from critical regions.
Each approach can be interpreted as a structured shrinkage operation on the correlation matrix. 
By attenuating cross-asset coupling, these procedures reduce the sensitivity of the active set to small perturbations in volatilities or correlations.

For clarity, we proceed in increasing order of generality. 
We begin with the constant-correlation case, then extend the ideas to block-diagonal structures, and finally consider general correlation matrices via structured approximation.

\subsection{Constant Correlation Matrix}\label{sec.constant.correl}

We consider first the simple case when $\Omega$ is a constant correlation matrix of the form~\eqref{eq.single.block}.  The two shrinkage approaches that we describe below provide the building blocks for more general cases.  

\subsubsection{Correlation-Oblivious Shrinkage}

The simplest robustification approach ignores correlations altogether and sets $y_i = \tfrac{1}{1-\rho}\theta_i$ for all $i$, which corresponds to choosing $\bar\theta = 0$ in Proposition~\ref{prop.one.block}. Since the factor $\tfrac{1}{1-\rho}$ is common to all components, it cancels upon normalization, and substituting into~\eqref{eq.sol.x} yields
$$x_i = \frac{\theta_i^2}{\sum_{j=1}^n \theta_j^2} = \frac{1/\sigma_i^2}{\sum_{j=1}^n 1/\sigma_j^2},\qquad i=1,\dots,n,$$
that is, inverse-variance weights. This is equivalent to solving~\eqref{eq.minvar} with $\Omega$ replaced by the identity, i.e., $V=\Sigma^2$, and within the constant-correlation parametrization it corresponds to \emph{shrinking $\rho$ to zero}.
From the perspective of Proposition~\ref{prop.one.block}, setting $\bar\theta=0$ implies $(\theta_i - \bar\theta)^+ = \theta_i > 0$ for every asset, so the active set is the full index set ${1,\dots,n}$ regardless of the volatility configuration. As a consequence, the positive-part truncation in~\eqref{eq.sol.y} is eliminated and the support of the portfolio becomes invariant to perturbations in the parameters.

A direct implication of this construction is that the resulting portfolio is fully dense, assigning strictly positive weight to all assets in the same spirit as diversification-based approaches such as \citet{lopez2016building} and \citet{raffinot2018hierarchical}. 
While this guarantees stability, it may move the solution away from the true minimum-variance allocation in settings where highly correlated assets exhibit considerably different volatilities, since the procedure retains all assets rather than concentrating on the lower-volatility ones.

The remaining approaches can be interpreted as partial or structured shrinkage schemes. They preserve selected aspects of the correlation structure, allowing sparsity to emerge when supported by the data, while still attenuating the coupling that generates fragility.

\subsubsection{Correlation-Aware Shrinkage}

We now introduce a structured shrinkage scheme that preserves part of the correlation information while explicitly controlling proximity to active-set breakpoints.
Unlike the correlation-oblivious approach, which eliminates the truncation mechanism entirely, this procedure retains the positive-part structure of Proposition~\ref{prop.one.block} and therefore allows sparse portfolios that more closely reflect the true minimum-variance solution.

Specifically, we consider portfolios of the form
\[
y = \frac{1}{1-\rho}\left(\theta -\theta_\varepsilon \1 \right)^+,
\]
where the adjusted threshold $\theta_\varepsilon$ is obtained from the
fixed-point solution $\bar\theta$ through a robustness correction.

The correction ensures that $\theta_\varepsilon$ maintains a minimum distance $\varepsilon$ from all breakpoints $\theta_i$, stabilizing the active set and preventing small perturbations in volatilities or correlations from triggering abrupt reallocation.
Since the factor $\tfrac{1}{1-\rho}$ is common to all components, it cancels upon normalization and the resulting portfolio weights are
\[
x_i = \frac{(\theta_i - \theta_\varepsilon)^+\cdot\theta_i}{\displaystyle\sum_{j=1}^n (\theta_j - \theta_\varepsilon)^+\cdot\theta_j},
\qquad i=1,\dots,n.
\]

To correct the threshold $\theta_\varepsilon$, we introduce a robustness parameter $\varepsilon > 0$ that specifies the minimal distance a threshold should maintain from any breakpoint.
This parameter can be externally calibrated or, in the absence of prior information, a reasonable choice is
\begin{equation} \label{eq:epsilon.heuristic}
\varepsilon := \frac{\max_i \theta_i}{n}.
\end{equation}
Given the  sequence of breakpoints ordered as \eqref{eq.order},
we define the robust-feasible set $\mathcal{F}$ as
\[
\mathcal{F}
= (-\infty,\, \theta_n - \varepsilon]
\;\cup\;
\bigcup_{\substack{i=1 \\ \theta_i - \theta_{i+1} \,\geq\, 2\varepsilon}}^{n-1}
[\theta_{i+1} + \varepsilon,\, \theta_i - \varepsilon].
\]

The final adjusted threshold $\theta_\varepsilon$ is obtained by projecting the initial fixed-point solution $\bar{\theta}$ onto $\mathcal{F}$:
\[
\theta_\varepsilon=\operatorname{proj}_{\mathcal{F}}(\bar{\theta})=\arg \min _{x \in \mathcal{F}}|x-\bar{\theta}|.
\]
This procedure ensures that $\theta_\varepsilon$ lies at least $\varepsilon$ away from all breakpoints, stabilizing the active set while preserving sparsity.

The robustification procedure admits a natural interpretation in terms of an
equivalent correlation parameter.
Given the adjusted threshold $\theta_\varepsilon$, define $\rho_\varepsilon$ as the unique solution to the fixed-point equation of
Proposition~\ref{prop.one.block} evaluated at $\theta_\varepsilon$:
\[
\theta_\varepsilon
= \frac{\rho_\varepsilon}{1-\rho_\varepsilon}
  \sum_{j: \theta_j > \theta_\varepsilon}(\theta_j - \theta_\varepsilon).
\]
It is easy to see that $\theta_\varepsilon$ is the solution to the fixed-point problem~\eqref{eq.fixed.point} with correlation $\rho_\varepsilon$ in
place of $\rho$. 
Consequently, the resulting portfolio is identical to the exact solution of
Proposition~\ref{prop.one.block} with correlation $\rho_\varepsilon$ in
place of $\rho$.

The above robustification has two equivalent interpretations.
It can be interpreted as a \emph{threshold shrinkage} from $\bar \theta$ to $\theta_\varepsilon$.  It can also be interpreted as a \emph{correlation shrinkage} from $\rho$ to  $\rho_\varepsilon$.
The truncation operator therefore remains operative: assets whose inverse
volatilities fall below the adjusted threshold are excluded, allowing the portfolio
to remain sparse when supported by the data and preserving economically meaningful
selection effects while mitigating threshold-induced instability.

\paragraph{Illustrative Example}

We illustrate the behavior of the correlation-oblivious and correlation-aware robustification schemes in two single-block eight-asset universes. In both cases, the robust thresholds are constructed using the heuristic $\varepsilon$ selection described in \eqref{eq:epsilon.heuristic}. In the first scenario, all assets have identical volatilities ($\sigma = \1$) and a constant correlation of $\rho = 0.90$; here both approaches effectively stabilize the portfolio, maintaining full diversification with minimal turnover, in contrast to the highly unstable Markowitz solution. The corresponding performance and allocations are shown in Figures~\ref{fig:var_8x8_one_c} and ~\ref{fig:por_8x8_one_c}, where the top panel reports the variance over time and the bottom panel shows the evolution of the portfolio weights.

In the second scenario, the asset volatilities are heterogeneous ($\sigma = (1,1,1,1,2,2,2,2)$) with the same correlation structure. The correlation-oblivious method maintains full allocations, and fails to prioritize the lower-volatility assets, producing suboptimal performance. By contrast, the correlation-aware approach adjusts the threshold using the $\varepsilon$-based correction, selectively investing in the lower-volatility assets, which improves realized variance, reduces turnover, and naturally generates a sparse portfolio. The corresponding performance and weight dynamics are reported in Figures~\ref{fig:var_8x8_onetwo_c} and ~\ref{fig:por_8x8_onetwo_c}, again with variance on top and portfolio allocations below.

\begin{figure}[htbp]
\centering

\begin{subfigure}{\textwidth}
    \centering
    \includegraphics[width=\textwidth]{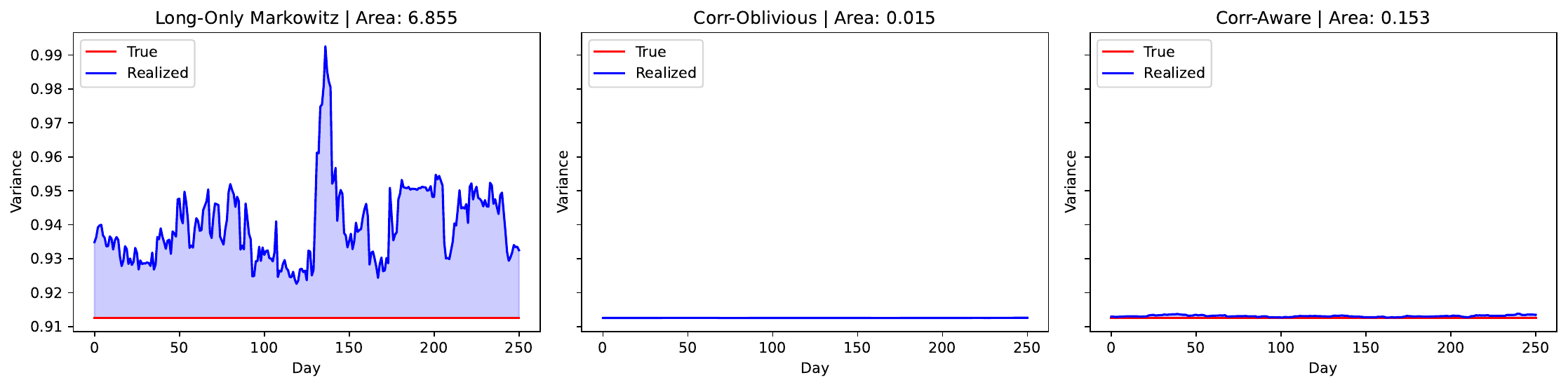}
    \caption{Variance (uniform volatilities).}
    \label{fig:var_8x8_one_c}
\end{subfigure}

\vspace{0.3cm}

\begin{subfigure}{\textwidth}
    \centering
    \includegraphics[width=\textwidth]{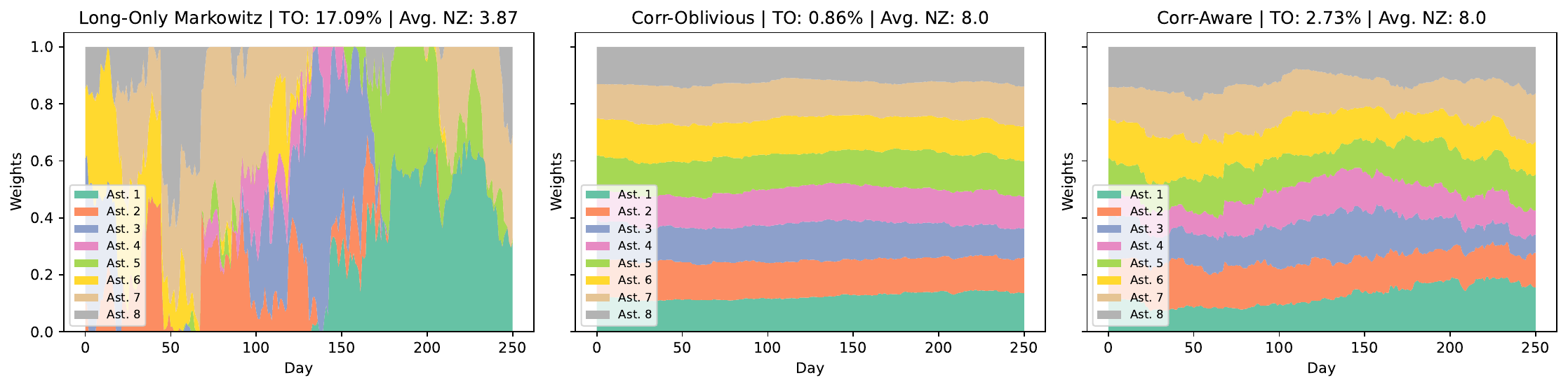}
    \caption{Weights (uniform volatilities).}
    \label{fig:por_8x8_one_c}
\end{subfigure}

\vspace{0.4cm}

\begin{subfigure}{\textwidth}
    \centering
    \includegraphics[width=\textwidth]{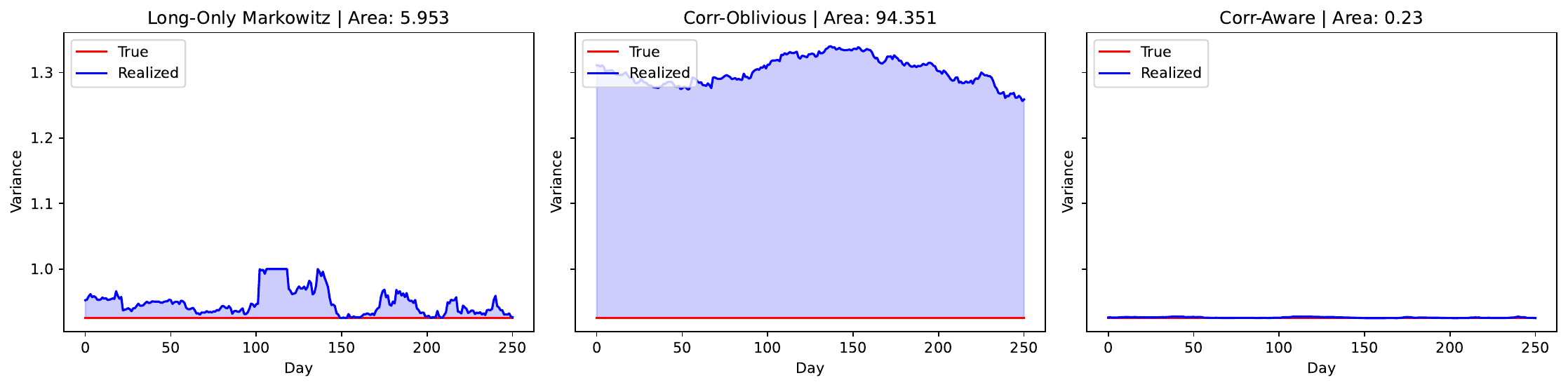}
    \caption{Variance (heterogeneous volatilities).}
    \label{fig:var_8x8_onetwo_c}
\end{subfigure}

\vspace{0.3cm}

\begin{subfigure}{\textwidth}
    \centering
    \includegraphics[width=\textwidth]{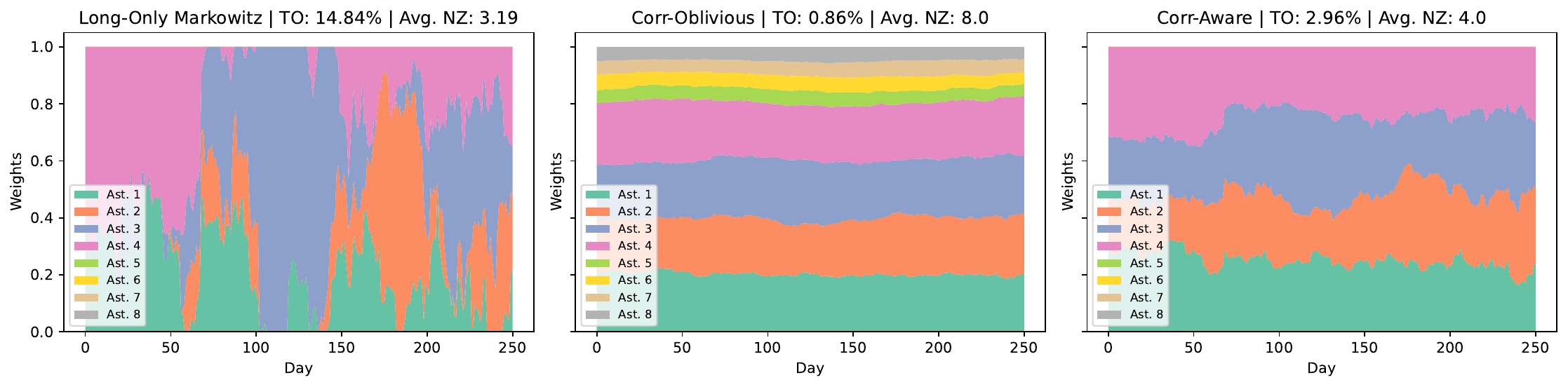}
    \caption{Weights (heterogeneous volatilities).}
    \label{fig:por_8x8_onetwo_c}
\end{subfigure}

\caption{Eight-asset universe under two volatility configurations. 
Top two panels: uniform volatilities $\sigma =\1$ with $\rho = 0.90$. 
Bottom two panels: heterogeneous volatilities $\sigma = (1,1,1,1,2,2,2,2)$ with $\rho = 0.90$. 
The Markowitz portfolio exhibits substantial turnover and instability in the active set. 
Robustification stabilizes allocations and reduces turnover. 
When volatilities differ, the correlation-aware approach preserves sparsity by allocating primarily to lower-volatility assets, leading to improved realized variance.}
\label{fig:8x8_combined}
\end{figure}

In both cases, the Markowitz portfolio (left column) exhibits pronounced fragility, with frequent changes in the active set leading to significant turnover and a persistent gap from the true optimum, as reflected in the large variance discrepancies. By contrast, the proposed correlation-aware shrinkage (right column) produces portfolios that closely track the true minimum-variance solution. In the uniform-volatility setting, it maintains full diversification across all assets, while in the heterogeneous case it selectively excludes the higher-volatility assets, yielding sparse and stable allocations with consistently low turnover.

The main distinction arises with the correlation-oblivious shrinkage approach (middle-column). Although it achieves the lowest turnover across both scenarios, it systematically produces fully dense portfolios. This behavior is beneficial when volatilities are homogeneous, but in the presence of heterogeneity it assigns positive weights to all assets (including those that are highly volatile and strongly correlated) resulting in portfolios that deviate substantially from the true minimum-variance allocation.

\subsection{Block-Diagonal Correlation Matrix}\label{sec.block.correl}

We now extend the previous shrinkage constructions to the case where $\Omega$ exhibits a block-diagonal structure of the form~\eqref{eq.block.structure}. 
This setting introduces a hierarchical dependence pattern: stronger within-block correlations and weaker cross-block coupling governed by $\rho$. 
The key idea is to exploit this structure to perform \emph{structured shrinkage}, attenuating instability both within and across groups while preserving economically meaningful interactions.
We propose a two-tier robustification scheme that decomposes the problem along the block structure and applies shrinkage at each level.

\paragraph{Tier 1 (within-block shrinkage).}
For each block $i=1,\dots,K$, let $\Sigma^{(i)} = \Diag(\sigma^{(i)})$ and
$V^{(i)} = \Sigma^{(i)}\,\Omega_i\,\Sigma^{(i)}$ be the within-block covariance
matrix. We solve the subproblem defined by $V^{(i)}$ and construct a robust
portfolio $x^{(i)}$ using either the correlation-oblivious or correlation-aware
shrinkage approaches from Section~\ref{sec.constant.correl}.
This step stabilizes the active set within each group by attenuating the effect
of within-block correlations.

\paragraph{Tier 2 (cross-block shrinkage).}
We then aggregate the block-level portfolios into a reduced $K$-dimensional
problem that captures cross-block interactions.
Specifically, we define the aggregated volatilities as
\begin{equation}
    (\sigma^{\text{agg}}_i)^2 := (x^{(i)})^\top V^{(i)}\, x^{(i)},
\end{equation}
and construct the aggregated covariance matrix
\begin{equation}
\Omega^{\text{agg}} = (1-\rho)I + \rho\,\1\1\transp,
\qquad
V^{\text{agg}} = \Sigma^{\text{agg}}\,\Omega^{\text{agg}}\,\Sigma^{\text{agg}},
\end{equation}
where $\Sigma^{\text{agg}}=\Diag(\sigma^{\text{agg}})$.
A robust portfolio $x^{\text{agg}}\in \R^K$ is then computed for this reduced problem
by applying the same shrinkage approach used in Tier~1, either the
correlation-oblivious or correlation-aware procedure from the
constant-correlation case.

\paragraph{Reconstruction.}
The final portfolio is obtained by combining the block-level and aggregate allocations:
\begin{equation}
x =
\begin{pmatrix}
x^{\text{agg}}_1 \cdot x_1 \\
\vdots \\
x^{\text{agg}}_K \cdot x_K
\end{pmatrix},
\end{equation}
where $x^{\text{agg}}_i \in \mathbb{R}$ is the scalar aggregate weight assigned
to block $i$ and $x^{(i)} \in \mathbb{R}^{n_i}$ is the within-block portfolio vector.
This construction admits a natural interpretation as \emph{multi-level shrinkage}. 
At the first level, correlations within each block are shrunk (either completely or through threshold adjustment) to stabilize the local active sets. 
At the second level, the cross-block correlation $\rho$ is effectively shrunk through the aggregated problem, reducing the coupling between groups.

In contrast to traditional global shrinkage approaches, this decomposition preserves heterogeneity across groups and allows sparsity to emerge both within and across blocks. 
Assets can be excluded at the block level through the within-block truncation mechanism, while entire groups can be down-weighted or excluded through the aggregated allocation. 
As a result, the procedure balances robustness and fidelity to the original covariance structure, yielding portfolios that remain stable under perturbations while retaining the ability to concentrate on the most attractive assets.

\subsection{Block-Diagonal Approximation}\label{sec:block.approx}

The robustification procedures developed in Section~\ref{sec.block.correl}
extend naturally to general correlation matrices through structured approximation.
Given an estimated correlation matrix $\Omega^{\mathrm{data}}$, we approximate it
by a matrix of the form~\eqref{eq.block.structure}, and then apply the shrinkage
methods derived for structured block correlations.
This can be interpreted as a data-driven shrinkage of the dependence structure:
strong local correlations are preserved within groups of similar assets, while
weak and noisy cross-group dependencies are compressed into a common background
correlation.

\subsubsection{Clustering Procedure}
To construct the block structure, assets are first grouped according to their
empirical correlations using hierarchical agglomerative clustering.
Following standard correlation-network constructions
\cite{lopez2016building, raffinot2018hierarchical, raffinot2018hcaa, mantegna1999hierarchical}, correlations are mapped into
a Euclidean-consistent distance via
\begin{equation}
d_{ij}=\sqrt{2(1-\Omega_{ij})}, \qquad i,j\in\{1,\dots,n\}.
\end{equation}
This transformation embeds the correlation matrix into a metric space in which
nearby assets exhibit similar dependence patterns.

We employ single-linkage clustering, where the distance between two clusters $C_1,C_2$ is defined as
\begin{equation}
d(C_1,C_2)=\min\{d_{ij}:i\in C_1,\;j\in C_2\}.
\end{equation}
Starting from singleton clusters, the algorithm recursively merges the pair of clusters with minimum inter-cluster distance. Single linkage is particularly suitable in financial applications, where dependence is largely factor-driven and propagates through chains of partial correlations rather than forming well-separated, compact groups \cite{raffinot2018hierarchical}. Consequently, it captures connectivity structures by merging clusters whenever a strong pairwise relationship exists, effectively tracing paths of shared dependence across assets \cite{lopez2016building}. 

Alternative linkage criteria, including complete, average, weighted, centroid, median, and Ward linkage, have been extensively studied in the hierarchical portfolio literature and can produce different cluster structures and portfolio allocations \cite{lopez2016building, raffinot2018hierarchical, raffinot2018hcaa}. In this paper, however, we restrict attention to single linkage for simplicity and to maintain a direct connection between the recovered dependence chains and the block structures underlying our theoretical analysis. Investigating how alternative linkage criteria interact with the proposed methodology, and whether they yield improved cluster recovery or portfolio performance, is an interesting direction for future research.

\subsubsection{Selection of the Clustering Threshold}
A key step is determining the level at which the dendrogram is cut to produce
the final partition.
Rather than fixing the number of clusters a priori, we adopt a data-driven rule
based on the sequence of linkage distances
\[
\ell_1 \le \ell_2 \le \cdots \le \ell_{n-1},
\]
where $\ell_k$ denotes the distance at which the $k$-th merge occurs.
Define the successive gaps
\begin{equation}
g_k=\ell_{k+1}-\ell_k, \qquad k=1,\dots,n-2,
\end{equation}
and let
\[
k^*=\arg\max_k g_k
\]
identify the largest discontinuity in the merging process. We then select the
cutoff
\begin{equation}
\tau=\frac{1}{2}(\ell_{k^*}+\ell_{k^*+1}).
\end{equation}
This rule selects the partition at the scale of maximal structural separation,
corresponding to the largest jump in clustering distance. Intuitively, it avoids
cutting the dendrogram in ambiguous regions where clusters merge gradually, and
instead identifies the point at which distinct dependence groups begin to merge.

Figure~\ref{fig:market_corr_matrices} revisits the market residual correlation matrix from Section \ref{sec:K.block.examples}  from a clustering perspective. When the assets are listed alphabetically, the block structure motivating our analysis is largely hidden. After applying single-linkage hierarchical clustering, however, the matrix reorganizes into three clear diagonal blocks corresponding to economically meaningful sectors. Thus, the block structure assumed in our theoretical development need not be specified exogenously; it can be recovered directly from market data. The corresponding dendrogram in Figure~\ref{fig:market_dendrogram} exhibits a pronounced gap between the fourth and fifth merge distances ($g_4=0.373$), roughly five times larger than any other successive gap. Using the largest-gap rule places the threshold at $\tau \approx 1.017$, cleanly separating the three sector-level clusters from their eventual cross-sector merger. This thresholding rule simultaneously identifies the cluster memberships and determines the number of clusters in a fully data-driven manner, yielding a robust partition that is directly tied to the observed correlation structure.

\begin{figure}[htbp]
    \centering
    \includegraphics[width=\textwidth]{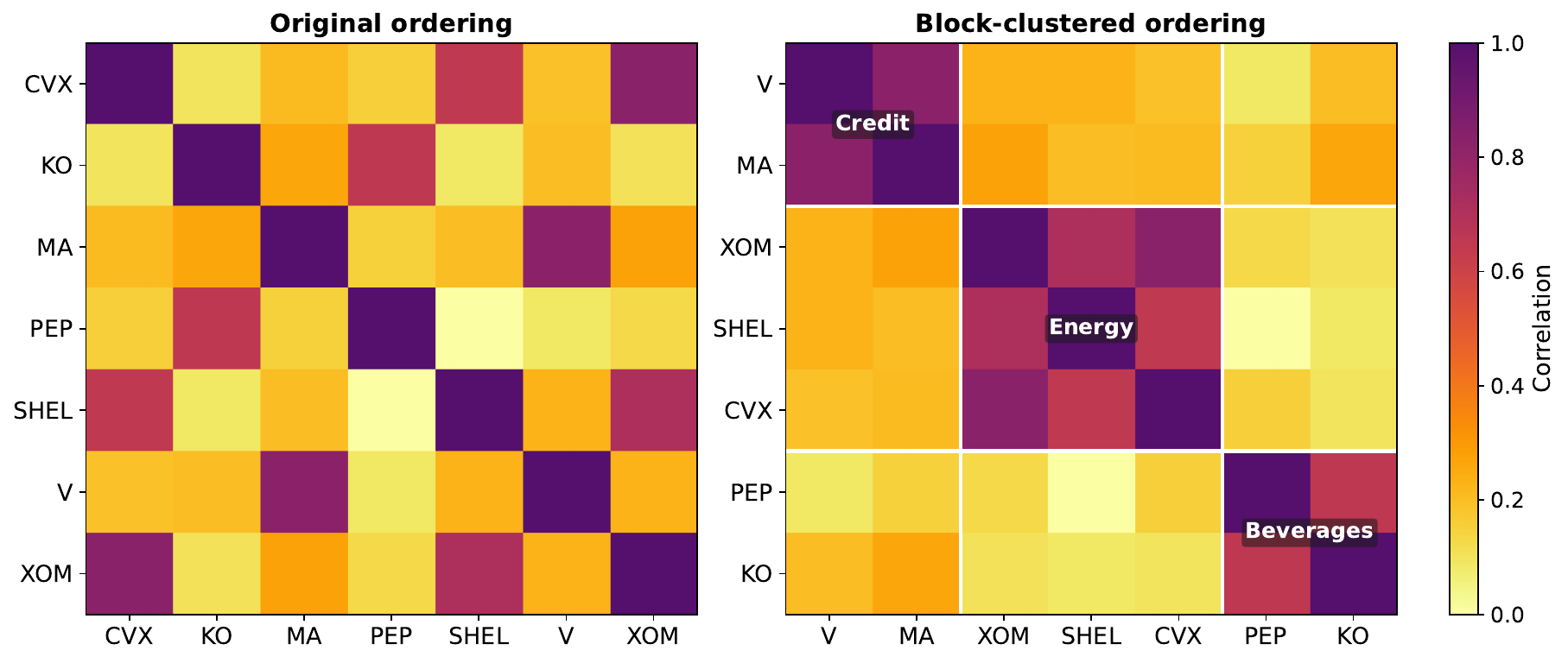}
    \caption{Residual correlation matrix $\Omega$ of the seven-ticker market example of
Section \ref{sec:K.block.examples} (monthly returns, 2018-2025, market factor removed).
We show the assets in alphabetical order; no block structure is visible along with the same matrix reordered by the single-linkage partition.}
    \label{fig:market_corr_matrices}
\end{figure}

\begin{figure}[htbp]
    \centering
    \includegraphics[width=\textwidth]{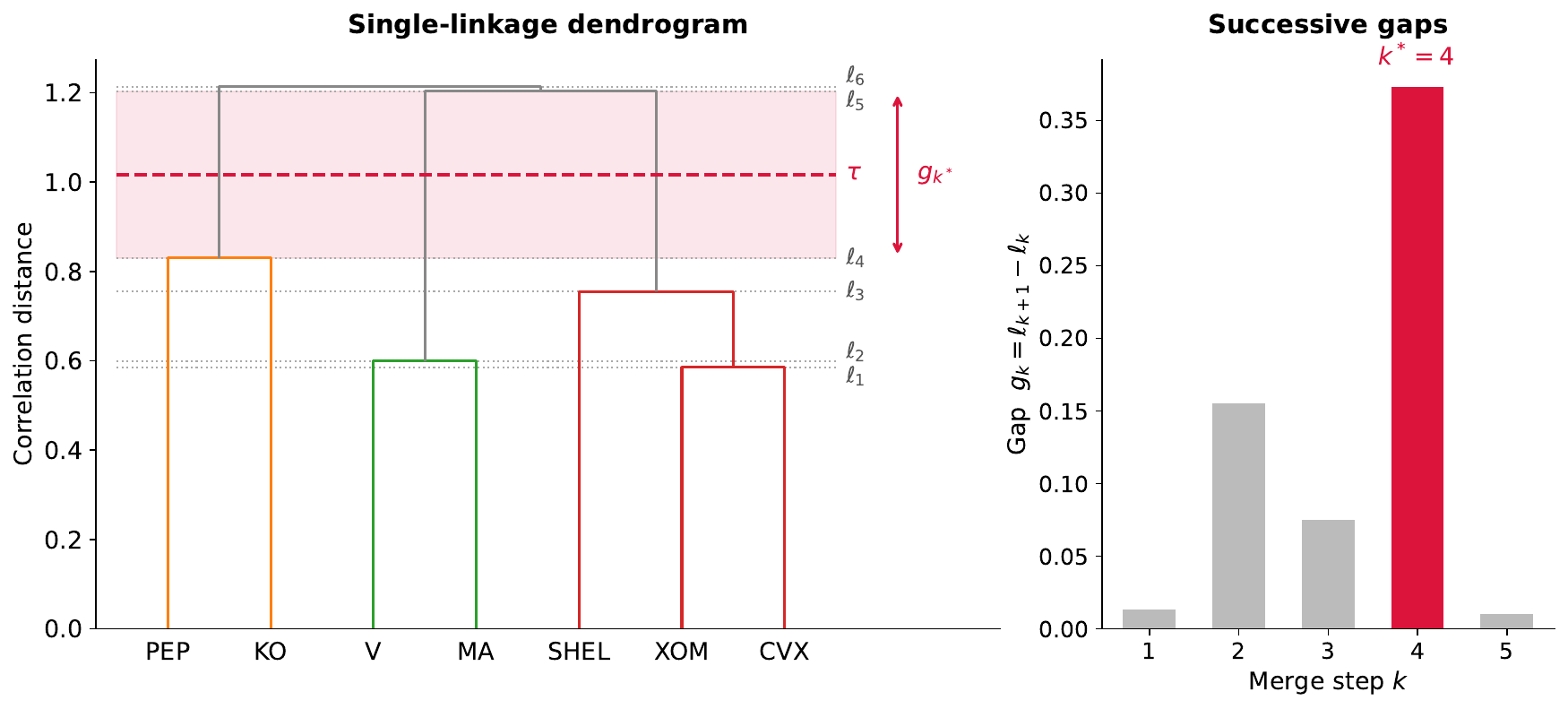}
    \caption{Largest-gap cut applied to the single-linkage dendrogram of
$\Omega$ from the market example in  Section \ref{sec:K.block.examples}.
{Left:} merge distances $\ell_1 \le \cdots \le \ell_6$ (dotted lines);
the shaded band spans the dominant gap $[\ell_{k^*},\ell_{k^*+1}]$ and the
dashed line marks the threshold $\tau = \tfrac{1}{2}(\ell_{k^*}+\ell_{k^*+1})
+ \varepsilon$.
{Right:} successive gaps $g_k = \ell_{k+1}-\ell_k$; the unique large
discontinuity at $k^* = 4$ provides the data-driven basis for the
three-block partition.}
    \label{fig:market_dendrogram}
\end{figure}

\subsubsection{Parameter Estimation}
Let $\{C_1,\dots,C_K\}$ denote the partition obtained from the clustering
procedure, with block sizes $n_1,\dots,n_K$ after reordering indices if
necessary.
Conditional on this partition, we estimate the parameters
$\rho,\rho_1,\dots,\rho_K$ in~\eqref{eq.block.structure} by least-squares
projection of $\Omega^{\mathrm{data}}$ onto the structured family.
The following result shows that the optimal parameters are obtained by averaging
empirical correlations separately over cross-block and within-block entries; the
proof is deferred to Appendix~\ref{app:proof.prop.block.fit}.

\begin{restatable}{proposition}{propositionblockfit}\label{prop:block.fit}
Let $\Omega^{\mathrm{data}} \in \mathbb{R}^{n\times n}$ be an estimated
correlation matrix and fix a partition $\{C_1,\dots,C_K\}$ with block sizes
$n_1,\dots,n_K$. Define the cross-block index set
\[
S_C:=\{(p,q): p\in C_i,\; q\in C_j,\; i\neq j\}.
\]
Then the minimizer of
\[
\min_{\rho,\rho_1,\dots,\rho_K}
\|\Omega^{\mathrm{data}}-\Omega\|_F^2,
\]
where $\Omega$ is constrained to satisfy~\eqref{eq.block.structure},
is given by
\begin{equation}\label{eq.rho.fit}
\rho^*
=
\frac{1}{|S_C|}
\sum_{(p,q)\in S_C}\Omega^{\mathrm{data}}_{pq},
\end{equation}
and
\begin{equation}\label{eq.rhoi.fit}
\rho_i^*
=
\frac{
\frac{1}{n_i(n_i-1)}
\sum_{\substack{j,k\in C_i\\j\neq k}}
\Omega^{\mathrm{data}}_{jk}
-\rho^*
}{1-\rho^*},
\qquad i=1,\dots,K.
\end{equation}
\end{restatable}

Proposition~\ref{prop:block.fit} shows that the fitted structured matrix is
obtained by averaging empirical correlations at two scales: a global
cross-block level captured by $\rho^*$, and local within-block adjustments
captured by $\rho_i^*$.

\begin{remark}
Proposition~\ref{prop:block.fit} characterizes the unconstrained Frobenius
projection onto the structured family~\eqref{eq.block.structure}. If the
recovered parameters violate the positive-definiteness conditions of
Section~\ref{sec:structured_correlation}, one may instead project onto the
corresponding feasible parameter set.
\end{remark}

The resulting approximation provides a structured representation of any
correlation matrix, regardless of whether an exact block-diagonal structure is
present in the data.
As an illustration, Figure~\ref{fig:block_approx_example} shows the residual
correlation matrix $\hat\Omega$ of the market example introduced in
Section~\ref{sec:K.block.examples} (Figure~\ref{fig:cov_ex3r})
together with its block-structured approximation obtained using the proposed
clustering and parameter estimation procedure.
The approximation yields $\hat\rho = 0.162$ and, via the Proposition \ref{prop.K.block}, we obtain 
$\hat\theta = 2.389$, with per-block excess thresholds
$\bar\theta = (3.435,\,1.438,\,3.479)$ for the Credit, Energy, and Beverages
blocks respectively, giving fixed-point thresholds
\[
  \hat\theta\1+ \bar\theta \;=\; (5.824,\;3.827,\;5.868).
\]
Comparing these thresholds against the asset-level inverse volatilities $\theta_j$
directly confirms that the instability documented in Figure~\ref{fig:block_ex3r}
is a structural consequence of the block fixed-point mechanism of
Proposition~\ref{prop.K.block}.
Within the Credit block, Mastercard satisfies $\theta_{\mathrm{MA}} = 5.88$,
placing it just $1\%$ above its breakpoint $\hat\theta + \bar\theta_{\mathrm{Credit}} = 5.82$;
within Energy, Exxon falls at $\theta_{\mathrm{XOM}} = 3.78$, marginally
\emph{below} the threshold $\hat\theta + \bar\theta_{\mathrm{Energy}} = 3.83$,
formally excluding it from the active set of the approximated problem, while
Shell and Chevron clear the same threshold by only $9\%$.
By contrast, Visa and Pepsi, with $\theta_{\mathrm{V}} = 6.70$ and
$\theta_{\mathrm{PEP}} = 7.42$, lie well above their respective breakpoints;
Coca-Cola, at $\theta_{\mathrm{KO}} = 6.74$, also clears the Beverages threshold
of $5.87$ by a comfortable margin.
The block fixed-point system therefore makes a sharp prediction: assets operating
at or near their block thresholds (Mastercard and the Energy names) should
be the most susceptible to toggling as the estimated covariance fluctuates across rolling windows, while assets with coordinates safely above their breakpoints should remain persistently active.
Figure~\ref{fig:block_ex3r} confirms this prediction closely: Visa, Pepsi, and for the most part Coca-Cola remain active throughout, whereas Mastercard and the three Energy names generate the bulk of the observed turnover, establishing that the portfolio instability is not a statistical artefact but a direct manifestation
of the fragility regime characterized analytically in Proposition~\ref{prop.K.block}

\begin{figure}[htbp]
    \centering
    \includegraphics[width=\textwidth]{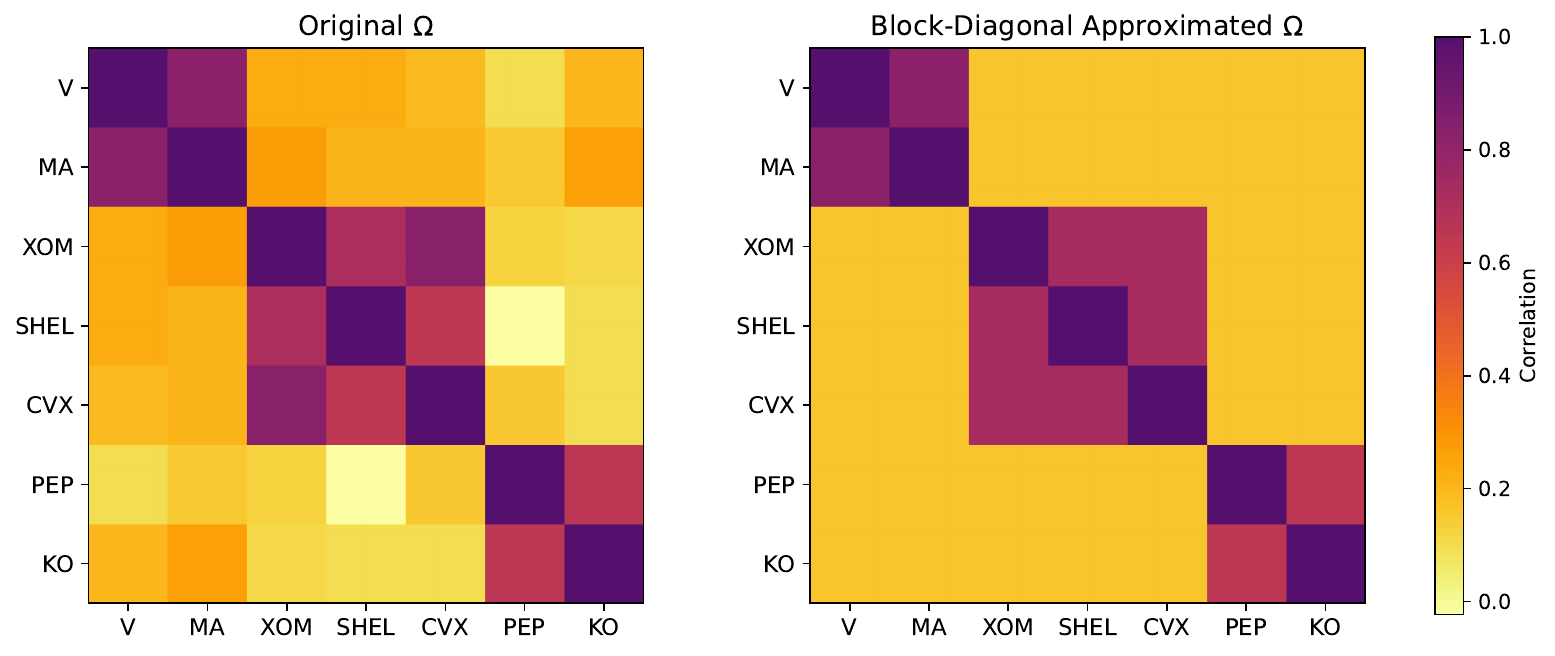}
    \caption{Residual correlation matrix $\Omega$ after removing the market factor from monthly returns (discussed in Section \ref{sec:K.block.examples} Figure~\ref{fig:cov_ex3r}), together with its block-structured approximation obtained using the proposed clustering procedure.}
    \label{fig:block_approx_example}
\end{figure}

\subsection{Block-Diagonal Shrinkage Methods}\label{sec:bds}

In this section, we introduce the proposed methods, which form the \emph{Block-Diagonal Shrinkage} (BDS) family. All three methods share the same estimation pipeline. First, the covariance matrix is transformed into a hierarchical representation using the single-linkage clustering and parameter estimation procedure described in Section~\ref{sec:block.approx}. This produces a block-diagonal approximation whose correlation structure is governed by only $K+1$ scalar parameters: $K$ within-block correlations and a single cross-block correlation.

This compression from $\tfrac{n(n-1)}{2}$ pairwise correlations to only
$K+1$ parameters constitutes a form of \emph{structural shrinkage}.
Rather than globally shrinking all covariance entries toward a target,
the approximation preserves the dominant dependence architecture while
removing much of the estimation noise that makes long-only
minimum-variance portfolios unstable. The resulting block representation
provides both statistical regularization and an explicit analytical
structure on which the portfolio construction procedures operate.

The three BDS variants differ only in how portfolio weights are computed
once the block structure has been estimated.

\begin{itemize}

\item \textbf{BDS-Obliv (correlation-oblivious block shrinkage).}
Applies the correlation-oblivious shrinkage procedure of Section~\ref{sec.block.correl}
to the block-diagonal approximation. Cross-block correlations are not used in
the allocation step, yielding a conservative and highly stable specification.

\item \textbf{BDS-Aware (correlation-aware block shrinkage).}
Applies the correlation-aware shrinkage procedure of Section~\ref{sec.block.correl}
to the same block structure. Compared to BDS-Obliv, it incorporates
within-block correlation information allowing sparse portfolios,
while still relying on the same block-level aggregation scheme.

\item \textbf{BDS-Direct (direct block-diagonal shrinkage).}
Solves the long-only minimum-variance problem directly on the
block-diagonal covariance using the fixed-point system of
Proposition~\ref{prop.K.block}, without additional threshold adjustment.
\end{itemize}

Conceptually, the BDS family separates {where} regularization is
introduced. BDS-Obliv and BDS-Aware combine structural shrinkage with an additional allocation-level stabilization step, whereas BDS-Direct
places all regularization at the covariance estimation stage and then
solves the resulting portfolio problem exactly. Section~\ref{sec:results}
shows empirically that the structural shrinkage alone in BDS-Direct is sufficient to achieve substantial improvements in stability while remaining close to the minimum-variance solution.

\section{Results}\label{sec:results}

We evaluate the block-diagonal shrinkage portfolio family against 24 benchmarks across three complementary data environments: a block-diagonal simulator that matches the theoretical setting in which the fragility mechanism was derived, a factor-based simulator that embeds no exogenous block structure (as in \citet{lopez2016building}) and serves as a structural counterexample, and real market data (following \citet{pedersen2021enhanced}). Consistent performance across all three environments is the primary evidence that the mechanism extends beyond any particular covariance model or asset universe.

\paragraph{Portfolio methods.}
We evaluate 27 methods: three classical benchmarks, 21 hierarchical variants, and three proposed methods. The classical benchmarks are the long-only Markowitz minimum-variance portfolio (LOM) \citep{markowitz1952}, the inverse-variance portfolio (IVP) \citep{clarke2006minimum}, and the equal-weight portfolio (EW) \citep{demiguel2009optimal}. The 21 hierarchical variants comprise three families: Hierarchical Risk Parity (HRP) \citep{lopez2016building}, Hierarchical Equal Risk Contribution (HERC) \citep{raffinot2018hierarchical}, and Hierarchical Clustering Asset Allocation (HCAA) \citep{raffinot2018hcaa}, each evaluated under seven agglomerative linkage criteria: single, complete, average, weighted, centroid, median, and Ward.
The proposed methods are the three members of the Block-Diagonal Shrinkage (BDS) family introduced in Section~\ref{sec:bds}: BDS-Obliv, BDS-Aware, and BDS-Direct. All three estimate block structure using the same single-linkage clustering procedure.

\paragraph{Performance metrics.}
The analysis focuses on two previously defined metrics (Section~\ref{sec:single_examples}): \emph{area}, the cumulative excess variance relative to the oracle minimum under the true covariance, and \emph{turnover}, the mean $L_1$ rebalancing magnitude. Because area is evaluated under the true covariance $V$, differences across methods isolate how estimation error propagates through portfolio construction over time; turnover captures the trading activity required to maintain the strategy.

Because IVP ignores cross-asset correlations, it rebalances only in response to changes in estimated variances, producing mechanically low turnover that is only weakly informative about portfolio stability. Reporting raw turnover alone therefore risks rewarding IVP for a property it obtains by construction. We therefore also report $\Delta\mathrm{Turn}$, the percentage increase in mean turnover relative to IVP within each panel, which interprets turnover as the incremental trading cost of exploiting dependence structure rather than as an absolute benchmark.

We additionally report \emph{annualized out-of-sample volatility} to verify that reductions in area translate into realized risk control. \emph{Active set} denotes the mean number of assets with positive weight, and the \emph{effective number of bets (ENB)}, computed as $1/\sum_i x_i^2$, measures allocation concentration, with larger values indicating more uniform weight distribution.

\paragraph{Backtest protocol.}
All experiments use a rolling-window backtest with $N = 250$ out-of-sample steps and daily rebalancing. In the simulated experiments (Sections~\ref{sec:bd_results} and~\ref{sec:lopez_results}), portfolio variance is evaluated under the known true covariance $V$ at each step, giving exact access to the oracle minimum; this quantity is unavailable in real market data.

All configurations use the sample covariance estimator and are averaged across 10 independent seeds. To assess sensitivity to estimation horizon $w$, the block-diagonal experiments evaluate both $w = 2n$ and $w = 3n$. Since rankings remain invariant across windows, subsequent experiments adopt $w = 2n$ as the reference specification. Appendix~\ref{sec:robustness} further evaluates sensitivity to cross-block correlation and eight alternative covariance estimators; BDS-Aware remains on the Pareto frontier throughout the correlation sweep, and method rankings are invariant across all estimator families.

The real-data protocol differs in window length and evaluation horizon and follows \citet{pedersen2021enhanced}; details are provided in Section~\ref{sec:real_data}. Computational details, including solver, hardware, and open-source repository are reported in Appendix~\ref{sec:appendix_results}.

\subsection{Block-Diagonal Covariance Structure}\label{sec:bd_results}

We begin with a covariance structure that matches the setting in which the fragility mechanism was derived. A universe of $n$ assets is partitioned into $K \sim \mathrm{Uniform}(\lfloor n/9\rfloor,\lfloor n/5\rfloor)$ blocks. Within-block correlations $\rho_i$ are drawn independently from $\mathrm{Uniform}(0.5,0.9)$ and cross-block dependence is controlled by $\rho = 0.1$, which is varied later in Appendix~\ref{sec:robustness}.

Two volatility regimes isolate the mechanism of interest. In the homogeneous regime, each block is assigned a single volatility level so all assets within a block share the same marginal risk; we vary the between-block spread $\sigma_s \in \{1.2, 1.5, 2.0\}$. In the heterogeneous regime, volatilities increase linearly from $0.5$ to $3.0$ within each block, creating systematic risk dispersion inside clusters.

The heterogeneous regime is the primary setting. It generates concentrated oracle portfolios and amplifies composition instability, which is precisely the behavior studied throughout this manuscript. The homogeneous regime serves as a control, verifying solutions do not under-perform in a dense-solution setting. Figure~\ref{fig:bd_dgp} illustrates the covariance structure across the three volatility conditions.

\begin{figure}[htbp]
\centering
\includegraphics[width=\textwidth]{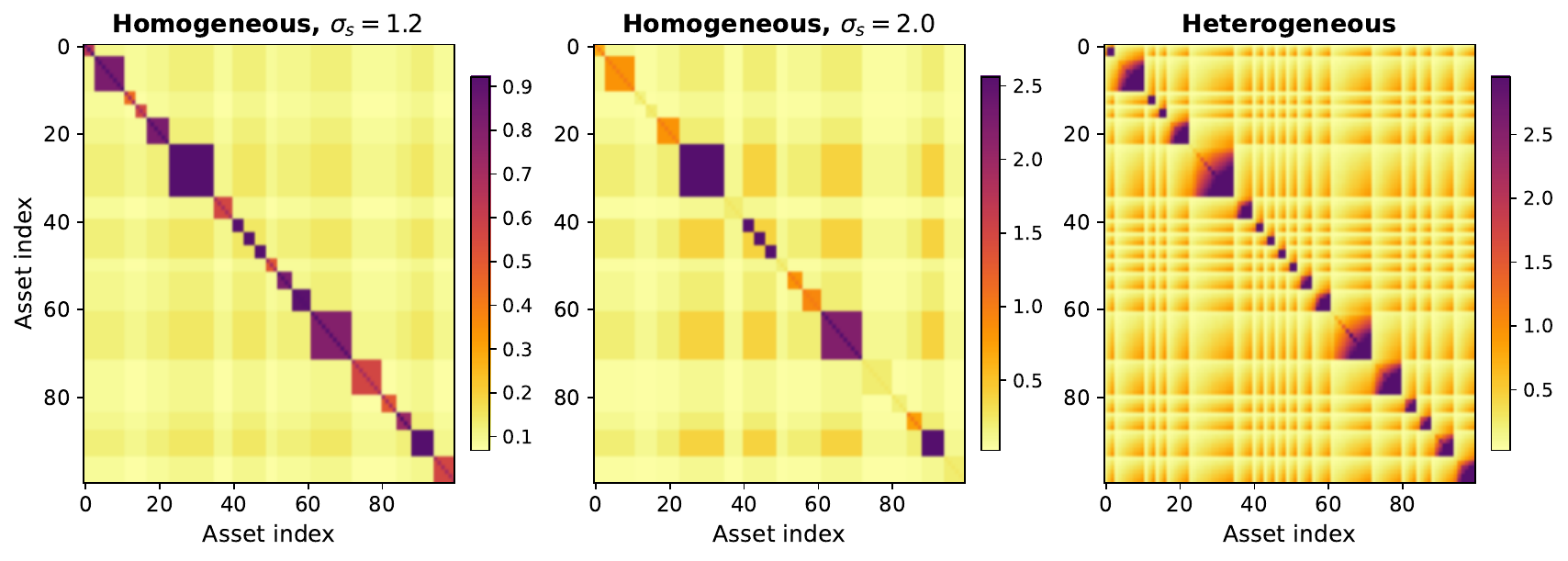}
\caption{Covariance matrices generated by the block-diagonal data-generating process at $n=100$ and $\rho=0.1$. All panels share the same block partition and within-block correlation draws; only volatility assignments differ: homogeneous with mild spread ($\sigma_s=1.2$), homogeneous with severe spread ($\sigma_s=2.0$), and heterogeneous with linearly increasing within-block volatilities. As volatility dispersion increases, concentrated low-risk regions emerge and the covariance structure becomes progressively less uniform.}
\label{fig:bd_dgp}
\end{figure}

\subsubsection{Pareto Dominance and Composition Stability}\label{sec:bd_main}

We evaluate all methods through the area--turnover Pareto frontier, where lower values on both axes are jointly preferred. Each point reports the mean over seeds and the frontier connects the non-dominated methods. All Pareto plots in this paper use logarithmic scales on both axes; consequently, visual separation reflects multiplicative differences, and even modest distances can correspond to order-of-magnitude improvements.

Figure~\ref{fig:e1_pareto_pooled} reports results averaged over $n \in \{50,\ldots,250\}$ and $w \in \{2n,3n\}$, separated by volatility regime.
BDS methods occupy the Pareto frontier throughout both  heterogeneous and homogeneous regimes, while every non-BDS benchmark is strictly dominated in at least one metric. The equal-weighted portfolio (EW) is omitted from the scatter because its zero turnover is trivial by construction; Table~\ref{tab:metrics_n100} confirms that this comes at substantially higher area.

In the heterogeneous setting, BDS-Direct and BDS-Aware define the dominant frontier, while BDS-Obliv occupies the low-turnover extreme by suppressing cross-block interaction at the cost of increased area. In the homogeneous regime, their difference in terms of turnover increases, but the ordering remains unchanged. Rank heatmaps across all $(n,\mathrm{regime})$ combinations are reported in Supplementary Material~\ref{sec:appendix_rank_e1_cal}.

\begin{figure}[htbp]
\centering
\includegraphics[width=\textwidth]{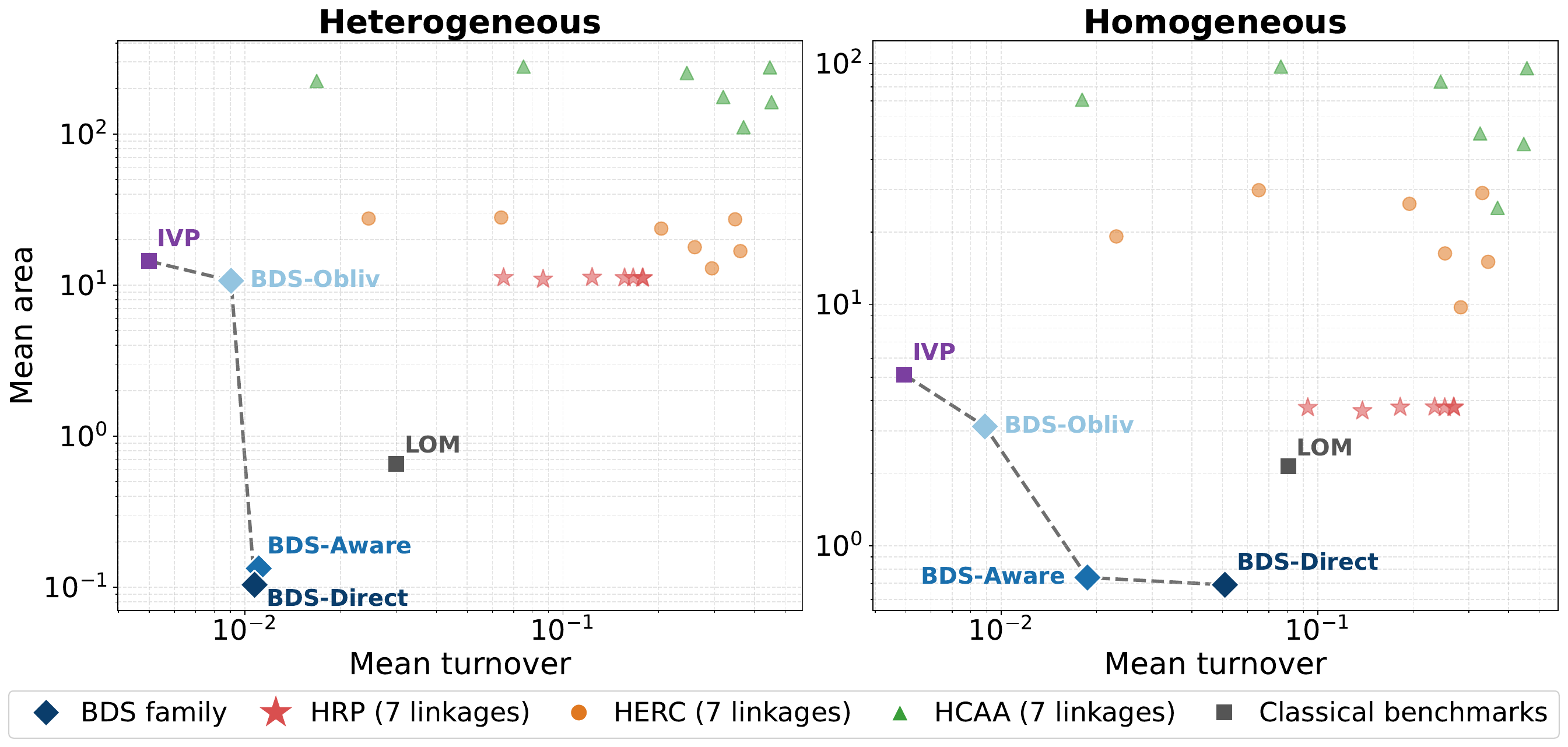}
\caption{Area--turnover Pareto scatter pooled over $n \in \{50,\ldots,250\}$ and $w \in \{2n,3n\}$ for heterogeneous and homogeneous regimes. Both axes use logarithmic scale. Each point is averaged over seeds, universe sizes, and window lengths within each regime. The per-$n$ breakdown appears in Supplementary Material~\ref{sec:appendix_pareto_by_n}.}
\label{fig:e1_pareto_pooled}
\end{figure}

Among non-BDS methods, long-only Markowitz is consistently the closest benchmark to the frontier, whereas hierarchical methods remain strictly interior across regimes. HRP is the strongest among them, suggesting that clustering alone provides some protection against estimation error. However, hierarchical diversification never closes the gap to methods that optimize directly over recovered structure.

Other classical methods that create fully dense, namely IVP and EW, achieve low turnover but incur substantially higher area, indicating that diversification alone does not resolve estimation-induced instability.
By contrast, BDS-Direct exploits covariance structure directly through block-level optimization. Its advantage depends on recoverable structure: gains are largest when block separation is clear and diminish as dependence becomes more diffuse. This sensitivity is not a weakness but the intended behavior of a structure-exploiting method.

Figure~\ref{fig:e1_fragility} quantifies performance relative to LOM through
$\log_{10}(\mathrm{area}_{\mathrm{method}}/\mathrm{area}_{\mathrm{LOM}})$: the zero line corresponds to LOM, negative values indicate improvement, and each unit on the y-axis is a tenfold change in relative area.
BDS-Direct falls monotonically from roughly $-0.6$ at $n=50$ to nearly $-1.0$ at $n=250$ in the heterogeneous regime, consistent with better recovery of latent block structure as dimensionality increases.
Hierarchical methods show the opposite behavior, trending upward and diverging into the red zone.
The double-headed arrow annotated at $n=250$ in each panel reports the fold-gap between BDS-Direct and the best performing benchmark directly, so the separation is legible without back-transforming the log scale; this gap grows with both universe size and volatility dispersion, reaching roughly $30\times$ at $\sigma_s = 2.0$ and remaining substantial in the heterogeneous regime.
Results remain stable across estimation windows (Supplementary Material~\ref{sec:appendix_tn_robustness}).

\begin{figure}[htbp]
\centering
\includegraphics[width=0.9\textwidth]{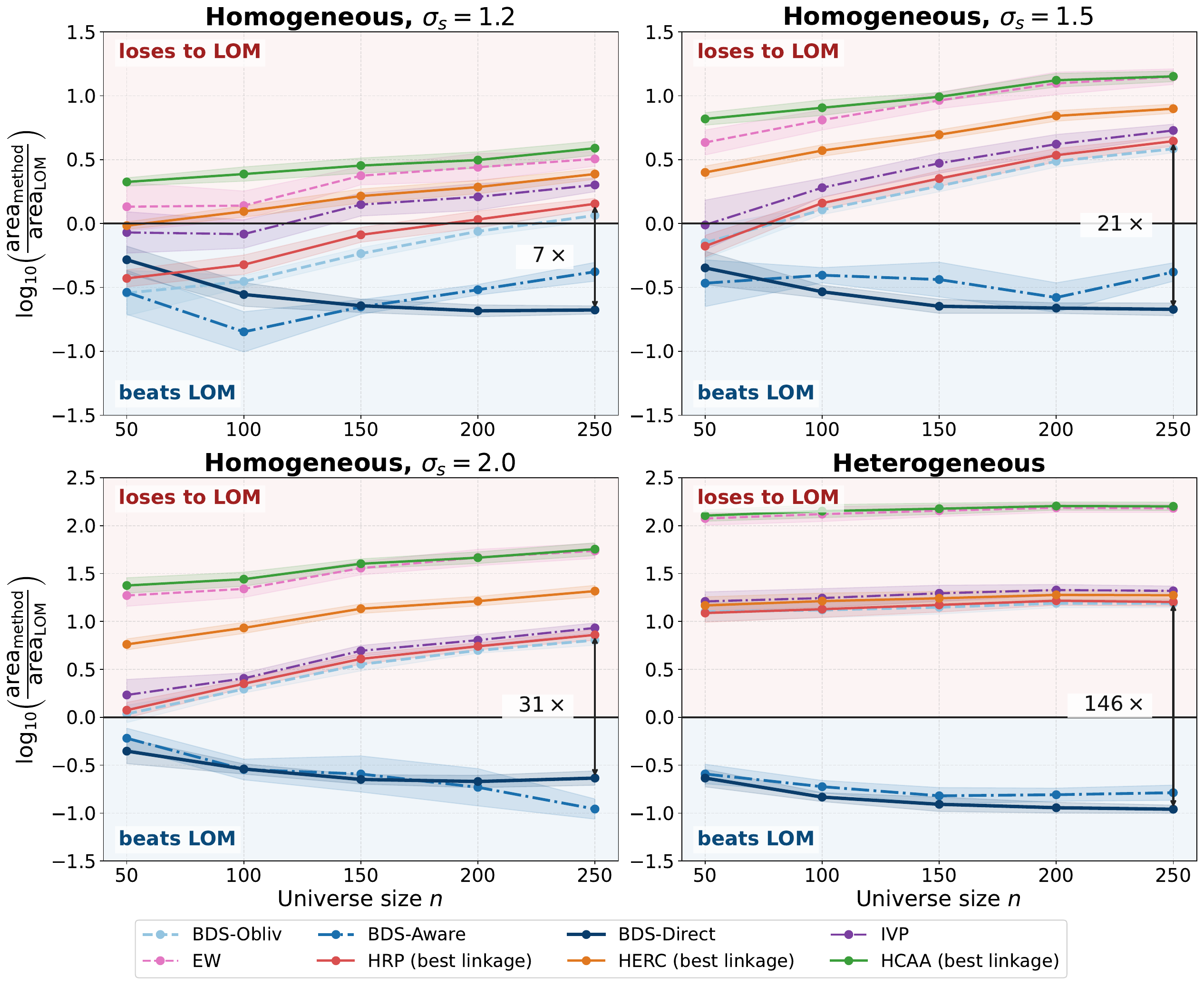}
\caption{$\log_{10}(\mathrm{area}{\mathrm{method}}/\mathrm{area}{\mathrm{LOM}})$ against universe size $n \in \{50,\ldots,250\}$ at $w=2n$. LOM serves as the strongest non-BDS baseline. Shaded bands show 95\% confidence intervals over 10 seeds. The annotation at $n=250$ reports the fold-gap in area between BDS-Direct and HRP. HRP, HERC, and HCAA are shown at the linkage with lowest mean area at $n=100$ (Ward in all three cases).}
\label{fig:e1_fragility}
\end{figure}

Table~\ref{tab:metrics_n100} isolates the source of these gains. In the heterogeneous regime, LOM and BDS-Direct produce portfolios of comparable size, yet their temporal behavior differs sharply: LOM repeatedly reconfigures its active set across rebalancing periods, whereas BDS-Direct maintains a substantially more stable composition. The reduction in area therefore arises primarily from composition stability rather than sparsity, and persists across window lengths (Supplementary Material~\ref{sec:appendix_tn_robustness}).

IVP and EW illustrate the opposite extreme. Both maintain broad support and achieve low turnover, yet incur much larger area, showing that stability without structural alignment is insufficient. Among hierarchical methods, HRP delivers the strongest variance--turnover compromise but remains far from the BDS frontier.

\begin{table}[htbp]
\centering
\caption{Mean metrics at $n=100$, $w=2n$, averaged over 10 seeds. Standard deviations omitted for readability. The homogeneous regime uses $\sigma_s=1.5$ for illustration; rankings are consistent across volatility spreads. HRP, HERC, and HCAA use their best linkage at $n=100$ (Ward in all cases). Bold marks the lowest area and volatility\ within each regime and the lowest turnover\ and $\Delta$Turn. EW turnover\ and $\Delta$Turn are omitted because zero rebalancing is trivial by construction.}
\label{tab:metrics_n100}
\resizebox{\textwidth}{!}{%
\begin{tabular}{l ccccc ccccc}
\toprule
 & \multicolumn{5}{c}{\textbf{Heterogeneous}} & \multicolumn{5}{c}{\textbf{Homogeneous} ($\sigma_s=1.5$)} \\
\cmidrule(lr){2-6}\cmidrule(lr){7-11}
Method & Area & Vol. & Turn. & $\Delta$Turn & Act./ENB & Area & Vol. & Turn. & $\Delta$Turn & Act./ENB \\
\midrule
BDS-Direct & \textbf{0.12} & \textbf{3.11} & 0.013 & +86\% & 17.6/16.9 & \textbf{0.74} & \textbf{4.68} & 0.066 & +1000\% & 47.4/26.8 \\
BDS-Aware  & 0.16 & \textbf{3.11} & 0.014 & +100\% & 18.5/17.1 & 1.00 & 4.70 & 0.030 & +400\% & 92.8/48.1 \\
BDS-Obliv  & 10.69 & 4.52 & 0.009 & \textbf{+29\%} & 100/25.4 & 3.19 & 4.94 & 0.012 & \textbf{+100\%} & 100/62.5 \\
\midrule
LOM \citep{markowitz1952}       & 0.83 & 3.20 & 0.037 & +429\% & 17.2/14.0 & 2.51 & 4.86 & 0.098 & +1533\% & 30.5/17.3 \\
IVP \citep{clarke2006minimum}        & 14.37 & 4.89 & \textbf{0.007} & --- & 100/34.6 & 4.91 & 5.09 & \textbf{0.006} & --- & 100/74.4 \\
EW \citep{demiguel2009optimal}         & 107.02 & 10.82 & --- & --- & 100/100 & 16.72 & 6.13 & --- & --- & 100/100 \\
HRP \citep{lopez2016building}        & 10.92 & 4.53 & 0.093 & +1229\% & 100/28.0 & 3.63 & 4.95 & 0.153 & +2450\% & 100/65.3 \\
HERC \citep{raffinot2018hierarchical}       & 13.27 & 4.77 & 0.255 & +3543\% & 100/24.0 & 9.37 & 5.56 & 0.254 & +4133\% & 100/66.3 \\
HCAA \citep{raffinot2018hcaa}       & 114.79 & 11.13 & 0.334 & +4671\% & 100/61.3 & 20.48 & 6.49 & 0.347 & +5683\% & 100/61.0 \\
\bottomrule
\end{tabular}}
\end{table}

The three BDS variants reveal different mechanisms for controlling instability. BDS-Obliv gains stability by suppressing cross-block interaction, which minimizes turnover but sacrifices variance performance under heterogeneous dependence. BDS-Aware and BDS-Direct preserve cross-block structure, with BDS-Direct consistently providing the closest approximation to the oracle solution when block recovery is accurate. Across variants, the dominant effect is stabilization of composition rather than increased diversification or sparsity.

In the homogeneous regime, lower volatility dispersion weakens active-set dynamics and produces denser optimal portfolios, reducing absolute performance differences. Nevertheless, rankings remain unchanged: BDS-Direct retains the lowest area, while BDS-Aware offers a lower-turnover alternative with only modest degradation in area. Overall, block-structured optimization remains effective even when sparse oracle structure is less pronounced.

Appendix~\ref{sec:robustness} reports robustness to cross-block correlation and covariance estimator choice; BDS-Aware remains on the Pareto frontier throughout the $\rho$ sweep and rankings are invariant across all estimator families.

\subsection{Factor Model Covariance Structure}\label{sec:lopez_results}

The data-generating process introduced in \citet{lopez2016building} constructs a covariance matrix from a latent factor model. A subset of $c$ assets loads on shared factors with loading noise $\eta$, a complementary subset of $u$ assets is purely idiosyncratic, and a global shock of magnitude $\delta$ affects all assets. No block structure is imposed exogenously. Any clustering arises endogenously from shared factor exposure rather than from designed partitions. As a result, this setting serves as a misspecified but structurally realistic stress test for methods that rely on latent grouping.

Despite the absence of explicit blocks, Figure~\ref{fig:lopez_clustering} shows that single-linkage clustering applied to even the mildest configuration ($n=60$, 25\% factor-loaded assets) recovers localized correlation clusters along the diagonal. These clusters correspond to subsets of assets sharing latent factor exposure. This induced structure is weak in magnitude but sufficiently coherent for block-based methods to exploit.

We evaluate six configurations $(u, c) \in \{(10,5), (20,10), (30,30), (40,20), (45,15), (67,33)\}$ under default parameters $\eta = 0.25$ and $\delta = 0.20$.

\begin{figure}[htbp]
\centering
\includegraphics[width=\textwidth]{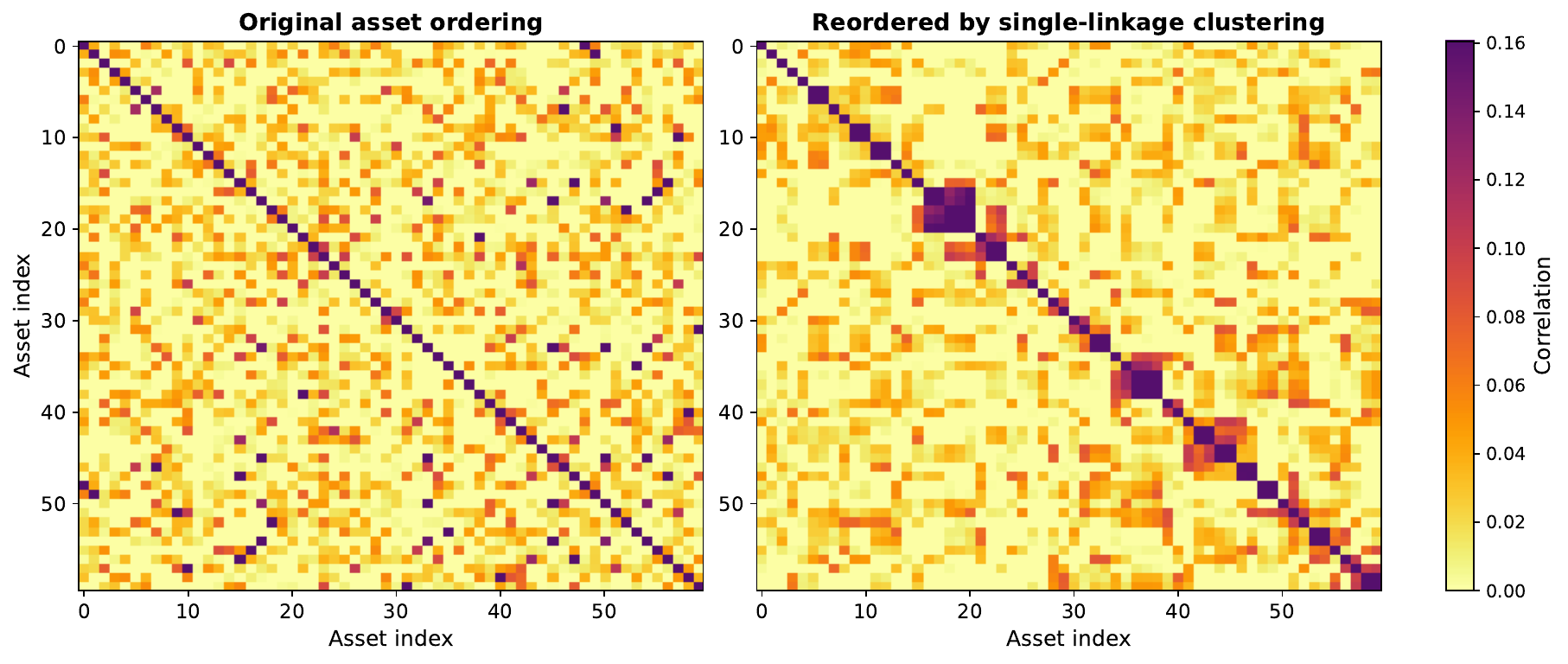}
\caption{Correlation matrix of a \citet{lopez2016building} data-generating process instance at the mildest configuration ($n=60$, 45 idiosyncratic and 15 factor-loaded assets). The original ordering shows no structural pattern. After single-linkage clustering, localized correlation blocks emerge along the diagonal, reflecting shared factor exposure.}
\label{fig:lopez_clustering}
\end{figure}

\subsubsection{Aggregate Performance}\label{sec:lopez_aggregate}

This factor model is a canonical benchmark for hierarchical risk parity work \citep{lopez2016building}, which shows that HRP improves over minimum-variance and inverse-variance portfolios in this setting. The baseline is replicated here: HRP-ward consistently dominates LOM and IVP across all configurations, confirming the expected behavior of hierarchical allocation under latent factor structure.

\begin{figure}[t]
\centering
\includegraphics[width=\textwidth]{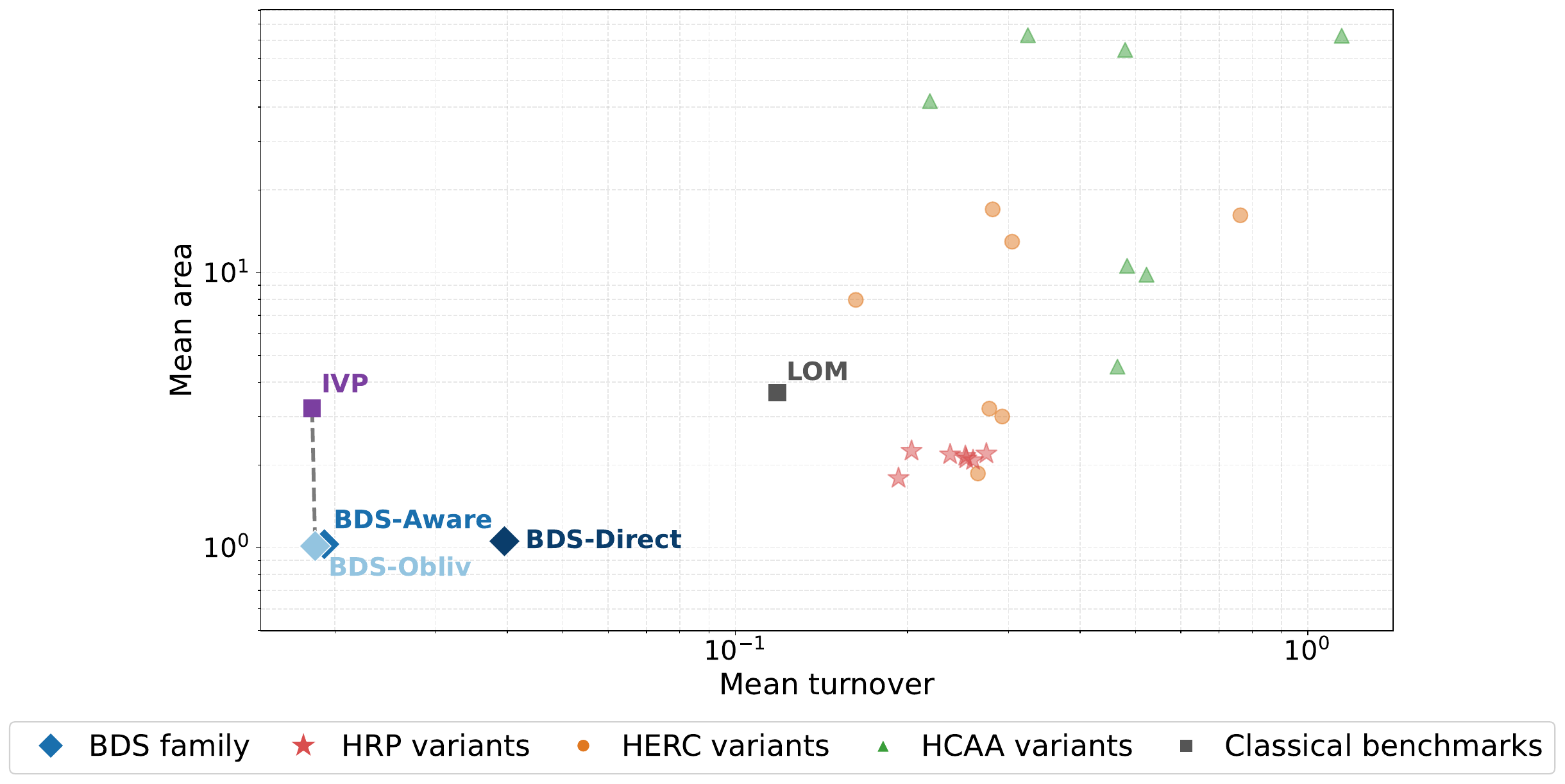}
\caption{Area–turnover Pareto scatter pooled over all six Lopez de Prado configurations and 10 seeds. BDS-Obliv anchors the low-turnover region, while BDS-Direct achieves the lowest area in the remaining configurations. HRP improves over classical benchmarks but remains interior to the BDS frontier across all cases. Per-configuration results are reported in Supplementary Material~\ref{sec:appendix_e2_byconfig}.}
\label{fig:e2_pareto}
\end{figure}

All BDS variants outperform this benchmark class across the full experimental grid. BDS-Obliv achieves the lowest area in four configurations, while BDS-Direct dominates in the remaining two. BDS-Aware remains consistently close to the frontier across all cases. In terms of turnover, all three BDS variants dominate HRP and LOM by a large margin. Across configurations, the BDS-to-LOM area ratio ranges from $3.0\times$ to $4.3\times$, with turnover reductions between $5\times$ and $11\times$ as reported in Table~\ref{tab:e2_metrics} in Supplementary Material~\ref{sec:appendix_e2_byconfig}.

These gains are notable because the model does not contain an explicit block structure. Instead, latent factor exposure induces partial clustering in the sample covariance, which is sufficient for block-based constructions to extract exploitable dependence. The results therefore indicate that block-diagonal portfolio construction is effective whenever the covariance exhibits sufficient clusterability, even when that structure is not explicitly present in the model.

Appendix~\ref{sec:lopez_sensitivity} reports sensitivity to factor-loading noise $\eta$: BDS-Aware and BDS-Obliv lead at low noise, while BDS-Direct becomes the dominant variant at high noise, reflecting a bias--variance trade-off in structural enforcement; the BDS family remains Pareto-stable throughout.

\subsection{Real-World Evidence}\label{sec:real_data}

The simulation evidence shows a consistent variance–turnover advantage for BDS methods across structurally distinct data-generating processes. A natural concern is that this behavior is an artifact of controlled environments where the covariance structure is known and evaluation is benchmarked against an oracle. In real markets, neither condition holds: the covariance must be estimated and the underlying structure is unobserved. This section tests whether the same Pareto dominance persists under these conditions.

\paragraph{Data and protocol.}
We use the Fama-French 49 industry portfolios \citep{fama1993common} as in \citet{pedersen2021enhanced}. The sample begins in July 1969, the first date with complete coverage, and the 12-month XSMOM construction removes the initial year, yielding an evaluation period from July 1970 to February 2026 ($T = 668$, $n = 49$). The empirical design follows \citet{pedersen2021enhanced} in using rolling covariance estimation and out-of-sample portfolio evaluation with signal-scaled returns.

Their baseline specification uses a 60-month window with 5\% correlation shrinkage, corresponding to $T/n \approx 1.22$. Our windows $w \in \{2n, 3n\}$ therefore operate in a comparatively richer estimation regime. Ledoit-Wolf shrinkage is used as the primary estimator since it adapts shrinkage intensity within each window and matches the risk-modeling setup in \citet{pedersen2021enhanced}. Appendix~\ref{sec:rd_estimator} reports a sample covariance robustness check showing that all conclusions are unchanged without shrinkage.

Raw industry returns are dominated by a common market component, with mean pairwise correlation of 0.558 across all pairs. In this setting, long-only Markowitz benefits mechanically from concentration into low-beta industries, which reflects market exposure rather than covariance estimation quality. To isolate cross-sectional structure, we remove the market factor via OLS and construct cross-sectional momentum (XSMOM) returns following the procedure in \citet{pedersen2021enhanced, jegadeesh1993returns, asness2013value, moskowitz2012time}. The resulting long and short legs are scaled to unit exposure. This reduces average pairwise correlation from 0.558 to 0.056, removing the dominant market mode so that performance differences reflect residual dependence structure rather than directional exposure.

\paragraph{Metrics.}
In real data the oracle covariance is unavailable, so the simulation area metric cannot be computed. We therefore report realized annualized out-of-sample volatility and mean turnover. Lower realized volatility corresponds to the main objective of the portfolio construction in practice, while turnover measures stability of portfolio composition. These two quantities jointly proxy the same variance–stability trade-off captured by area in simulation.

\subsubsection{Pareto Structure}\label{sec:rd_pareto}

Figure~\ref{fig:rd_pareto} shows the realized volatility–turnover frontier for real data, mirroring Figures~\ref{fig:e1_pareto_pooled} and~\ref{fig:e2_pareto}.
At $w = 2n$, the frontier consists of IVP, BDS-Obliv, and BDS-Direct. At $w = 3n$, BDS-Aware enters the frontier while BDS-Direct remains the lowest-volatility method among non-IVP strategies. No hierarchical method lies on the frontier at either window. This mirrors the heterogeneous simulation regime in Figure~\ref{fig:e1_pareto_pooled}, where the same three BDS variants anchor performance while all hierarchical constructions remain strictly interior.

LOM is dominated at both windows. At $w = 2n$, BDS-Direct achieves 9.1\% lower realized volatility and 37\% lower turnover simultaneously. This rules out a variance–rebalancing trade-off explanation: LOM is strictly inferior in both dimensions.
Hierarchical methods reduce estimation variance through recursive aggregation but do so at a steep turnover cost. Even the strongest specification across all linkage choices, HRP-median, remains dominated by BDS-Direct at $w=2n$ with higher volatility and substantially higher turnover. This aligns with the simulation results in Section~\ref{sec:lopez_aggregate}, where HRP improves on classical benchmarks but remains interior to the BDS frontier. The implication is consistent across datasets: recursive balancing stabilizes estimation, but does not recover the structure exploited by block-based construction once market exposure is removed.

Within the BDS family, the ordering is stable. BDS-Obliv achieves the lowest turnover in the class by suppressing cross-block interaction, placing it closest to IVP on the turnover axis but at higher volatility. BDS-Direct anchors the low-volatility frontier at short windows, while BDS-Aware becomes competitive at longer windows where estimation noise is reduced and cross-block dependence can be more reliably exploited. This pattern is consistent with the noise sensitivity behavior observed in Appendix~\ref{sec:robustness}.

\begin{figure}[htbp]
    \centering
    \includegraphics[width=\textwidth]{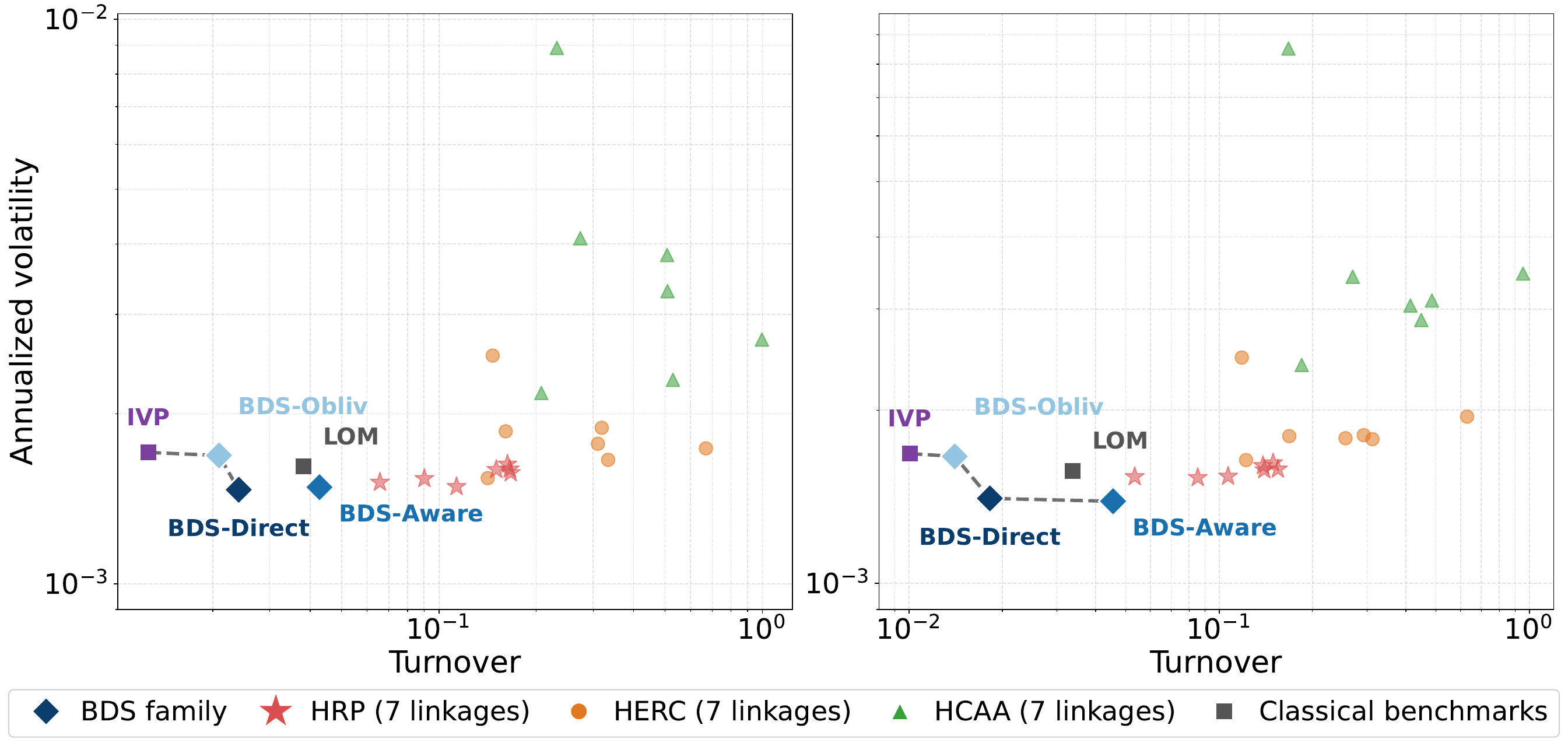}
    \caption{Realized annualized volatility vs.\ mean monthly turnover for 26 portfolio methods on Fama-French 49 industry XSMOM returns using the Ledoit-Wolf estimator, across estimation windows $w \in \{2n, 3n\}$ (July 1970--February 2026). All Pareto frontiers are shown on log-log axes, so distances correspond to multiplicative improvements. The frontier consists only of BDS methods and IVP at both windows; all hierarchical methods and LOM are strictly dominated in both metrics.}
    \label{fig:rd_pareto}
\end{figure}

The aggregate frontier compresses time variation into a single summary point per method. Figure~\ref{fig:rd_rolling} shows that the advantage is persistent across regimes rather than episodic. At $w = 2n$, BDS-Direct maintains lower rolling volatility than LOM in 84.8\% of months across the full sample period after initialization, including major crises such as 2001–02 and 2008–09. HRP tracks LOM closely in volatility but at consistently higher turnover, reinforcing that its gains are not efficiency improvements but alternative risk allocation behavior. Appendix~\ref{sec:rd_tn} reports results across both estimation windows: the volatility advantage of BDS-Direct over LOM grows from 9.1\% at $w=2n$ to 10.4\% at $w=3n$, and rankings remain unchanged.

\begin{figure}[htbp]
    \centering
    \includegraphics[width=\textwidth]{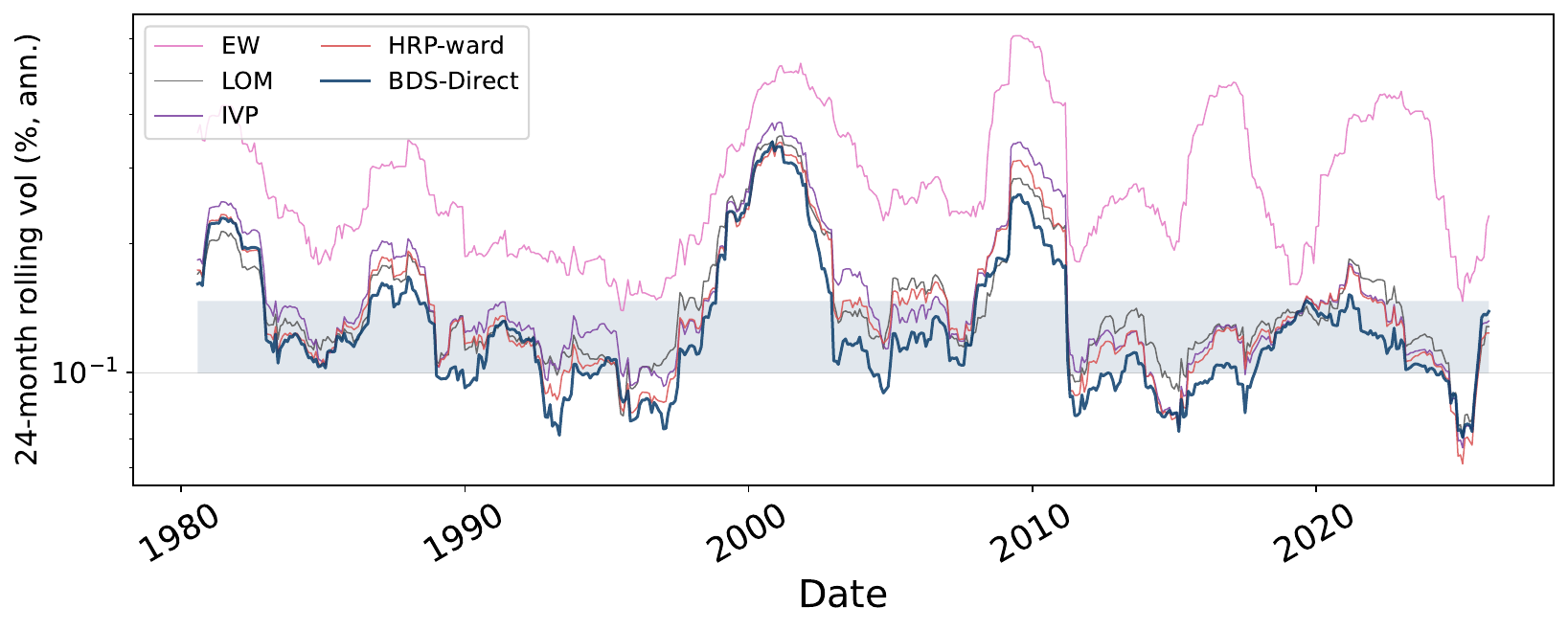}
    \caption{24-month rolling realized volatility for selected methods on Fama-French 49 industry XSMOM returns using Ledoit-Wolf estimation with $w=2n$. Best hierarchical method used (HRP) Ward linkage for consistency with simulated experiments. Best BDS method (BDS-Direct) is below LOM in 84.8\% of months respectively over 1970–2026.}
    \label{fig:rd_rolling}
\end{figure}

Taken together, the real-data results replicate the simulation conclusions: once the dominant market factor is removed, block-structured portfolio construction delivers a consistently superior volatility–turnover frontier, and BDS methods remain Pareto dominant across all specifications examined.

\section{Conclusions}\label{sec:conclusions}

This paper develops a structural explanation for instability in long-only minimum-variance portfolio selection by exploiting closed-form characterizations under block-structured correlation regimes. Building on analytically tractable solutions for constant and block-diagonal correlation matrices, we show that portfolio fragility is driven not only by estimation error in the covariance matrix but, more fundamentally, by threshold effects induced by the interaction between correlation structure and the positive-part operator in the optimal weights. These insights lead to a family of mechanism-driven robustification strategies that act directly on the correlation structure. Rather than modifying the optimization problem, we introduce structured shrinkage procedures that selectively attenuate unstable coupling patterns while preserving the dominant dependence structure. The resulting methods operate at different levels of granularity (from global correlation shrinkage to block-level and cluster-aware adjustments) and are designed to stabilize the active set implied by the closed-form solution.

Across numerical experiments, behavior is consistent and strongly regime-dependent. In homogeneous-volatility settings, all heuristic methods improve upon classical Markowitz, with inverse-correlation weighting among the strongest baselines. The correlation-oblivious shrinkage achieves the best overall performance, delivering the lowest realized risk and optimality loss while preserving full diversification. The correlation-aware variant provides no additional benefit in this regime, consistent with its design for volatility dispersion rather than uniform risk settings.

In heterogeneous-volatility regimes the ordering changes materially. Standard heuristics are uniformly dominated by Markowitz in both risk and objective value, reflecting limited adaptability to dispersion in marginal risks. The correlation-aware shrinkage becomes dominant, achieving the best trade-off between realized volatility, optimality loss, and sparsity. The correlation-oblivious variant remains competitive but is less responsive to volatility dispersion, while still improving over classical heuristics. Across regimes, volatility heterogeneity is the primary driver of instability, and explicitly incorporating it is necessary for robust performance.

These patterns extend directly to the portfolio construction results in Section~\ref{sec:results}. After removing the dominant market factor, all three block-diagonal shrinkage methods lie on the empirical volatility–turnover frontier across estimation windows, covariance estimators, and dataset specifications. This dominance is structural rather than episodic: it persists across time, estimator choice, and universe size, and is driven by stabilization of the active set rather than sparsity alone. In contrast, hierarchical methods and classical benchmarks remain strictly interior to the frontier in all specifications, indicating that recursive diversification without explicit exploitation of cross-sectional dependence is insufficient once the market mode is removed. Across all experiments, gains are simultaneous in volatility and turnover, ruling out simple trade-offs between risk reduction and trading intensity.

Taken together, the evidence supports a single mechanism: instability in long-only minimum-variance portfolios is primarily a compositional phenomenon induced by structured dependence, not solely an estimation problem. The proposed shrinkage procedures are effective because they stabilize the active set implied by the closed-form solution under block structure, yielding consistent improvements across regimes, estimators, and asset universes.

Future work will extend these ideas in several directions. One natural avenue is to develop robustification procedures directly from Proposition~\ref{prop.K.block}, rather than through sequential decomposition. This remains challenging because the block-diagonal characterization is defined implicitly through a coupled fixed-point system in which all thresholds depend jointly on global and block-level correlation structure. As a result, perturbations intended to stabilize one region of the solution can propagate throughout the system, making globally consistent robustification nontrivial.
A particularly interesting direction is to move beyond minimum-variance portfolios and study whether the mechanisms identified here extend to mean--variance formulations. While this work isolates instability arising from covariance structure alone, practical portfolio construction also depends on expected return estimates, which are themselves a major source of estimation error and instability. Understanding how correlation-driven threshold effects interact with return uncertainty may provide a more complete structural explanation of portfolio fragility and motivate new robustification strategies for joint mean--variance estimation.

\clearpage
\begin{appendices}
\section{Proofs}

\subsection{Proof of Lemma \ref{lemma:correlation.reformulation}}
\label{app:lemma.correlation.reformulation}

\lemmacorrelationreformulation*

\begin{proof}
By using the change of variables $z := \Sigma x$,  problem~\eqref{eq.minvar} can be  equivalently stated as
\begin{equation}\label{eq.z.form}
\begin{aligned}
\min_{z\in\mathbb{R}^n}\quad & \frac12 z^\top \Omega z \\
\text{s.t.}\quad & \theta^\top z = 1,\\
& z\ge 0.
\end{aligned}
\end{equation}
The KKT conditions for~\eqref{eq.z.form} are
\begin{subequations}\label{eq.kkt.z}
\begin{align}
\Omega z &= \lambda \theta + \mu, \label{eq.kkt.z.a}\\
\theta^\top z &= 1, \\
z &\ge 0,\quad \mu\ge 0,\quad \mu^\top z = 0, \label{eq.kkt.z.b}
\end{align}
\end{subequations}
where $\lambda\in\mathbb{R}$ is the multiplier associated with the equality
constraint.

Multiplying~\eqref{eq.kkt.z.a} by $z^\top$ and using complementary slackness
gives
\[
z^\top \Omega z
=
\lambda\, \theta^\top z + z^\top \mu
=
\lambda.
\]
Since $\Omega$ is positive definite and $z\neq 0$ (because
$\theta^\top z =1$), it follows that
$
\lambda = z^\top \Omega z >0.
$
Hence dividing~\eqref{eq.kkt.z.a} by $\lambda$ and letting $y :=z/\lambda$ and $\tilde \mu := \mu/\lambda$ we get
\[
\Omega y = \theta + \tilde\mu,
\]
and also
\[
y\ge 0,\qquad \tilde\mu\ge 0,\qquad \tilde\mu^\top y
=
\frac{1}{\lambda^2}\mu^\top z
=
0.
\]
Thus $(y,\tilde\mu)$ satisfies
\begin{subequations}\label{eq.kkt.y}
\begin{align}
\Omega y &= \theta + \tilde\mu, \\
y &\ge 0,\quad \tilde\mu\ge 0,\quad \tilde\mu^\top y = 0,
\end{align}
\end{subequations}
which are precisely the KKT conditions of~\eqref{eq.min.correl}. Therefore,
$y$ solves~\eqref{eq.min.correl}.

Finally, since $z = \Sigma x$ and $z = \lambda y$ it follows that 
$x=\Theta z$ and $1 = \theta\transp z = \lambda\theta^\top y = \lambda \1^\top \Theta y$.  Therefore,
\[
x=\Theta z = \lambda \Theta y =  \frac{\Theta y}{\1^\top \Theta y},
\]
which establishes~\eqref{eq.sol.x}.
\end{proof}

\subsection{Proof of Proposition \ref{prop.one.block}} \label{app:proof.prop.one.block}

\propositiononeblock*

\begin{proof}
Substituting $\Omega = (1-\rho)I + \rho \1\1\transp$ into the optimality condition~\eqref{eq.kkt_a} gives
\begin{equation}\label{eq.kkt_expanded}
(1-\rho)y + \rho (\1\transp y)\1 = \theta + \mu.
\end{equation}
Define
\begin{equation}\label{eq.theta.bar.single}
\bar\theta := \rho (\1\transp y).
\end{equation}
Then~\eqref{eq.kkt_expanded} becomes
\begin{equation}\label{eq.balance}
(1-\rho)y = \theta - \bar\theta \1 + \mu \Leftrightarrow (1-\rho)y-\mu = \theta - \bar\theta \1.
\end{equation}
Thus the complementarity conditions~\eqref{eq.kkt_b}: $y \ge 0$, $\mu \ge 0$, and $\mu_i y_i = 0$ for all $i$ imply that 
\[
(1-\rho)y = (\theta - \bar\theta \1)^+, 
\quad
\mu = (\bar\theta \1 - \theta)^+,
\]
and hence
\begin{equation}\label{eq.y.final}
y = \frac{1}{1-\rho}(\theta - \bar\theta \1)^+.
\end{equation}
Finally, substituting~\eqref{eq.y.final} into the definition~\eqref{eq.theta.bar.single} yields
\begin{align*}
\bar\theta 
= \rho \, \1\transp y = \frac{\rho}{1-\rho} \, \1\transp (\theta - \bar\theta \1)^+,
\end{align*}
which is precisely the fixed-point equation~\eqref{eq.fixed.point}.
We next show that \eqref{eq.fixed.point.sol} solves the fixed-point equation~\eqref{eq.fixed.point} under the ordering~\eqref{eq.order}.  
Let $i$ be the largest index such that $\bar\theta \le \theta_i$. Then~\eqref{eq.order} and~\eqref{eq.fixed.point} imply
 that
\[
\bar \theta = \frac{\rho}{1-\rho} = \1\transp (\theta - \bar\theta \1)^+ = 
\frac{\rho}{1-\rho}
\sum_{j=1}^i(\theta_j - \bar \theta) = \frac{\rho}{1-\rho}\left(\sum_{j=1}^i\theta_j - i\bar \theta\right).
\]
After rearranging and simplifying terms we obtain
\[
\bar \theta = \frac{\rho}{1+(i-1)\rho} \sum_{j=1}^i \theta_j,
\]
which is~\eqref{eq.fixed.point.sol}. To finish, recall that $i$ is the largest index such that $\bar \theta \le \theta_i$.  That is, the largest index $i$ such that
\[
\frac{\rho}{1+(i-1)\rho}\sum_{j=1}^i \theta_j \le \theta_i.
\]
By rearranging terms, this inequality can be equivalently stated as
\[
\rho \le \frac{\theta_i}{\theta_i + \sum_{j=1}^{i-1} (\theta_j - \theta_i)}
\]
which is the same as~\eqref{eq.largest.i}
Thus, $i$ is the largest index for which ~\eqref{eq.largest.i}  holds.
\end{proof}

\subsection{Proof of positive definiteness of a block-diagonal correlation matrix}\label{app,proof.pd}

\begin{lemma}
Let
\[
\Omega = (1-\rho)\operatorname{Diag}(\Omega_1,\dots,\Omega_K)
+ \rho \1\1\transp,
\qquad
\Omega_i=(1-\rho_i)I_i+\rho_i \1_i\1_i\transp.
\]
Then $\Omega \succ 0$ if
\[
-\frac{1}{n_i-1}<\rho_i<1,
\qquad i=1,\dots,K,
\]
and
\[
-\left(
\sum_{i=1}^K
\frac{n_i}{1+(n_i-1)\rho_i}
-1
\right)^{-1}
< \rho < 1.
\]
\end{lemma}

\begin{proof}
Each block $\Omega_i$ is a constant-correlation matrix with eigenvalues
\[
1-\rho_i
\quad (\text{multiplicity } n_i-1),
\qquad
1+(n_i-1)\rho_i
\quad (\text{simple eigenvalue}).
\]
Hence,
\[
\Omega_i \succ 0
\quad \Longleftrightarrow \quad
-\frac{1}{n_i-1} < \rho_i < 1.
\]
Under these conditions,
\[
A:=\operatorname{Diag}(\Omega_1,\dots,\Omega_K)
\]
is positive definite.

Since positive definiteness is preserved under congruence transformations,
\[
\Omega \succ 0
\quad \Longleftrightarrow \quad
A^{-1/2}\Omega A^{-1/2}\succ 0.
\]
Using the definition of $\Omega$,
\[
A^{-1/2}\Omega A^{-1/2}
=
(1-\rho)I+\rho vv\transp,
\qquad
v:=A^{-1/2}\1.
\]
This is a rank-one perturbation of $(1-\rho)I$, whose eigenvalues are
\[
1-\rho
\quad (\text{multiplicity } n-1),
\qquad
1-\rho+\rho\|v\|^2
\quad (\text{simple eigenvalue}).
\]
Hence,
\[
\Omega \succ 0
\]
if and only if
\[
1-\rho>0,
\qquad
1-\rho+\rho\|v\|^2>0.
\]
It remains to compute $\|v\|^2$. Using the block structure,
\[
\|v\|^2
=
\1\transp A^{-1}\1
=
\sum_{i=1}^K
\1_i\transp \Omega_i^{-1}\1_i.
\]
Now, since
\[
\Omega_i=(1-\rho_i)I_i+\rho_i \1_i\1_i\transp,
\]
the vector $\1_i$ is an eigenvector of $\Omega_i$ associated with the eigenvalue
\[
1-\rho_i+\rho_i n_i
=
1+(n_i-1)\rho_i.
\]
Therefore,
\[
\Omega_i^{-1}\1_i
=
\frac{1}{1+(n_i-1)\rho_i}\1_i,
\]
and consequently
\[
\1_i\transp \Omega_i^{-1}\1_i
=
\frac{n_i}{1+(n_i-1)\rho_i}.
\]
Substituting back yields
\[
\|v\|^2
=
\sum_{i=1}^K
\frac{n_i}{1+(n_i-1)\rho_i}.
\]
The second positivity condition therefore becomes
\[
1+\rho\left(
\sum_{i=1}^K
\frac{n_i}{1+(n_i-1)\rho_i}
-1
\right)>0.
\]
Moreover, under
\[
-\frac{1}{n_i-1}<\rho_i<1,
\]
we have
\[
1+(n_i-1)\rho_i<n_i,
\]
and hence
\[
\frac{n_i}{1+(n_i-1)\rho_i}>1.
\]
Therefore,
\[
\sum_{i=1}^K
\frac{n_i}{1+(n_i-1)\rho_i}
-1
>0,
\]
so the previous inequality is equivalent to
\[
\rho>
-\frac{1}{
\displaystyle
\sum_{i=1}^K
\frac{n_i}{1+(n_i-1)\rho_i}
-1
}.
\]
Combining this with the condition $1-\rho>0$ yields
\[
-\frac{1}{
\displaystyle
\sum_{i=1}^K
\frac{n_i}{1+(n_i-1)\rho_i}
-1
}
< \rho < 1,
\]
which completes the proof.
\end{proof}

\subsection{Proof of Proposition \ref{prop.K.block}} \label{app:proof.prop.k.block}

\propositionkblock*

\begin{proof}
Let $y$ be an optimal solution of \eqref{eq.min.correl}. By
Lemma~\ref{lemma:correlation.reformulation}, the KKT conditions
\eqref{eq.kkt_a}--\eqref{eq.kkt_b} hold, i.e.,
\[
\Omega y = \theta + \mu,
\qquad
y \ge 0,\quad \mu \ge 0,\quad \mu^\top y = 0.
\]
Partition $y$, $\theta$, and $\mu$ according to the block structure,
\[
y=(y^{(1)},\dots,y^{(K)}),\qquad
\theta=(\theta^{(1)},\dots,\theta^{(K)}),\qquad
\mu=(\mu^{(1)},\dots,\mu^{(K)}).
\]
Using the decomposition
\[
\Omega
=
(1-\rho)\Diag(\Omega_1,\dots,\Omega_K)
+\rho \1\1^\top,
\qquad
\Omega_i=(1-\rho_i)I_i+\rho_i \1_i\1_i^\top,
\]
the $i$-th block of the KKT stationarity condition becomes
\[
(1-\rho)\Omega_i y^{(i)}
+
\Bigl(\rho\sum_{j=1}^K \1_j^\top y^{(j)}\Bigr)\1_i
=
\theta^{(i)}+\mu^{(i)}.
\]
Define
\[
\bar\theta_i := (1-\rho)\rho_i\,\1_i^\top y^{(i)},
\qquad
\hat\theta := \rho \sum_{j=1}^K \1_j^\top y^{(j)}.
\]
Then we rewrite the block equation as
\[
(1-\rho)(1-\rho_i)y^{(i)}
=
\theta^{(i)} - (\hat\theta + \bar\theta_i)\1_i + \mu^{(i)}.
\]
Applying complementarity \eqref{eq.kkt_b} yields
\[
y^{(i)}
=
\frac{1}{(1-\rho)(1-\rho_i)}
\left(
\theta^{(i)} - (\hat\theta + \bar\theta_i)\1_i
\right)^+,
\]
which proves \eqref{eq.sol.y.block}.

Finally, substituting this expression back into the definitions of
$\bar\theta_i$ and $\hat\theta$ yields the fixed-point relations
\[
\bar\theta_i
=
\frac{\rho_i}{1-\rho_i}
\1_i^\top
\left(
\theta^{(i)} - (\hat\theta + \bar\theta_i)\1_i
\right)^+,
\]
and
\[
\hat\theta
=
\frac{\rho}{1-\rho}
\sum_{j=1}^K
\frac{1}{1-\rho_j}
\1_j^\top
\left(
\theta^{(j)} - (\hat\theta + \bar\theta_j)\1_j
\right)^+,
\]
completing the proof.
\end{proof}

\subsection{Proof of Proposition \ref{prop:block.fit}} \label{app:proof.prop.block.fit}

\propositionblockfit*

\begin{proof}
Since both $\Omega^{\mathrm{data}}$ and $\Omega$ are correlation matrices,
their diagonal entries are equal to one and therefore do not contribute to the
Frobenius loss. Thus,
\[
\|\Omega^{\mathrm{data}}-\Omega\|_F^2
=
\sum_{(j,k):\,j\neq k}
\left(\Omega^{\mathrm{data}}_{jk}-\Omega_{jk}\right)^2.
\]

Under the structured form~\eqref{eq.block.structure}, all cross-block entries
share the common value $\rho$, while all off-diagonal entries within block $i$
share the value
\[
v_i := (1-\rho)\rho_i+\rho.
\]
Therefore, the loss can be decomposed as
\[
\|\Omega^{\mathrm{data}}-\Omega\|_F^2
=
\sum_{(p,q)\in S_C}
\left(\Omega^{\mathrm{data}}_{pq}-\rho\right)^2
+
\sum_{i=1}^K
\sum_{\substack{j,k\in C_i\\ j\neq k}}
\left(\Omega^{\mathrm{data}}_{jk}-v_i\right)^2.
\]
The unique minimizer of the first term of this separable loss is evidently~\eqref{eq.rho.fit} and the unique minimizer of the second term is
\[
v_i^*
=
\frac{1}{n_i(n_i-1)}
\sum_{\substack{j,k\in C_i\\ j\neq k}}
\Omega^{\mathrm{data}}_{jk}, \; i=1,\dots,K.
\]
Finally by recovering $\rho_i$ from
$
v_i=(1-\rho)\rho_i+\rho
$
we get
\[
\rho_i^*
=
\frac{v_i^*-\rho^*}{1-\rho^*},
\]
which establishes~\eqref{eq.rhoi.fit}.
\end{proof}

\section{Robustness Results}\label{sec:appendix_results}

All BDS, HRP, HERC, and HCAA portfolios admit closed-form solutions and are implemented in \texttt{NumPy}. The LOM benchmark is formulated in \texttt{Pyomo} \citep{bynum2021pyomo} and solved with \texttt{IPOPT} \citep{wachter2006ipopt}. Experiments run on Ubuntu with eight Intel Xeon Gold 6234 CPUs (3.30 GHz) and 1 TB RAM. Code is available at our \href{https://github.com/dovallev/fragility-minimum-variance}{GitHub repository}\footnote{https://github.com/dovallev/fragility-minimum-variance}.


\subsection{Robustness: Cross-Block Correlation and Covariance Estimator}\label{sec:robustness}

The previous experiments rely on a fixed covariance structure and a baseline estimator. We now test robustness along two  dimensions:  changes of cross-block correlation and alternative covariance estimation procedures. Both perturbations act on the same object, the reliability of cross-sectional dependence information used by portfolio construction. This separates sensitivity to misspecification in the data-generating process from sensitivity to statistical estimation error.

\paragraph{Cross-block correlation.}

We vary $\rho \in \{0.05, 0.1, 0.2, 0.3, 0.4, 0.5\}$ while holding all other aspects of the data-generating process fixed. Figure~\ref{fig:e4_pareto_rho} reports Pareto fronts for the illustrative case of $n=100$ and $\sigma_s=1.2$.
Increasing cross-block dependence induces a continuous but asymmetric degradation across methods. When $\rho$ is small, separating assets into blocks aligns well with the dependence structure, and methods that suppress cross-block interaction remain competitive. As $\rho$ increases, this representation becomes increasingly restrictive and the cost of ignoring inter-block coupling rises sharply in realized variance.

BDS-Aware remains on the Pareto frontier throughout the entire sweep, indicating that preserving both within- and cross-block information yields a robust trade-off across dependence regimes. At low global correlation, BDS-Obliv also lies on the frontier and achieves the lowest turnover, but its area deteriorates rapidly as $\rho$ increases and cross-block dependence becomes more important. At the same time, BDS-Direct enters the frontier and progressively becomes the lowest-area BDS variant at higher $\rho$, indicating that enforcing structure through the recovered block representation becomes increasingly beneficial as global coupling strengthens. Despite this within-family transition, all BDS variants remain separated from the benchmark and hierarchical methods across the full range of $\rho$, preserving a superior area--turnover trade-off throughout.

\begin{figure}[htbp]
\centering
\includegraphics[width=\textwidth]{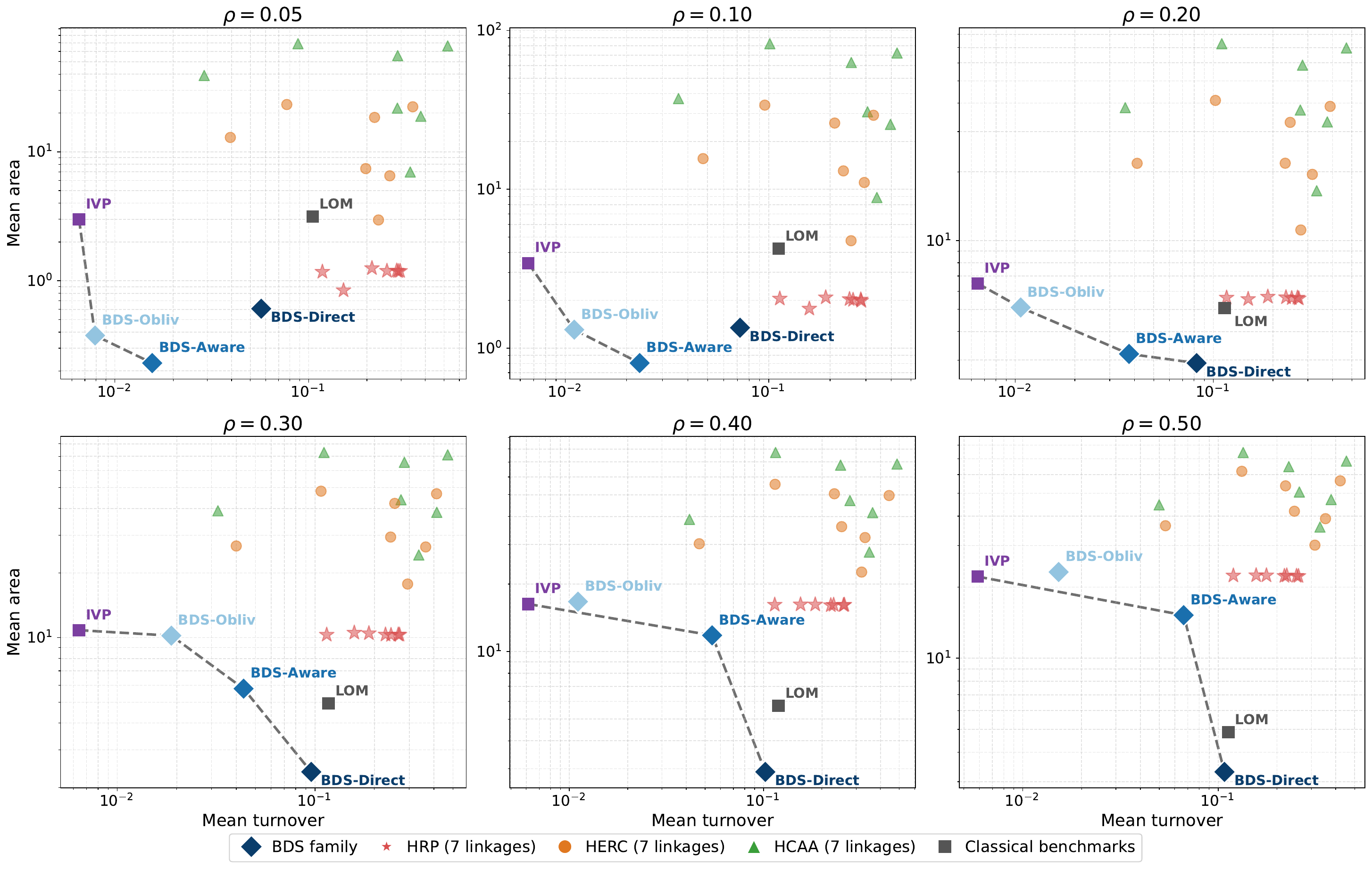}
\caption{Pareto scatter at $n=100$ and $\sigma_s=1.2$ under increasing cross-block correlation. BDS-Aware remains on or near the frontier across all $\rho$. BDS-Obliv degrades as cross-block dependence increases, while BDS-Direct becomes the lowest-area method at higher $\rho$.}
\label{fig:e4_pareto_rho}
\end{figure}

\paragraph{Covariance estimator.}

We next vary the covariance estimator to isolate sensitivity to statistical estimation error in cross-sectional dependence. At $n=100$ in the heterogeneous regime, we consider nine estimators spanning three families: structure-preserving estimators (sample covariance, CC-LW \cite{ledoit2004honey}, PCA \cite{connor1993test}), identity shrinkage estimators (fixed 5\% and 10\% shrinkage \cite{james1961estimation}, Ledoit–Wolf \cite{ledoit2004well}, OAS \cite{chen2010shrinkage}), and exponentially weighted estimators (EWMA with $\lambda \in \{0.95, 0.90\}$ \cite{longerstaey1996riskmetricstm}). Figure~\ref{fig:e5_pareto_by_family} aggregates results by estimator family.

Across estimators, the primary effect is a rescaling of estimation noise rather than a change in method ordering. Structure-preserving estimators retain covariance geometry and therefore preserve the separation between BDS methods and classical benchmarks. Identity shrinkage improves conditioning and narrows performance gaps by benefiting unconstrained or weakly structured portfolios, yet BDS-Direct remains in the frontier throughout. OAS is the only case where improved conditioning slightly favors long-only Markowitz in realized variance, without altering turnover ordering.

EWMA estimators increase temporal variability and amplify instability in classical portfolios. BDS-Direct is also affected, but remains comparatively stable in both area and turnover. Across all estimator families, the ranking of methods is invariant; only the magnitude of separation changes. This indicates that performance differences are driven primarily by portfolio construction rather than covariance estimation.

\begin{figure}[htbp]
\centering
\includegraphics[width=\textwidth]{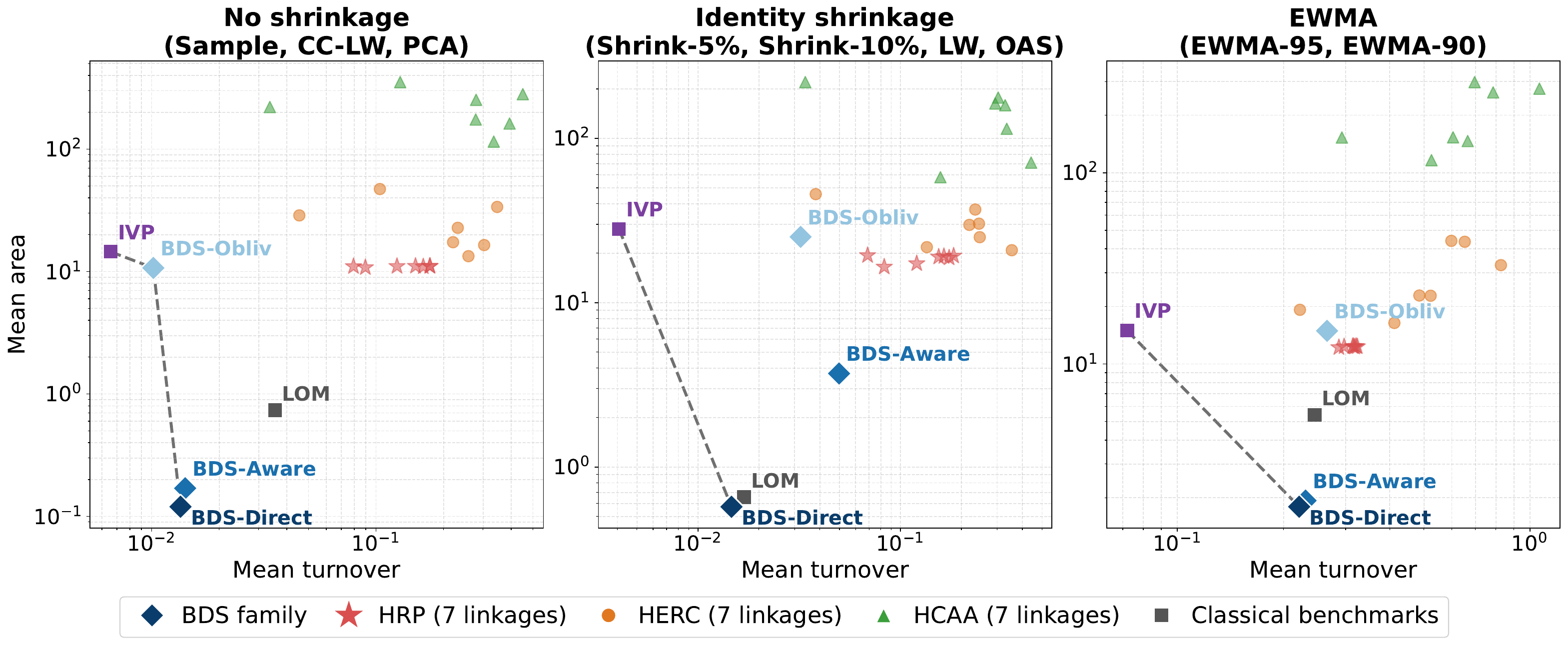}
\caption{Pareto scatter by covariance-estimator family at $n=100$ and $w=2n$. BDS-Direct remains near the frontier across all estimator classes. Differences across estimators primarily affect the scale of estimation noise rather than the ordering of methods.}
\label{fig:e5_pareto_by_family}
\end{figure}


\subsection{Hyperparameter Sensitivity}\label{sec:lopez_sensitivity}

We next vary the factor-loading noise $\eta$ while holding $(u, c) = (67, 33)$ fixed to control the signal-to-noise ratio in factor recovery. Increasing $\eta$ weakens factor identifiability and amplifies estimation error in the covariance matrix. The global shock parameter $\delta$ is held fixed since it does not materially affect comparative rankings in this regime.

Figure~\ref{fig:e3_pareto_sweep} shows a clear regime transition within the BDS family. At low noise ($\eta = 0.05$), BDS-Aware and BDS-Obliv are on the frontier, while BDS-Direct is slightly dominated due to unnecessary regularization when structure is accurately estimated. At intermediate noise ($\eta = 0.25$), all three variants become competitive on the frontier. At high noise ($\eta = 1.0$), BDS-Direct becomes dominant, improving area by approximately 25–28

This shift reflects a bias–variance trade-off in structural enforcement. BDS-Direct imposes the strongest block constraint and therefore benefits most when estimation error dominates signal recovery. The weaker variants are better aligned with low-noise regimes but degrade more rapidly as covariance estimates become unreliable.
Across all regimes, the BDS family remains Pareto-stable, but the identity of the optimal variant depends systematically on the signal-to-noise ratio, with BDS-Direct emerging as the preferred high-noise regularizer.

\begin{figure}[htbp]
\centering
\includegraphics[width=\textwidth]{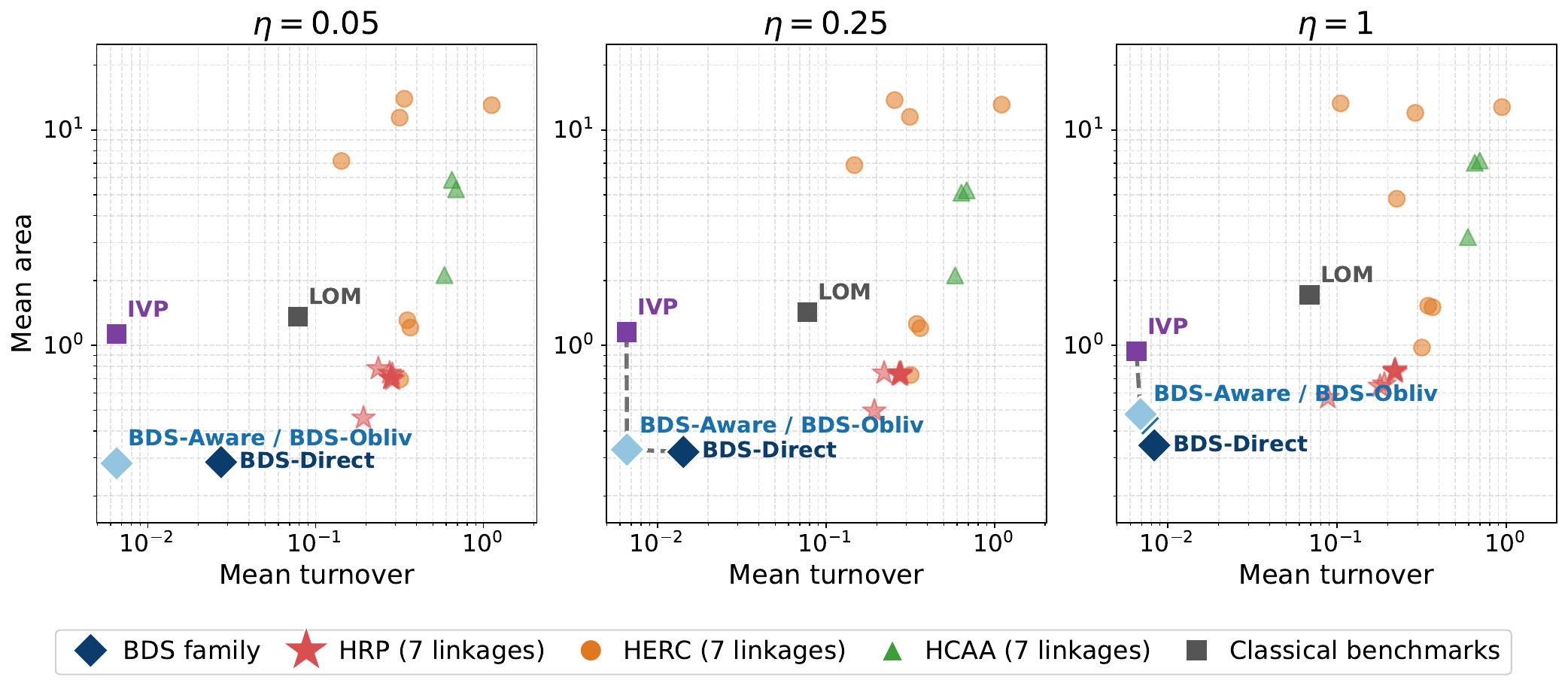}
\caption{Area–turnover Pareto frontier across factor-loading noise levels in the Lopez de Prado model ($u=67$, $c=33$): $\eta \in \{0.05, 0.25, 1.0\}$. Points are averaged over 10 seeds. The frontier transitions from BDS-Aware and BDS-Obliv dominance at low noise to BDS-Direct dominance at high noise, reflecting a shift from estimation-accurate to estimation-limited regimes.}
\label{fig:e3_pareto_sweep}
\end{figure}


\subsection{Robustness across Estimation Windows}\label{sec:rd_tn}

Table~\ref{tab:rd_tn} reports realized volatility and turnover across both estimation windows. Two patterns are stable.
First, BDS-Direct maintains a consistent volatility advantage over LOM and this advantage increases slightly with window length, from 9.1\% at $w=2n$ to 10.4\% at $w=3n$. If the effect were driven only by estimation noise at short windows, the gap would shrink as estimation improves. Instead, better estimates strengthen the ability of BDS-Direct to exploit recovered block structure.

Second, turnover declines for all methods as the window increases, but the reduction is larger for BDS-Direct. Between $w=2n$ and $w=3n$, turnover decreases by 24\% for BDS-Direct versus 12\% for LOM. This reflects a structural difference: LOM reoptimizes over a dense covariance at every step, while BDS-Direct operates on a compressed block representation whose composition stabilizes more quickly as estimation improves.
Rankings remain unchanged across windows. BDS-Direct and BDS-Aware minimize volatility, IVP minimizes turnover, and hierarchical methods remain strictly interior.

\begin{table}[htbp]
\centering
\small
\caption{Realized volatility, turnover, and relative turnover across estimation windows $w \in \{2n, 3n\}$. Data: Fama-French 49 industry XSMOM returns with Ledoit-Wolf estimator (1970–2026). All hierarchical results use Ward linkage for consistency. All hierarchical variants are strictly dominated by BDS-Direct in both metrics across all 21 linkage choices (see Figure~\ref{fig:rd_pareto}).}
\label{tab:rd_tn}
\begin{tabular}{l rrr rrr}
\toprule
& \multicolumn{3}{c}{$w = 2n$}
& \multicolumn{3}{c}{$w = 3n$} \\
\cmidrule(lr){2-4}\cmidrule(lr){5-7}
Method & Vol. & Turn. & $\Delta$Turn & Vol. & Turn. & $\Delta$Turn \\
\midrule
BDS-Obliv  & 0.001688 & 0.0209 & \textbf{+65\%} & 0.001662 & 0.0141 & \textbf{+40\%} \\
BDS-Aware  & 0.001483 & 0.0427 & +236\% & \textbf{0.001389} & 0.0455 & +350\% \\
BDS-Direct & \textbf{0.001467} & 0.0241 & +90\% & 0.001405 & 0.0182 & +80\% \\
\midrule
LOM \citep{markowitz1952}        & 0.001614 & 0.0382 & +201\% & 0.001568 & 0.0337 & +234\% \\
EW \citep{demiguel2009optimal}         & 0.003117 & ---    & --- & 0.003084 & ---    & --- \\
IVP \citep{clarke2006minimum}        & 0.001708 & \textbf{0.0127} & --- & 0.001681 & \textbf{0.0101} & --- \\
HRP-ward \citep{lopez2016building}   & 0.001593 & 0.1504 & +1084\% & 0.001575 & 0.1398 & +1284\% \\
HERC-ward \citep{raffinot2018hierarchical}  & 0.001862 & 0.1609 & +1167\% & 0.001803 & 0.1683 & +1566\% \\
HCAA-ward \citep{raffinot2018hcaa}  & 0.004094 & 0.2736 & +2054\% & 0.003413 & 0.2695 & +2568\% \\
\bottomrule
\end{tabular}
\end{table}

\subsection{Robustness to Covariance Estimator: Sample Covariance}\label{sec:rd_estimator}

We next replace Ledoit-Wolf with the unshrunk sample covariance, which removes shrinkage-based regularization and provides a stricter test of structural robustness.

Table~\ref{tab:rd_tn_sample} shows that BDS-Direct remains dominant over LOM across all windows. The volatility advantage increases at short windows, reaching 14.0\% at $w=2n$, and remains comparable at longer windows. The turnover advantage is also larger than under Ledoit-Wolf, reflecting increased instability in unregularized estimates.

\begin{table}[htbp]
\centering
\small
\caption{Realized volatility and turnover under sample covariance estimation across $w \in \{2n, 3n\}$. Data: Fama-French 49 industry XSMOM returns (1970–2026). All hierarchical methods use Ward linkage for consistency. BDS-Direct remains strictly dominant over all hierarchical variants across all linkage choices and windows.}
\label{tab:rd_tn_sample}
\begin{tabular}{l rrr rrr}
\toprule
& \multicolumn{3}{c}{$w = 2n$}
& \multicolumn{3}{c}{$w = 3n$} \\
\cmidrule(lr){2-4}\cmidrule(lr){5-7}
Method & Vol. & Turn. & $\Delta$Turn & Vol. & Turn. & $\Delta$Turn \\
\midrule
BDS-Obliv  & 0.001300 & 0.0765 & +163\% & 0.001374 & 0.0382 & \textbf{+88\%} \\
BDS-Aware  & 0.001222 & 0.2127 & +631\% & 0.001328 & 0.1570 & +673\% \\
BDS-Direct & \textbf{0.001189} & 0.0731 & \textbf{+151\%} & \textbf{0.001258} & 0.0427 & +110\% \\
\midrule
LOM \citep{markowitz1952}        & 0.001381 & 0.0979 & +236\% & 0.001400 & 0.0908 & +347\% \\
EW \citep{demiguel2009optimal}         & 0.003117 & ---    & --- & 0.003084 & ---    & --- \\
IVP \citep{clarke2006minimum}        & 0.001314 & \textbf{0.0291} & --- & 0.001370 & \textbf{0.0203} & --- \\
HRP-ward \citep{lopez2016building}   & 0.001277 & 0.2007 & +590\% & 0.001372 & 0.1590 & +683\% \\
HERC-ward \citep{raffinot2018hierarchical}  & 0.001298 & 0.1722 & +492\% & 0.001376 & 0.1514 & +646\% \\
HCAA-ward \citep{raffinot2018hcaa}  & 0.003231 & 0.2704 & +829\% & 0.003075 & 0.2389 & +1077\% \\
\bottomrule
\end{tabular}
\end{table}

However, performance becomes less temporally stable. Rolling-volatility dominance drops from 84.8\% of months under Ledoit-Wolf to 76.6\% under sample covariance, indicating that improved average performance comes at the cost of higher variability across time.
A second difference is structural. Under Ledoit-Wolf, BDS-Aware overtakes BDS-Direct at longer windows. This transition disappears under sample covariance, consistent with the idea that shrinkage stabilizes cross-block signals sufficiently for more refined structure to become useful. Without it, estimation noise dominates and the more aggressive regularization of BDS-Direct remains uniformly preferable.
Despite these differences, the ranking is unchanged in aggregate: BDS methods remain strictly superior to LOM and hierarchical constructions across all specifications. Corresponding figures are reported in Supplementary Material~\ref{sec:appendix_real_sample}.

\end{appendices}
\clearpage

\bibliography{references}

\clearpage
\section*{Supplementary Material}
\addcontentsline{toc}{section}{Supplementary Material}

\setcounter{section}{0}

\renewcommand{\thesection}{S\arabic{section}}


\section{Extended Results}


\subsection{Rank Heatmaps: Calibrated Sweep}\label{sec:appendix_rank_e1_cal}

Figure~\ref{fig:app_rank_e1_cal} reports area and turnover ranks across all $(n,\text{regime})$ cells in the calibrated sweep at $w=2n$, $\sigma_s=1.2$, $n\in\{50,100,150,200,250\}$. Rows are sorted by mean area rank. BDS-Direct and BDS-Aware consistently occupy the top two area ranks in heterogeneous regimes, while BDS-Obliv dominates homogeneous regimes. On turnover, BDS-Obliv ranks first in all cells, followed by BDS-Direct and BDS-Aware. No benchmark or hierarchical method enters the top two area ranks in heterogeneous regimes, consistent with Figure~\ref{fig:e1_pareto_pooled}.

\begin{figure}[h]
    \centering
    \includegraphics[width=0.95\textwidth]{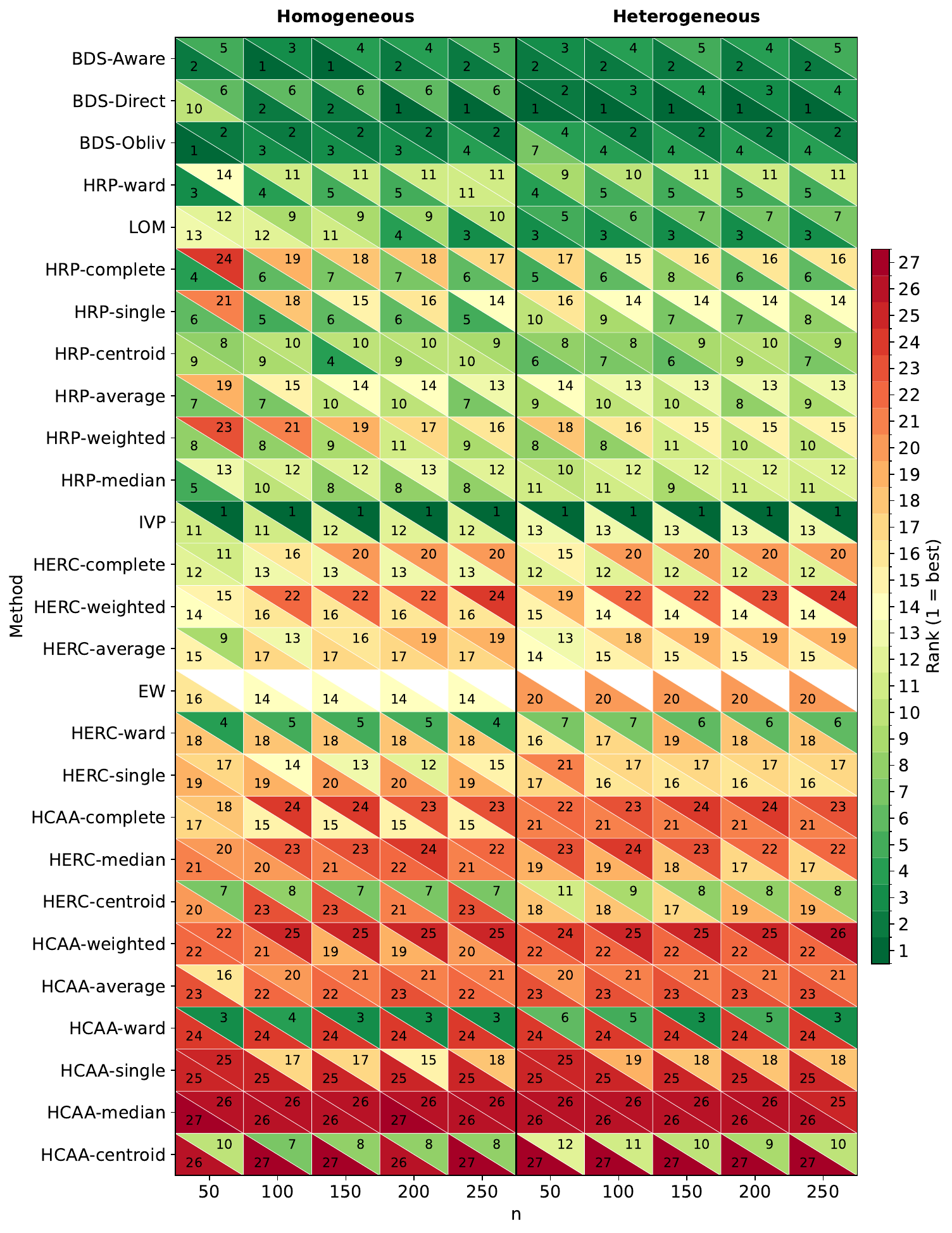}
    \caption{Method ranks across $(n,\text{regime})$ cells in the calibrated sweep ($w=2n$, $\sigma_s=1.2$). Each cell shows area rank (lower-left) and turnover rank (upper-right), where 1 is best (green). BDS-Direct and BDS-Aware dominate area in heterogeneous regimes, while BDS-Obliv dominates turnover across all regimes. Equal Weight is excluded from turnover ranking due to zero turnover.}
    \label{fig:app_rank_e1_cal}
\end{figure}

\subsection{Pareto Frontier by Universe Size: Heterogeneous Regime}\label{sec:appendix_pareto_by_n}

Figure~\ref{fig:app_pareto_by_n_het} reports Pareto frontiers by universe size at $w=2n$ in the heterogeneous regime. BDS methods define the frontier for every $n$, confirming that pooled dominance is not aggregation-driven. The performance gap increases with $n$, consistent with Figure~\ref{fig:e1_fragility}.

\begin{figure}[h]
    \centering
    \includegraphics[width=\textwidth]{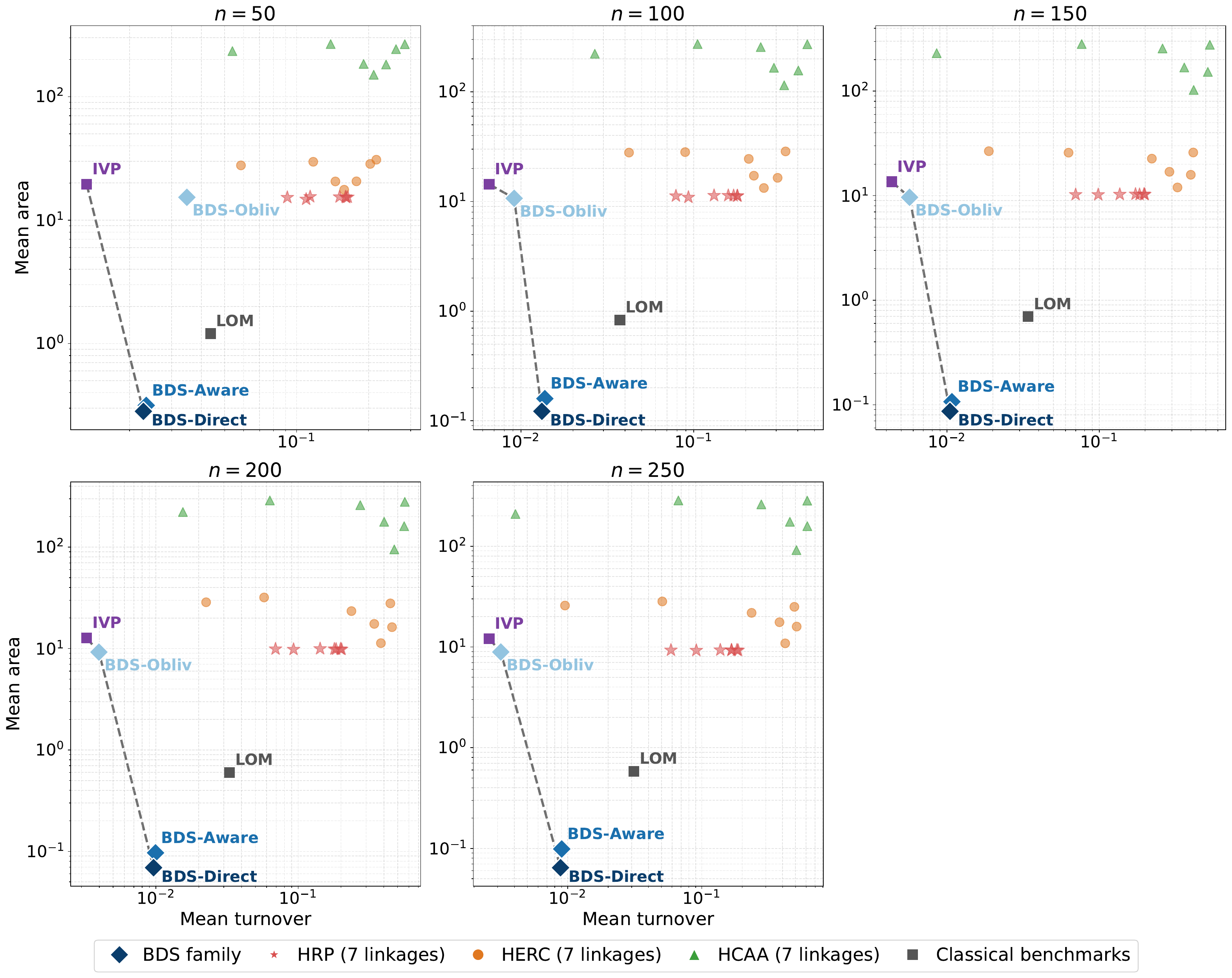}
    \caption{Pareto frontier by universe size $n$ in the heterogeneous regime ($w=2n$). BDS methods define the frontier at every $n$, with a widening performance gap as $n$ increases.}
    \label{fig:app_pareto_by_n_het}
\end{figure}

\subsection{Window Length Robustness: Rank Heatmap}\label{sec:appendix_tn_robustness}

Figure~\ref{fig:app_tn_robustness} compares ranks at $w=2n$ and $w=3n$. Each cell fixes ordering by the $w=2n$ rank; discrepancies across triangles indicate changes. Rankings are stable across windows: BDS-Direct and BDS-Aware remain top in heterogeneous regimes, while BDS-Obliv dominates homogeneous regimes.

\begin{figure}[h]
    \centering
    \includegraphics[width=\textwidth]{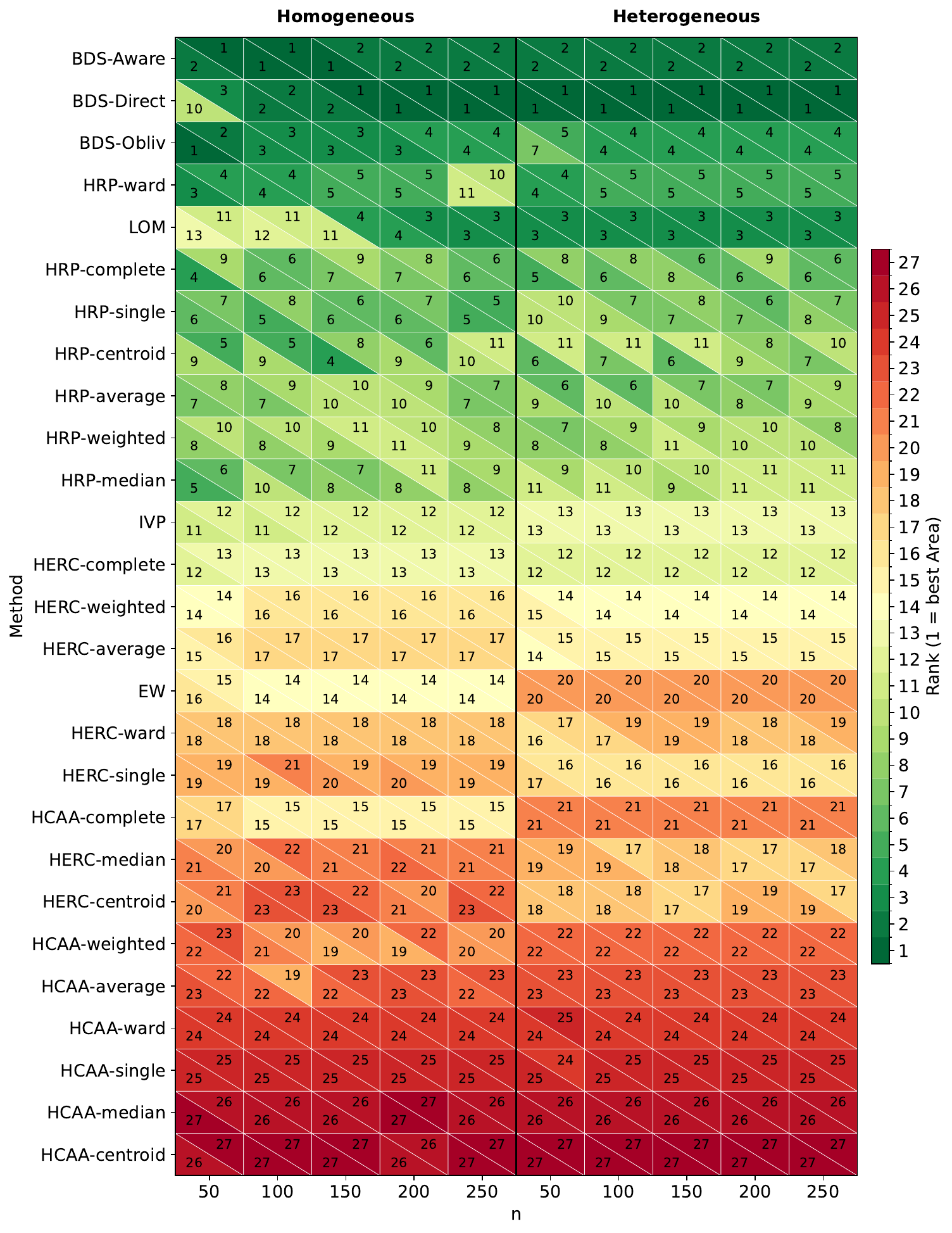}
    \caption{Area ranks across $(n,\text{regime})$ for $w\in\{2n,3n\}$. Each cell shows $w=2n$ (lower-left) and $w=3n$ (upper-right), with 1 as best (green). Ordering is fixed to $w=2n$. Rankings are stable across window lengths.}
    \label{fig:app_tn_robustness}
\end{figure}


\subsection{Robustness: Covariance Estimator}\label{sec:appendix_robustness_estimator}

Figure~\ref{fig:e5_pareto_by_family} in Appendix~\ref{sec:robustness} reports estimator-level Pareto results. Table~\ref{tab:app_estimator} reports mean area at $n=100$. Under structure-preserving estimators (Sample, CC-LW, PCA), BDS-Direct is uniformly best, with roughly $5\times$–$7\times$ improvement over LOM. Under identity shrinkage, LOM improves and the gap narrows; under OAS, LOM is the only estimator where it outperforms BDS-Direct. Under EWMA, LOM deteriorates sharply while BDS-Direct remains stable.

Figure~\ref{fig:app_e5_vol_ratio} reports volatility changes relative to sample covariance. Identity shrinkage slightly benefits LOM while increasing BDS-Direct volatility by about $2$–$3\%$, consistent with reduced dispersion in cross-block structure. Under EWMA, LOM volatility increases substantially, while BDS-Direct remains comparatively stable.

\begin{figure}[h]
    \centering
    \includegraphics[width=\textwidth]{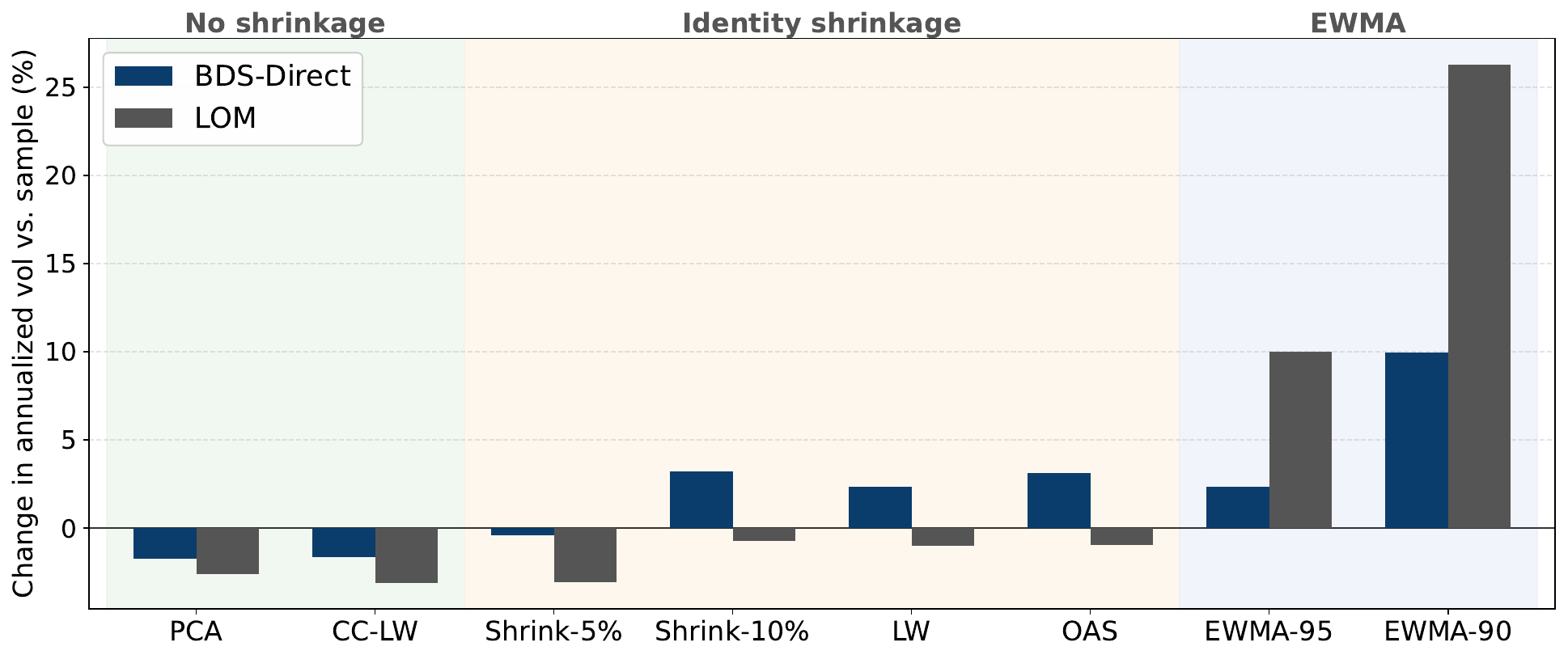}
    \caption{Change in realized volatility relative to sample covariance baseline in the heterogeneous regime ($n=100$, $w=2n$). Bars compare BDS-Direct (blue) and LOM (gray) across estimators. Positive values indicate higher volatility. Sample baseline omitted by construction.}
    \label{fig:app_e5_vol_ratio}
\end{figure}

\begin{table}[h]
\centering
\footnotesize
\caption{Mean area at $n=100$ across covariance estimators (10 seeds). Bold indicates column minimum.}
\label{tab:app_estimator}
\resizebox{\textwidth}{!}{%
\begin{tabular}{l ccc cccc cc}
\toprule
 & \multicolumn{3}{c}{No shrinkage} & \multicolumn{4}{c}{Identity shrinkage} & \multicolumn{2}{c}{EWMA} \\
\cmidrule(lr){2-4}\cmidrule(lr){5-8}\cmidrule(lr){9-10}
Method & Sample & CC-LW & PCA & Shrink-5\% & Shrink-10\% & LW & OAS & EWMA-95 & EWMA-90 \\
\midrule
BDS-Direct  & \textbf{0.146} & \textbf{0.108} & \textbf{0.105} & \textbf{0.105} & \textbf{0.637} & \textbf{0.790} & 0.760 & \textbf{1.007} & \textbf{2.580} \\
BDS-Aware   & 0.183 & 0.187 & 0.140 & 1.980 & 4.801 & 3.386 & 4.664 & 1.050 & 3.032 \\
BDS-Obliv   & 10.749 & 10.784 & 10.573 & 18.387 & 27.073 & 27.697 & 27.390 & 14.040 & 15.998 \\
\midrule
LOM \citep{markowitz1952} & 0.825 & 0.628 & 0.760 & 0.336 & 0.692 & 0.884 & \textbf{0.710} & 3.321 & 7.528 \\
HRP-ward \citep{lopez2016building} & 10.981 & 11.107 & 10.307 & 15.454 & 20.495 & 20.893 & 20.757 & 11.938 & 12.482 \\
IVP \citep{clarke2006minimum} & 14.585 & 14.546 & 14.406 & 22.671 & 29.629 & 30.152 & 30.048 & 15.008 & 14.871 \\
\bottomrule
\end{tabular}}
\end{table}


\subsection{Misspecified Factor Structure: Per-Configuration Results}\label{sec:appendix_e2_byconfig}

Figure~\ref{fig:app_e2_pareto_by_config} reports Pareto frontiers for each Lopez configuration. BDS methods define the frontier in all cases, while HRP improves over LOM but remains strictly dominated. The advantage of BDS-Direct increases with universe size, reaching a $4.3\times$ reduction in area at $n=100$. Table~\ref{tab:e2_metrics} reports full per-configuration means and standard deviations.

\begin{figure}[h]
    \centering
    \includegraphics[width=\textwidth]{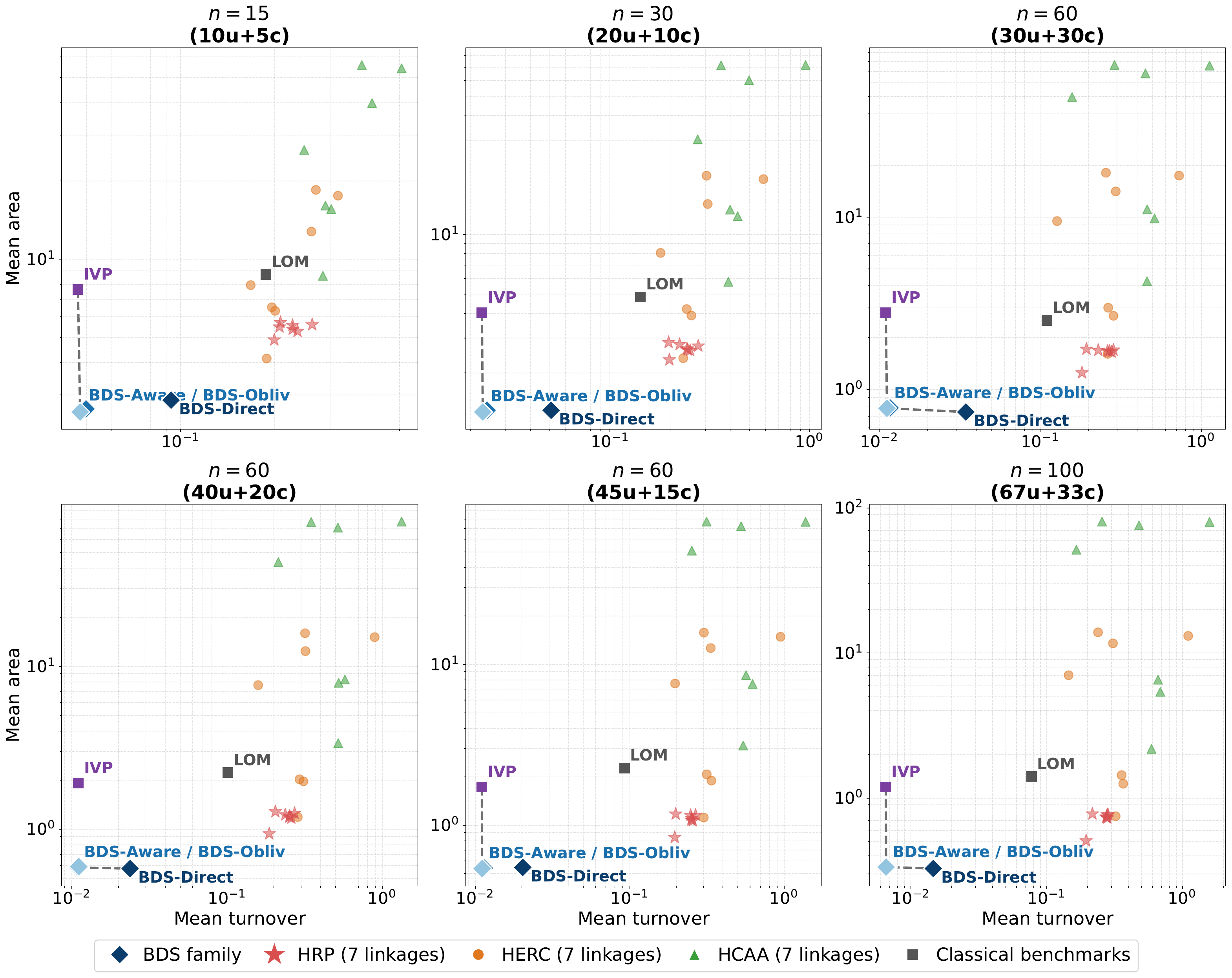}
    \caption{Pareto frontiers for six Lopez configurations (10 seeds). BDS methods define the frontier in all cases; HRP remains strictly dominated.}
    \label{fig:app_e2_pareto_by_config}
\end{figure}

\begin{table}[h]
\centering
\footnotesize
\caption{Mean (std. dev.) across all six Lopez de Prado configurations (10 seeds). Bold marks the lowest area and volatility\ within each panel, and the lowest turnover\ and $\Delta$Turn. Active set and ENB shown without std or highlighting. EW turnover\ and $\Delta$Turn omitted (---) as zero rebalancing is trivial by construction.}
\label{tab:e2_metrics}
\setlength{\tabcolsep}{3pt}
\begin{tabular}{l rrrrrrr}
\toprule
 & BDS-Obliv & BDS-Aware & BDS-Direct & LOM \citep{markowitz1952} & IVP \citep{clarke2006minimum} & HRP-ward \citep{lopez2016building} & EW \citep{demiguel2009optimal} \\
\midrule
\multicolumn{8}{l}{\textit{$n=15$,\ $u=10$,\ $c=5$}} \\
\quad Area          & \textbf{2.58 (0.35)} & 2.65 (0.37) & 2.86 (0.31) & 8.73 (1.06) & 7.63 (1.27) & 4.89 (0.53) & 6.99 (1.00) \\
\quad Volatility           & \textbf{5.10 (0.27)} & \textbf{5.10 (0.28)} & 5.14 (0.28) & 5.72 (0.35) & 5.67 (0.28) & 5.37 (0.26) & 5.62 (0.31) \\
\quad Turnover         & \textbf{0.047 (0.001)} & 0.050 (0.003) & 0.093 (0.007) & 0.187 (0.011) & \textbf{0.047 (0.001)} & 0.199 (0.010) & --- \\
\quad $\Delta$Turn   & \textbf{0\%} & +6\% & +98\% & +298\% & --- & +323\% & --- \\
\quad Active / ENB  & 15.0\ /\ 11.8 & 15.0\ /\ 11.8 & 11.7\ /\ 9.9 & 9.9\ /\ 7.5 & 15.0\ /\ 14.0 & 15.0\ /\ 12.0 & 15.0\ /\ 15.0 \\
\midrule
\multicolumn{8}{l}{\textit{$n=30$,\ $u=20$,\ $c=10$}} \\
\quad Area          & \textbf{1.27 (0.22)} & 1.30 (0.21) & 1.29 (0.27) & 4.83 (0.70) & 4.02 (0.72) & 2.33 (0.35) & 3.78 (0.70) \\
\quad Volatility           & \textbf{3.61 (0.19)} & \textbf{3.61 (0.19)} & 3.62 (0.20) & 4.07 (0.20) & 3.95 (0.20) & 3.77 (0.20) & 3.93 (0.20) \\
\quad Turnover         & \textbf{0.023 (0.000)} & 0.024 (0.002) & 0.051 (0.003) & 0.142 (0.011) & \textbf{0.023 (0.000)} & 0.198 (0.009) & --- \\
\quad $\Delta$Turn   & \textbf{0\%} & +4\% & +122\% & +517\% & --- & +761\% & --- \\
\quad Active / ENB  & 30.0\ /\ 24.6 & 30.0\ /\ 24.5 & 24.4\ /\ 20.7 & 19.8\ /\ 14.7 & 30.0\ /\ 28.8 & 30.0\ /\ 25.1 & 30.0\ /\ 30.0 \\
\midrule
\multicolumn{8}{l}{\textit{$n=60$,\ $u=30$,\ $c=30$}} \\
\quad Area          & 0.78 (0.13) & 0.78 (0.15) & \textbf{0.74 (0.12)} & 2.52 (0.25) & 2.79 (1.17) & 1.25 (0.11) & 2.77 (1.19) \\
\quad Volatility           & 2.95 (0.15) & 2.95 (0.16) & \textbf{2.94 (0.15)} & 3.27 (0.15) & 3.30 (0.20) & 3.05 (0.14) & 3.29 (0.19) \\
\quad Turnover        & \textbf{0.011 (0.000)} & 0.012 (0.001) & 0.035 (0.003) & 0.110 (0.006) & \textbf{0.011 (0.000)} & 0.182 (0.013) & --- \\
\quad $\Delta$Turn   & \textbf{0\%} & +9\% & +218\% & +900\% & --- & +1555\% & --- \\
\quad Active / ENB  & 60.0\ /\ 46.4 & 60.0\ /\ 46.4 & 43.5\ /\ 34.4 & 33.2\ /\ 24.4 & 60.0\ /\ 58.8 & 60.0\ /\ 47.7 & 60.0\ /\ 60.0 \\
\midrule
\multicolumn{8}{l}{\textit{$n=60$,\ $u=40$,\ $c=20$}} \\
\quad Area          & 0.59 (0.06) & 0.58 (0.05) & \textbf{0.57 (0.07)} & 2.23 (0.31) & 1.92 (0.29) & 0.94 (0.07) & 1.95 (0.25) \\
\quad Volatility           & 2.51 (0.14) & 2.51 (0.15) & \textbf{2.50 (0.15)} & 2.83 (0.18) & 2.72 (0.17) & 2.57 (0.15) & 2.71 (0.15) \\
\quad Turnover         & \textbf{0.011 (0.000)} & \textbf{0.011 (0.001)} & 0.024 (0.002) & 0.101 (0.008) & \textbf{0.011 (0.000)} & 0.187 (0.008) & --- \\
\quad $\Delta$Turn   & \textbf{0\%} & \textbf{0\%} & +118\% & +818\% & --- & +1600\% & --- \\
\quad Active / ENB  & 60.0\ /\ 49.4 & 60.0\ /\ 49.4 & 49.5\ /\ 42.6 & 40.0\ /\ 29.5 & 60.0\ /\ 58.8 & 60.0\ /\ 50.1 & 60.0\ /\ 60.0 \\
\midrule
\multicolumn{8}{l}{\textit{$n=60$,\ $u=45$,\ $c=15$}} \\
\quad Area          & \textbf{0.54 (0.06)} & 0.54 (0.06) & 0.55 (0.06) & 2.27 (0.25) & 1.73 (0.34) & 0.84 (0.06) & 1.86 (0.42) \\
\quad Volatility           & \textbf{2.36 (0.13)} & \textbf{2.36 (0.13)} & 2.37 (0.13) & 2.69 (0.12) & 2.59 (0.16) & 2.40 (0.12) & 2.62 (0.18) \\
\quad Turnover         & \textbf{0.011 (0.000)} & \textbf{0.011 (0.001)} & 0.020 (0.001) & 0.093 (0.004) & \textbf{0.011 (0.000)} & 0.196 (0.011) & --- \\
\quad $\Delta$Turn   & \textbf{0\%} & \textbf{0\%} & +82\% & +745\% & --- & +1682\% & --- \\
\quad Active / ENB  & 60.0\ /\ 51.5 & 60.0\ /\ 51.4 & 52.0\ /\ 46.4 & 42.3\ /\ 30.8 & 60.0\ /\ 58.8 & 60.0\ /\ 51.5 & 60.0\ /\ 60.0 \\
\midrule
\multicolumn{8}{l}{\textit{$n=100$,\ $u=67$,\ $c=33$}} \\
\quad Area          & 0.34 (0.02) & 0.34 (0.02) & \textbf{0.33 (0.03)} & 1.41 (0.15) & 1.19 (0.21) & 0.51 (0.03) & 1.20 (0.21) \\
\quad Volatility           & 1.95 (0.09) & 1.95 (0.09) & \textbf{1.95 (0.09)} & 2.23 (0.09) & 2.16 (0.08) & 1.99 (0.08) & 2.16 (0.08) \\
\quad Turn.         & \textbf{0.007 (0.000)} & \textbf{0.007 (0.000)} & 0.015 (0.001) & 0.077 (0.004) & \textbf{0.007 (0.000)} & 0.195 (0.009) & --- \\
\quad $\Delta$Turn   & \textbf{0\%} & \textbf{0\%} & +114\% & +1000\% & --- & +2686\% & --- \\
\quad Active / ENB  & 100.0\ /\ 83.9 & 100.0\ /\ 83.8 & 83.3\ /\ 71.5 & 66.2\ /\ 48.4 & 100.0\ /\ 98.8 & 100.0\ /\ 82.2 & 100.0\ /\ 100.0 \\
\bottomrule
\end{tabular}
\end{table}

\subsection{XSMOM Robustness to Covariance Estimator: Sample Covariance Figures}\label{sec:appendix_real_sample}

Appendix~\ref{sec:rd_estimator} reports that the Pareto dominance of BDS-Direct on
FF49 XSMOM returns, established under the Ledoit-Wolf estimator in the main
text, replicates under the unshrunk sample covariance estimator at every
estimation window, with a larger volatility margin at the shorter window
but lower temporal persistence overall. Figures~\ref{fig:app_sample_pareto}
and~\ref{fig:app_sample_rolling} reproduce the corresponding Pareto scatter and
rolling-volatility figures for the sample-covariance specification, same
convention as Figures~\ref{fig:rd_pareto} and~\ref{fig:rd_rolling}.

\begin{figure}[h]
    \centering
    \includegraphics[width=\textwidth]{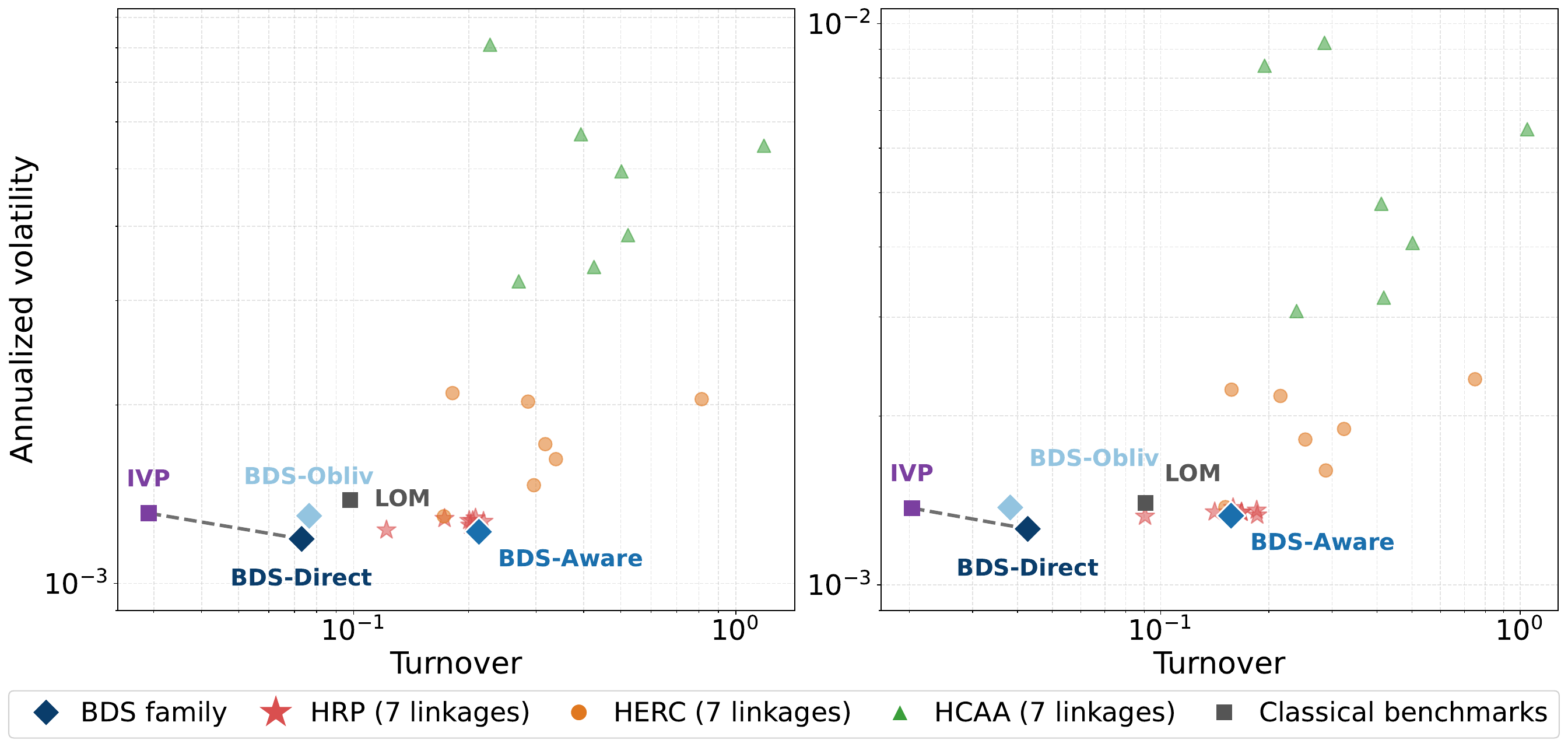}
    \caption{Realized annualized volatility vs.\ mean monthly turnover for each of
    26 portfolio methods on the Fama-French 49 industry XSMOM returns, sample
    covariance estimator, across two estimation windows $w \in \{2n,\,
    3n\}$ (monthly rebalancing, July~1970--February~2026). Same convention as
    Figure~\ref{fig:rd_pareto}. The frontier consists exclusively of BDS methods and
    IVP at every window; LOM and all 21 hierarchical variants are strictly
    dominated on both axes simultaneously. Absolute turnover levels are roughly
    twice those under Ledoit-Wolf (Figure~\ref{fig:rd_pareto}), reflecting the
    instability of the unshrunk estimator; the qualitative ranking is unchanged.}
    \label{fig:app_sample_pareto}
\end{figure}

\begin{figure}[h]
    \centering
    \includegraphics[width=\textwidth]{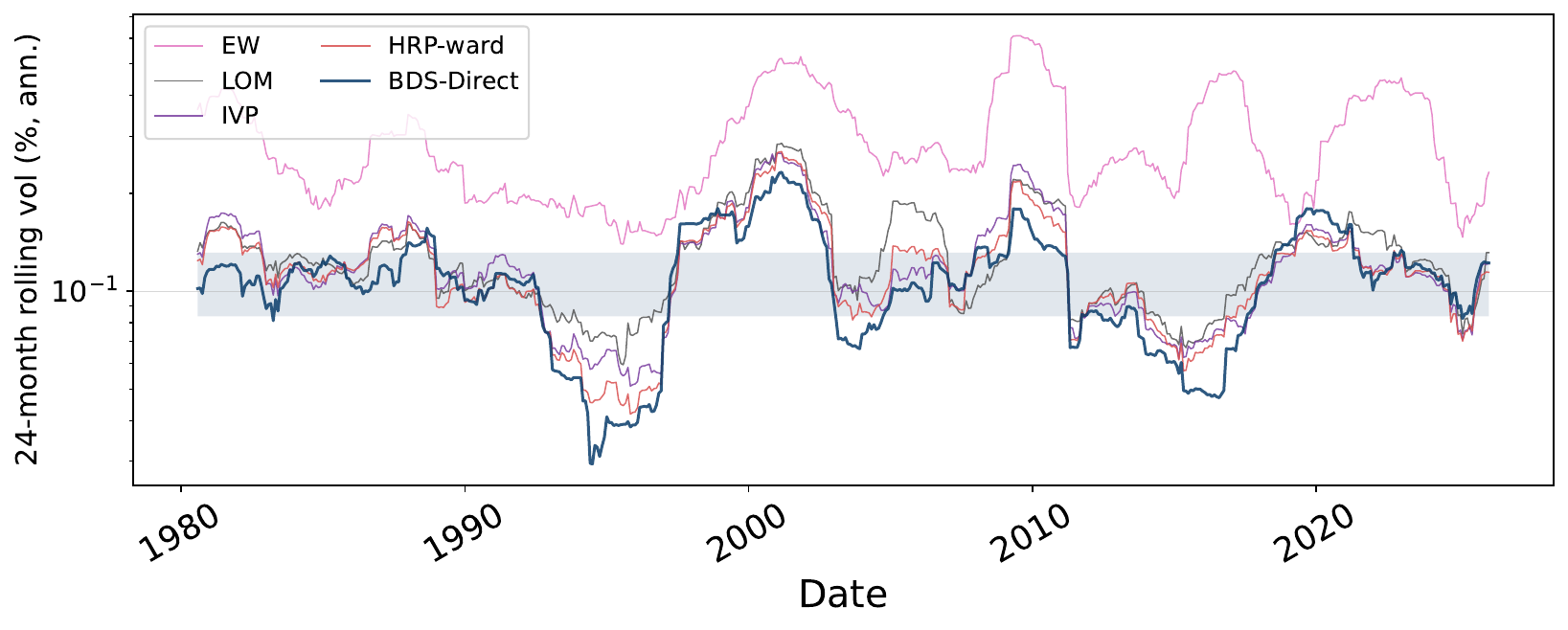}
    \caption{24-month rolling realized annualized volatility (\%) for nine
    portfolio methods on the Fama-French 49 industry XSMOM returns, sample
    covariance estimator, estimation window $w = 2n$ (monthly rebalancing).
    Same convention as Figure~\ref{fig:rd_rolling}. Best BDS method (BDS-Direct) 
    falls below LOM in $76.6\%$ months across the
    full 1980--2026 evaluation period, including through the dot-com
    correction, the global financial crisis, and the COVID-19 disruption, a
    somewhat less persistent advantage than under Ledoit-Wolf ($84.8\%$ Figure~\ref{fig:rd_rolling}), consistent with the noisier,
    unshrunk covariance estimate.}
    \label{fig:app_sample_rolling}
\end{figure}

\end{document}